\documentclass[11pt,a4paper]{article}
\usepackage[dvips]{graphicx}
\usepackage{amssymb,amsmath,amsthm}
\usepackage{enumerate}
\usepackage{ascmac}
\usepackage[dvips]{graphicx}
\usepackage{float}
\usepackage{wrapfig}
\usepackage{fancybox}
\usepackage[usenames]{color}

\newtheorem{df}{Definition}[section]
\newtheorem{thm}{Theorem}[section]

\newtheorem{lem}{Lemma}[section]
\newtheorem{pro}{Proposition}[section]
\newtheorem{rem}{Remark}[section]

\newcommand{\dis}{\displaystyle}
\newcommand{\R}{{\Bbb R}}
\newcommand{\N}{{\Bbb N}}
\newcommand{\pa}{\partial}
\newcommand{\loc}{\text{loc}}
\newcommand{\h}{\hspace{0.5mm}}
\newcommand{\hh}{\hspace{5mm}}

\def\cad#1{\csname #1\endcsname}
\title{Blow-up behavior of solutions to the heat equation with nonlinear boundary conditions}
\author{Junichi Harada \\[1mm]
{\small Department of Applied Physics, School of Science and Engineering,}\\
{\small Waseda University, 3-4-1 Okubo, Shinjuku-ku, Tokyo 169-8555, Japan}\\
{\small harada-j@aoni.waseda.jp}}
\date{}
\begin{document}
\maketitle

%%%%%%%%%%%%%%%%%%%%%%%%%%%%%%%%%%%%%%%%%%%%%%%%%%%%%%%%%%%
\begin{abstract}
We study the asymptotic behavior of blow-up solutions of
the heat equation with nonlinear boundary conditions.
In particular,
we classify the asymptotic behavior of blow-up solutions
and investigate the spacial singularity of their blow-up profiles.
\end{abstract}
%%%%%%%%%%%%%%%%%%%%%%%%%%%%%%%%%%%%%%%%%%%%%%%%%%%%%%%%%%%

%%%%%%%%%%%%%%%%%%%%%%%%%%%%%%%%%%%%%%%%%%%%%%%%%%%%%%%%%%%
\section{Introduction}
%%%%%%%%%%%%%%%%%%%%%%%%%%%%%%%%%%%%%%%%%%%%%%%%%%%%%%%%%%%

We study positive solutions of the heat equation
with nonlinear boundary conditions:
\[
\begin{cases}
\dis
 \pa_tu = \Delta u
& \text{in }\in\R_+^n\times(0,T),
\\ \dis
 \pa_{\nu}u=u^q
& \text{on }\in\pa\R_+^n\times(0,T),
\\ \dis
 u(x,0) = u_0(x)
& \text{in }\R_+^n,
\end{cases}
\eqno{({\rm P})}
\]
where $\R_+^n=\{x=(x',x_n)\in\R^{n-1}\times\R;x_n>0\}$, $\pa_{\nu}=-\pa/\pa x_n$,
$1<q<n/(n-2)$ if $n\geq3$ and
\[
 u_0\in C(\overline{\R_+^n})\cap L^\infty(\R_+^n),
\hspace{7.5mm}
 u_0\geq0.
\]
A solution $u(x,t)$ is said to blow up in a finite time,
if there exists $T>0$ such that
\[
 \limsup_{t\to T}\|u(t)\|_{L^{\infty}(\R_+^n)} = \infty.
\]
It is known that
a solution blows up in a finite time for any $q>1$,
if the initial data is positive and large enough.
In particular,
every positive solution blows up in a finite time
even if the initial data is small for the case $1<q\leq(n+1)/n$ (\cite{Deng-F-L}).
In this paper,
we study the asymptotic behavior of blow-up solutions of (P) near the blow-up time
and their blow-up profiles.
If a limit
\[
 U(x) = \lim_{t\to T}u(x,t) \in[0,\infty]
\]
exists for any $x\in\overline{\R_+^n}$,
we call $U(x)$ a blow-up profile of $u(x,t)$.
For a one dimensional case,
Fila-Quittner \cite{Fila-Q} (see also \cite{Deng-F-L}) constructed
a backward self-similar blow-up solution:
\[
 u_B(x,t)=(T-t)^{-1/2(q-1)}\varphi_0((T-t)^{-1/2}x),
\]
where $\varphi_0\in BC^2(\R_+)$ is a positive solution of
\begin{equation}\label{BS-eq}
\begin{cases} 
\dis
 \varphi_0'' - \frac{\xi}{2}\varphi_0' - \frac{\varphi_0}{2(q-1)} = 0
& \text{in } \xi\in\R_+,
\\ \dis
 \pa_{\nu}\varphi_0(0) = \varphi_0(0)^q.
\end{cases}
\end{equation}
This special blow-up solution $u_B(x,t)$ has the following blow-up profile:
\[
 U_B(x) = c_Bx^{-1/(q-1)}.
\]
As for general blow-up solutions for a multidimensional case with $1<q<n/(n-2)$ if $n\geq3$,
the author in \cite{Harada} proved that
if a positive $x_n$-axial symmetric solution blows up at the origin in a finite time,
then its blow-up profile $U(x)$ satisfies
\begin{equation}\label{Profile-eq}
 U(x) = c_B(1+o(1))(\cos\theta)^{-1/(q-1)}|x|^{-1/(q-1)}
\end{equation}
along $x_n=|x|\cos\theta$ for any fixed $\theta\in[0,\pi/2)$.
Unfortunately
the expression \eqref{Profile-eq} does not hold on $\pa\R_+^n$,
since $\cos\theta=0$ if $x\in\pa\R_+^n$.
The purpose of this paper is to derive a formula on $\pa\R_+^n$ corresponding to \eqref{Profile-eq}.

We recall known results concerning
the asymptotic behavior of blow-up solutions of the following semilinear heat equation
with $1<p<(n+2)/(n-2)$ if $n\geq3$:
\[
 \pa_tu = \Delta u +u^p
\hspace{5mm}\text{in }\R^n\times(0,T)\hspace{5mm}(u>0).
\eqno{({\rm F})}
\]
Let $u(x,t)$ be a solution of (F) which blows up at the origin
in a finite time $T>0$.
Giga-Kohn (\cite{Giga-K1,Giga-K2,Giga-K3}) derived the following asymptotic formula
for blow-up solutions:
\begin{equation}\label{2-eq}
 \lim_{t\to T}\sup_{|x|<\nu(T-t)^{1/2}}
 \left| (T-t)^{1/(p-1)}u(x,t)-(p-1)^{-1/(p-1)} \right|=0
\hspace{5mm}\text{for }\nu>0.
\end{equation}
This description gives the first approximation for blow-up solutions.
In their paper,
they introduced self-similar variables: $x=(T-t)^{1/2}y$, $s=-\log(T-t)$
and a rescaled function defined by
\[
 \varphi(y,s) = (T-t)^{1/(p-1)}u(x,t).
\]
The asymptotic formula \eqref{2-eq} is equivalent to
\[
 \lim_{s\to\infty}\varphi(y,s) =(p-1)^{-1/(p-1)}
\hspace{5mm}\text{in } C_{\loc}(\R^n).
\]
Filippas-Kohn \cite{Filippas-K} and Herrero-Vel\'azquez \cite{Herrero-V1}
independently studied the second approximation
for blow-up solutions for a one dimensional case (see also \cite{Bebernes-B,Filippas-L,Liu}).
Let
$H_k(y)$ and $\lambda_k$ ($k=0,1,\cdots$) be the $k$-th eigenfunction and eigenvalue of
\[
 -\left( H''-\frac{y}{2}H'+H \right)= \lambda H
\]
in $L_\rho^2(\R)=\{H\in L^2_{\loc}(\R);\int_{-\infty}^\infty H(y)^2e^{-y^2/4}dy<\infty\}$.
It is well known that
$H_k(y)$ is a $k$-th polynomial and $\lambda_k=(k-2)/2$.
Then
one of the following cases occurs (\cite{Filippas-K,Herrero-V2,Herrero-V1}):
\\[2mm]
\begin{tabular}{cl}
(i) &
there exists $\kappa_p>0$ such that\\[1mm]
& \hspace{30mm}
$\varphi(y,s) = (p-1)^{-1/(p-1)}-\kappa_pH_2(y)s^{-1} +o(s^{-1})$
\hspace{3mm}in $L_\rho^2(\R)$,\\[2mm]
(ii) &
there exist even integer $m\geq4$ and $\kappa\not=0$ such that\\[1mm]
& \hspace{29mm}
$\varphi(y,s) = (p-1)^{-1/(p-1)}+\kappa H_m(y)e^{-\lambda_ms}+o(e^{-\lambda_ms})$
\hspace{3mm}in $L_\rho^2(\R)$.
\end{tabular}
\\[2mm]
In particular,
the case (i) actually occurs
if the initial data is even and monotone decreasing on $(0,\infty)$.
As a further step of this second approximation formula,
the blow-up profile for solutions of (F) was derived by Herrero-Vel\'azquez 
(\cite{Herrero-V2,Herrero-V3,Herrero-V4,Herrero-V1,Velazquez1,Velazquez2,Velazquez3}).
For a one dimensional case,
Herrero-Vel\'azquez \cite{Herrero-V3} proved that
if (i) occurs,
then there exists $\kappa_p'>0$ such that the blow-up profile $U(x)$ satisfies
\[
 \lim_{x\to0}\left( \frac{x^2}{|\log x|} \right)^{1/(p-1)}U(x)=\kappa_p',
\]
on the other hand, if (ii) occurs,
then there exists $\kappa'>0$ such that the blow-up profile $U(x)$ satisfies
\[
 \lim_{x\to0}x^{m/(p-1)}U(x)=\kappa'.
\]

In this paper,
we investigate the blow-up profile for solutions of (P).
Following their arguments,
we introduce a rescaled solution of (P):
\[
  \varphi(y,s) = (T-t)^{1/2(q-1)}u((T-t)^{1/2}y,t),
\hspace{5mm}
 s=-\log(T-t).
\]
This rescaled solution $\varphi(y,s)$ solves
($s_T=-\log T$)
\[
\begin{cases}
\dis
 \pa_s\varphi = \Delta\varphi - \frac{y}{2}\cdot\nabla\varphi - \frac{\varphi}{2(q-1)}
& \text{in }\R_+^n\times(s_T,\infty),
\\ \dis
 \pa_{\nu}\varphi = \varphi^q
& \text{on }\in\pa\R_+^n\times(s_T,\infty).
\end{cases}
\]
Then
it is shown in \cite{Chlebik-F2} that
\[
 \lim_{s\to\infty}\varphi(y,s) = \varphi_0(y_n)
\hspace{5mm}\text{in } C_{\loc}(\overline{\R_+^n}),
\]
where $\varphi_0(y_n)$ is a unique bounded positive solution of \eqref{BS-eq}
(see also \cite{Chlebik-F,Hu-Y}).
This formula is equivalent to
\[
 u(x,t) = (T-t)^{1/2(q-1)}\left( \varphi_0((T-t)^{-1/2}x_n)+o(1) \right),
\hspace{5mm}|x|<\nu(T-t)^{1/2}
\]
for any $\nu>0$.
By virtue of the asymptotic formula of $\varphi_0(\xi)$:
$\varphi_0(\xi)\sim c_B\xi^{1/(q-1)}$ as $\xi\to\infty$,
we formally obtain
\[
 U(x) = c_B(1+o(1))x_n^{-1/(q-1)}.
\]
This formal argument was justified in \cite{Harada}.
As is stated above,
this formula has no meaning on $\pa\R_+^n$,
since the right-hand side diverges on $\pa\R_+^n$.
To obtain the blow-up profile on $\pa\R_+^n$,
we need to derive more precise large time behavior of a rescaled solution $\varphi(y,s)$,
that is the second approximation formula.
To do that
we introduce a new function
\[
 v(y,s)=\varphi(y,s)-\varphi_0(y_n)
\]
and study the large time behavior of $v(y,s)$.
Then $v(y,s)$ solves
\begin{equation}\label{PQ-eq}
\begin{cases}
\dis
 \pa_sv = \Delta v - \frac{y}{2}\cdot \nabla v -\frac{v}{2(q-1)}
& \text{in }\R_+^n\times(0,\infty),
\\ \dis
 \pa_{\nu}v = q\varphi_0^{q-1}v +O(v^2)
& \text{on }\pa\R_+^n\times(0,\infty).
\end{cases}
\end{equation}
A corresponding eigenvalue problem is given by
\begin{equation}\label{operator-eq}
\begin{cases}
\dis
 -\left( \Delta E-\frac{y}{2}\cdot\nabla E-\frac{E}{2(q-1)} \right) = \lambda E
& \text{in } \R_+^n,\\ \dis
 \pa_\nu E = q\varphi_0^{q-1}E
& \text{on } \pa\R_+^n.
\end{cases}
\end{equation}
Let
$E_i(y)\in L_\rho^2(\R_+^n)$ be the $i$-th eigenfunction of \eqref{operator-eq},
where $L_\rho^2(\R_+^n)$ is a weighted $L^2$-space defined by
\[
 L_\rho^2(\R_+^n) =
 \left\{
 v\in L_{\loc}^2(\R_+^n);\ \|v\|_{L_\rho^2(\R_+^n)}^2:=\int_{\R_+^n}|v(y)|^2e^{-|y|^2/4}dy<\infty
 \right\}.
\]
The first result in this paper
is a classification of the large time behavior of $v(y,s)$.

%%%%%%%%%%%%%%%%%%%%%%%%%%%%%%%%%%%%%%%%%%%%%%%%%%%%%%%%%%%
\begin{thm}\label{1-thm}
Let $u(x,t)$ be a positive $x_n$-axial symmetric solution of $({\rm P})$
which blows up at the origin and $v(y,s)$ be given above.
Then one of two cases occurs.
\\[2mm]
\begin{tabular}{cl}
{\rm(I)} &
there exists $\nu_q>0$ such that
\hspace{1.5mm}
% \\[1mm] % & \hspace{50mm}
$\dis \|v(s)-\nu_qs^{-1}{\cal E}\|_{L_\rho^2(\R_+^n)}=o(s^{-1})$,
\\[2mm]
{\rm(II)} &
there exist $c>0$ and $\gamma>0$ such that
\hspace{1.5mm}
% \\[1mm] & \hspace{50mm}
$\dis\|v(s)\|_{L_\rho^2(\R_+^n)} \leq ce^{-\gamma s}$,
\end{tabular}
\\[2mm]
where ${\cal E}(y)$ is the eigenfunction of \eqref{operator-eq} with zero eigenvalue,
which is defined in \eqref{calE-eq}.
\end{thm}
%%%%%%%%%%%%%%%%%%%%%%%%%%%%%%%%%%%%%%%%%%%%%%%%%%%%%%%%%%%

The second result gives a partial answer to our motivation.

%%%%%%%%%%%%%%%%%%%%%%%%%%%%%%%%%%%%%%%%%%%%%%%%%%%%%%%%%%%
\begin{thm}\label{2-thm}
Let $u(x,t)$ and $v(y,s)$ be as in Theorem {\rm\ref{1-thm}}.
Additionally we assume that
\[
 x'\cdot \nabla'u_0 \leq 0.
\]
If the case $({\rm I})$ in Theorem {\rm\ref{1-thm}} occurs,
then the blow-up profile $U(x)\in C(\overline{\R_+^n}\setminus\{0\})$ exists
and there exist positive constants $c_1<c_2$ such that
\[
 c_1\left( \frac{|\log|x'||}{|x'|^2} \right)^{1/2(q-1)}
\leq U(|x'|,0) \leq
 c_2\left( \frac{|\log|x'||}{|x'|^2} \right)^{1/2(q-1)}
\hspace{3mm}\mathrm{for}\ |x'|<1.
\]
\end{thm}
%%%%%%%%%%%%%%%%%%%%%%%%%%%%%%%%%%%%%%%%%%%%%%%%%%%%%%%%%%%

%%%%%%%%%%%%%%%%%%%%%%%%%%%%%%%%%%%%%%%%%%%%%%%%%%%%%%%%%%%
\begin{rem}
The author in {\rm\cite{Harada}} proved that
if the initial data $u_0$ is $x_n$-axial symmetric and
satisfies
\[
 x'\cdot \nabla'u_0 \leq 0\hspace{2mm}(\nabla' u_0\not\equiv0),\hspace{7.5mm}\pa_nu_0\leq0,
\eqno{\rm{(a)}}
\]
then $v(y,s)$ actually behaves as the case $({\rm I})$ in Theorem {\rm\ref{1-thm}}.
As for the case {\rm(F)},
if the initial data is radially symmetric and monotone decreasing,
then those properties are preserved for $t>0$.
Therefore
the solutions has a unique local maximum point at the origin for $t>0$
and no local minimum points for $t>0$.
From the view point of this geometry of the solution,
it is easily proved that
the rescaled solution $\varphi(y,s)$ satisfies the asymptotic formula {\rm(i)}.
However
this kind of observation can not be applicable to solutions of {\rm(P)},
since solutions treated here are not radially symmetric.
\end{rem}
%%%%%%%%%%%%%%%%%%%%%%%%%%%%%%%%%%%%%%%%%%%%%%%%%%%%%%%%%%%

We explain our strategy of the proof for Theorem \ref{1-thm} and Theorem \ref{2-thm}.
Our argument mainly consists of three steps.
In the first part,
we regard \eqref{PQ-eq} as a dynamical system on the Hilbert space $L_\rho^2(\R_+^n)$
and study the large time behavior of $v(y,s)$ in $L_\rho^2(\R_+^n)$.
To do that we expand $v(y,s)$ by using eigenfunctions of \eqref{operator-eq} as follows:
\[
 v(s) = \sum_{i=1}^{\infty} a_i(s)E_i
\hspace{5mm}\text{in } L_\rho^2(\R_+^n).
\]
Here we determine the large time behavior of the coefficients $a_i(s)$.
Each coefficient $a_i(s)$ satisfies some ordinary differential equations.
We will see that
these ordinary differential equations are finally reduced to
well understood ordinary differential inequalities discussed in \cite{Filippas-K},
and this proves Theorem \ref{1-thm}.

In the second step,
we assume that $v(y,s)$ behaves as the case (I) in Theorem \ref{1-thm}
and investigate the large time behavior of $v(y,s)$ along $|y|\sim s^{1/2}$ on $\pa\R_+^n$.
This step is crucial and much harder than the first step.
The first step provides the precise decay rate of $v(y,s)$ on $L_\rho^2(\R_+^n)$,
however the convergence is taken in a local sense.
Hence
in the original valuable $(x,t)$, the following asymptotic formula:
\[
\begin{array}{lll}
 u(x,t)
\hspace{-2mm}&=&\hspace{-2mm}
 (T-t)^{-1/2(q-1)}\varphi((T-t)^{-1/2}x,-\log(T-t))
\\[2mm]
\hspace{-2mm}&=&\hspace{-2mm}
 (T-t)^{-1/2(q-1)}\left( \varphi_0((T-t)^{-1/2}x_n)+o(1) \right)
\end{array}
\]
holds only for a small region $|x|\leq\nu(T-t)^{1/2}$ for any $\nu>0$.
Arguments given in \cite{Herrero-V3}, \cite{Herrero-V1}
revealed that
the convergence in this small region is not sufficient to extract information
of the blow-up profile of solutions of (F).
In particular for a one dimensional case,
they proved that if an initial data is nonnegative, symmetric and nonincreasing,
then the solution $u(x,t)$ satisfies
\[
 \lim_{t\to T}(T-t)^{1/(p-1)}u(z(T-t)^{1/2}|\log(T-t)|^{1/2},t) =
 \left( (p-1)+K_pz^2 \right)^{-1/(p-1)}
\]
in $C_{\loc}(\R)$.
Unlike this case,
we do not know whether this limit exists for solutions of (P).
However
we will see that
there exist $\theta>0$ and $0<A_-<A_+$ such that
\[
 A_- \leq (T-t)^{1/2(q-1)}u(z(T-t)^{1/2}|\log(T-t)|^{1/2},t)|_{\pa\R_+^n} \leq A_+
\hspace{5mm}\text{for }|z|=\theta.
\]

Finally in the third step,
we determine the singularity of the blow-up profile.
To do that
we introduce another rescaled function:
\[
 v_s(x,t) = e^{-ms}u(e^{-s/2}x+\theta\sqrt{s}e^{-s/2}\vec{e},T+(t-1)e^{-s}),
\]
where $s\gg1$ is a parameter and $\vec{e}=(1,0.\cdots,0)$.
We establish a uniform pointwise estimate of $v_s(x,t)$ with respect to $s\gg1$ and $t\in(0,1)$.
As a consequence of this uniform pointwise estimate,
we can derive Theorem \ref{2-thm}.

%%%%%%%%%%%%%%%%%%%%%%%%%%%%%%%%%%%%%%%%%%%%%%%%%%%%%%%%%%%
The rest of this paper is organized as follows.
In Section \ref{Preliminary-sec},
we provide eigenfunctions and eigenvalues of \eqref{operator-eq} 
and establish the global heat kernel estimates of the linearized equation.
Moreover
we give a representation formula of solutions of the linearized equation.
Section \ref{DynamicalSystemApproach-sec} is devoted to study
the large time behavior of $v(y,s)$ in $L_\rho^2(\R_+^n)$.  
In particular,
an exact decay rate of $v(y,s)$ in $L_\rho^2(\R_+^n)$ is derived.
In Section \ref{LongRange-sec},
we obtain a refined asymptotic formula of $v(y,s)$.
Then
the large time behavior of $v(y,s)$ along $|y|\sim s^{1/2}$ on $\pa\R_+^n$ is discussed,
which is a crucial part in this paper.
Finally in Section \ref{Spacial-sec},
we study the blow-up profile by applying the arguments given in \cite{Herrero-V3,Merle-Z}.
In Appendix,
some inequalities and some properties of the linearized operator are discussed.

%%%%%%%%%%%%%%%%%%%%%%%%%%%%%%%%%%%%%%%%%%%%%%%%%%%%%%%%%%%
\section{Preliminary}
\label{Preliminary-sec}
%%%%%%%%%%%%%%%%%%%%%%%%%%%%%%%%%%%%%%%%%%%%%%%%%%%%%%%%%%%

Throughout this paper,
for simplicity of notations,
we put $\N_0=\N\cup\{0\}$ and
\[
 m=1/2(q-1).
\]
Let $u(x,t)$ be a solution of (P) which blows up at the origin in a finite time  $T>0$.
To study the blow-up behavior of $u(x,t)$,
we put $s_T=-\log T$ and
\[
 \varphi(y,s) = e^{-ms}u(e^{-s/2}y,T-e^{-s}).
\]
Then $\varphi(y,s)$ satisfies
\begin{equation}\label{(B)-eq}
\begin{cases}
\dis
 \pa_s\varphi = \Delta\varphi - \frac{y}{2}\cdot\nabla\varphi - m\varphi
& \text{in }\R_+^n\times(s_T,\infty),
\\ \dis
 \pa_{\nu}\varphi = \varphi^q
& \text{on }\pa\R_+^n\times(s_T,\infty),
\\ \dis
 \varphi(y,s_T)=T^mu_0(T^{1/2}y)
& \text{in }\R_+^n.
\end{cases}
\end{equation}
It is known that
if $1<q<n/(n-2)$,
\eqref{(B)-eq} admits the unique bounded positive stationary solution $\varphi_0(y_n)$
depending only on $y_n$-variable,
that is a positive solution of
\[
\begin{cases}
\dis
 \varphi_0'' - \frac{\xi}{2}\varphi_0'-m\varphi_0 = 0
& \text{in }\xi\in\R_+,
\\ \dis
 \pa_{\nu}\varphi_0(0) = \varphi_0(0)^q.
\end{cases}
\]
The existence and the uniqueness of this equation are shown
in Lemma 3.1 of \cite{Fila-Q} and in Theorem 3.1 of \cite{Chlebik-F2}, respectively.
For the rest of this paper,
we put
\[
 B = \varphi_0(0).
\]
Here
we recall the fact concerning the asymptotic behavior of $\varphi(y,s)$.

%%%%%%%%%%%%%%%%%%%%%%%%%%%%%%%%%%%%%%%%%%%%%%%%%%%%%%%%%%%
\begin{thm}[\cite{Chlebik-F}, Theorem 3.2 in \cite{Chlebik-F2}]
\label{Ua-thm}
Let $\varphi(y,s)$ be defined above.
Then
$\varphi(y,s)$ is uniformly bounded on $\R_+^n\times(s_T,\infty)$
and
$\varphi(y,s)$ converges to $\varphi_0(y_n)$ as $s\to\infty$
uniformly on any compact set in $\overline{\R_+^n}$.
\end{thm}
%%%%%%%%%%%%%%%%%%%%%%%%%%%%%%%%%%%%%%%%%%%%%%%%%%%%%%%%%%%

A boundedness of $\varphi(y,s)$ is the first step to study the large time behavior of $\varphi(y,s)$.
Once the boundedness is assured,
from the energy identity,
we immediately obtain $\varphi_s(y,s)\to0$ as $s\to\infty$.
Furthermore
as a consequence of a boundedness of $\varphi(y,s)$,
a boundedness of spacial and time derivatives of $\varphi(y,s)$ are also derived.

%%%%%%%%%%%%%%%%%%%%%%%%%%%%%%%%%%%%%%%%%%%%%%%%%%%%%%%%%%%
\begin{lem}\label{21-lem}
Let $\varphi(y,s)$ be as in Theorem {\rm\ref{Ua-thm}}.
Then
for any $\delta>0$
there exists $c_\delta>0$ such that
\[
 \sum_{|\alpha|\leq2}
 |D^{\alpha}\varphi(y,s)| + (1+|y|)^{-1/2}|\pa_s\varphi(y,s)| \leq c_\delta
\hh\mathrm{for}\ (y,s)\in\R_+^n\times(s_T+\delta,\infty),
\]
where $D$ represents the spacial derivatives.
\end{lem}
%%%%%%%%%%%%%%%%%%%%%%%%%%%%%%%%%%%%%%%%%%%%%%%%%%%%%%%%%%%

%%%%%%%%%%%%%%%%%%%%%%%%%%%%%%%%%%%%%%%%%%%%%%%%%%%%%%%%%%%
\begin{proof}
The proof follows from the argument in Proposition 1$'$ of \cite{Giga-K1}.
Since their proof relies only on the scaling argument,
the argument can be applicable to (P).
\end{proof}
To derive the second approximation for $\varphi(y,s)$ as $s\to\infty$,
we set
\[
 v(y,s)=\varphi(y,s)-\varphi_0(y_n).
\]
Then
$v(y,s)$ satisfies
\begin{equation}\label{(B2)-eq}
\begin{cases}
\dis
 \pa_sv = \Delta v - \frac{y}{2}\cdot\nabla v - mv
& \text{in }\R_+^n\times(s_T,\infty),
\\[2mm] \dis
 \pa_{\nu}v = qB^{q-1}v + f(v)
& \text{on }\pa\R_+^n\times(s_T,\infty),
\\ \dis
 v(y,s_T)=\varphi(y,s_T)-\varphi_0(y_n)
& \text{in }\R_+^n,
\end{cases}
\end{equation}
where $f(v)$ is given by
\[
 f(v)=(v+B)^q-B^q-qB^{q-1}v.
\]
Then
since $\varphi(y,s)$ is positive and uniformly bounded on $\R_+^n\times(s_T,\infty)$,
we easily see that
\begin{equation}\label{f(v)regularity-eq}
 f(v) = \frac{q(q-1)B^{q-2}}{2}v^2+O(v^3),
\hspace{7.5mm}
 |f'(v)| \leq c'|v|
\end{equation}
for $(y,s)\in\R_+^n\times(s_T,\infty)$.
We define the linear operator $A$ associated with \eqref{(B2)-eq} and its domain by
\[
\begin{array}{c}
\dis
 Av = \left(\Delta-\frac{y}{2}\cdot\nabla-m \right)v,\\[3mm]
D(A) = \{v\in H_{\rho}^2(\R_+^n);
 \pa_{\nu}v=qB^{q-1}v\text{ on } \pa\R_+^n\},
\end{array}
\]
where $H_{\rho}^2(\R_+^n)$ is a weighted Sobolev space defined in
Section \ref{LinearOperator-sec}.

%%%%%%%%%%%%%%%%%%%%%%%%%%%%%%%%%%%%%%%%%%%%%%%%%%%%%%%%%%%
\subsection{Linear operator $A$}
\label{LinearOperator-sec}
%%%%%%%%%%%%%%%%%%%%%%%%%%%%%%%%%%%%%%%%%%%%%%%%%%%%%%%%%%%

Here we consider the following eigenvalue problems:
\begin{equation}\label{eigene-eq}
\begin{cases}
\dis
 -\left( \Delta - \frac{y}{2}\cdot\nabla - m \right)E = \lambda E
& \text{in } \R_+^n,
\\[2mm] \dis
 \pa_{\nu}E = qB^{q-1}E
& \text{on } \pa\R_+^n.
\end{cases}
\end{equation}
We define a wight function:
\[
 \rho(y) = e^{-|y|^2/4}.
\]
Then
it is clear that $\rho(y')=-e^{-|y'|^2/4}$ on $\pa\R_+^n$.
Moreover we define functional spaces:
\[
\begin{array}{c}
\dis
 L_{\rho}^p(\R_+^n) =
 \left\{
 v\in L_{\loc}^p(\R_+^n);\int_{\R_+^n}|v(y)|^p\rho(y)dy<\infty
 \right\},
\\[6mm] \dis
 H_{\rho}^k(\R_+^n) =
 \left\{
 v\in L_{\rho}^2(\R_+^n); D^{\alpha}v\in L_{\rho}^2(\R_+^n)
 \text{ for any } \alpha=(\alpha_1,\cdots,\alpha_n)
 \text{ satisfying } |\alpha|\leq k
 \right\},
\\[2mm] \dis
L_{\rho}^p(\pa\R_+^n) =
 \left\{
 v\in L_{\loc}^p(\pa\R_+^n);\int_{\pa\R_+^n}|v(y')|^p\rho(y')dy'<\infty
 \right\}.
\end{array}
\]
The norms are given by
\[
\begin{array}{c}
\dis
 \|v\|_{L_{\rho}^p(\R_+^n)}^p = \int_{\R_+^n} |v(y)|^p\rho(y)dy,
\hspace{7.5mm}
 \|v\|_{H_{\rho}^k(\R_+^n)}^2 =
 \sum_{|\alpha|\leq k}\|D^{\alpha}v\|_{L_{\rho}^2(\R_+^n)}^2,
\\ \dis
 \|v\|_{L_{\rho}^p(\pa\R_+^n)}^p = \int_{\pa\R_+^n} |v(y')|^p\rho(y')dy'
\end{array}
\]
and the inner product on $L_{\rho}^2(\R_+^n)$ is naturally defined by
\[
 (v_1,v_2)_{\rho} = \int_{\R_+^n}v_1(y)v_2(y)\rho(y)dy.
\]
For simplicity,
the norm of $L_{\rho}^2(\R_+^n)$ is denoted by
$\|\cdot\|_{\rho}=\|\cdot\|_{L_{\rho}^2(\R_+^n)}$.
Since
the operator $A$: $D(A)\to L_{\rho}^2(\R_+^n)$
is self-adjoint and has a compact inverse from $L_{\rho}^2(\R_+^n)$ to $L_{\rho}^2(\R_+^n)$
(see Appendix),
$L_{\rho}^2(\R_+^n)$ is spanned by the eigenfunctions of \eqref{eigene-eq}.
Let $\tilde{H}_k(\xi)$ be the $k$-th Hermite polynomial defined by
\[
 \tilde{H}_k(\xi) = (-1)^ke^{\xi^2}\frac{d^k}{d\xi^k}\left( e^{-\xi^2} \right)
\hspace{5mm}(k\in\N_0)
\]
and set $H_k(\xi)=c_k\tilde{H}_k(\xi/2)$,
where $c_k$ is a normalization constant such that
$\int_{-\infty}^{\infty}H_k(\xi)^2e^{-\xi^2/4}d\xi=1$.
From classical results,
it is known that $H_k(\xi)$ satisfies
\[
 -\left( H_k''-\frac{\xi}{2}H_k' \right) = \frac{k}{2}H_k
\hspace{5mm}\text{in } \R.
\]
Moreover let $I_k(\xi)$ and $\kappa_k$ ($k\in\N$) be the $k$-th eigenfunction with
$\int_0^{\infty}I_k(\xi)^2e^{-\xi^2/4}d\xi=1$ and the $k$-th eigenvalue of
\begin{equation}\label{Ieigen-eq}
\begin{cases}
\dis
 -\left( I'' - \frac{\xi}{2}I' \right) = \kappa I
& \text{in } \R_+,
\\[3mm] \dis
 \pa_{\nu}I(0) = qB^{q-1}I(0)
&
\end{cases}
\end{equation}
and $\alpha$ be a multi-index:
\[
 \alpha = (\alpha_1,\cdots,\alpha_n)
 \in {\cal A} := \N_0^{n-1}\times\N.
\]
Then
the eigenfunction $E_\alpha(y)$ of \eqref{eigene-eq}
and
its eigenvalue $\lambda_{\alpha}$
are given by
\[
 E_\alpha(y) =
 H_{\alpha_1}(y_1)\cdots H_{\alpha_{n-1}}(y_{n-1})I_{\alpha_n}(y_n),
\hspace{5mm}
 \lambda_{\alpha} =
 \sum_{i=1}^{n-1} \frac{\alpha_i}{2} + \kappa_{\alpha_n} + m.
\]
Here we recall a classical result about some special functions
(see Lemma 3.1 in \cite{Fila-Q}).

%%%%%%%%%%%%%%%%%%%%%%%%%%%%%%%%%%%%%%%%%%%%%%%%%%%%%%%%%%%
\begin{lem}\label{Keep-lem}
Let $b(\xi)\in L_\rho^2(\R_+)$ be a positive solution of
\[
 -\left( b'' - \frac{\xi}{2}b'\right) = \mu b
\hspace{5mm} \mathrm{in}\ \R_+
\hspace{5mm} (\mu<0)
\]
and set
\[
 U(a,a',r)
=
 \frac{1}{\Gamma(a)}\int_0^{\infty}e^{-rt}t^{a-1}(1+t)^{a'-a-1}dt
\hspace{5mm}(a,r>0),
\]
where $\Gamma$ is the Gamma function.
Then
there exists $c_1>0$ such that $b(\xi)$ is expressed by
\[
 b(\xi)
=
 c_1U(-\mu,1/2,\xi^2/4).
\]
\end{lem}
%%%%%%%%%%%%%%%%%%%%%%%%%%%%%%%%%%%%%%%%%%%%%%%%%%%%%%%%%%%

By using this formula,
we compute the first and the second eigenvalue of \eqref{Ieigen-eq}.

%%%%%%%%%%%%%%%%%%%%%%%%%%%%%%%%%%%%%%%%%%%%%%%%%%%%%%%%%%%
\begin{lem}\label{22-lem}
Let $\kappa_i$ be the $i${\rm-}th eigenvalue of \eqref{Ieigen-eq}.
Then
it holds that
\[
 \kappa_1 = -(m+1),\hspace{7.5mm} \kappa_2>0.
\]
\end{lem}
%%%%%%%%%%%%%%%%%%%%%%%%%%%%%%%%%%%%%%%%%%%%%%%%%%%%%%%%%%%%%%

%%%%%%%%%%%%%%%%%%%        Proof        %%%%%%%%%%%%%%%%%%%
\begin{proof}
Since $\kappa_1$ is the first eigenvalue of \eqref{Ieigen-eq},
it is characterized by
\[
 \kappa_1 =
 \inf_{I\in H_{\rho}^1(\R_+)}
 \frac{\|\pa_{\xi}I\|_{L_{\rho}^2(\R_+)}^2-qB^{q-1}I(0)^2}
 {\|I\|_{L_{\rho}^2(\R_+)}^2}.
\]
This implies $\kappa_1<0$.
Therefore by Lemma \ref{Keep-lem},
there exists $c_1>0$ such that
the first eigenfunction $I_1(\xi)$ is given by
\[
 I_1(\xi) =
 c_1U\left( -\kappa_1,\frac{1}{2},\frac{\xi^2}{4} \right).
\]
From (3.8) in \cite{Fila-Q},
we note that for $a>0$
\begin{equation}\label{U(a,b,c)-eq}
\begin{array}{l}
 \dis
 U(a,1/2,0) = \frac{\sqrt{\pi}}{\Gamma(a+1/2)},
\hspace{7.5mm}
 \lim_{\xi\to0}\pa_\xi U(a,1/2,\xi^2/4) = -\frac{a\sqrt{\pi}}{\Gamma(a+1)}.
\end{array}
\end{equation}
Therefore we obtain
\[
 \pa_{\nu}I_1(0)/I_1(0)
=
 (-\kappa_1)
 \Gamma\left( -\kappa_1+\frac{1}{2} \right)/\Gamma(-\kappa_1+1).
\]
On the other hand,
we recall from the proof of Lemma 3.1 in \cite{Fila-Q} that
$\varphi_0(\xi)$ is written by
\[
 \varphi_0(\xi)
=
 \frac{m^{1/(q-1)}}{\sqrt{\pi}}\left(\frac{\Gamma(m+1/2)^{q}}{\Gamma(m+1)} \right)^{1/(q-1)}
 U(m,1/2,\xi^2/4).
\]
Then
by using \eqref{U(a,b,c)-eq},
we get
\[
 qB^{q-1} = q\varphi_0(0)^{q-1} =
 qm\Gamma\left(m+\frac{1}{2}\right)/\Gamma(m+1).
\]
Therefore
since $\pa_{\nu}I_1(0)=qB^{q-1}I_1(0)$,
$\kappa_1$ is determined by
\begin{equation}
\label{kappa_1-eq}
 qm\Gamma\left(m+\frac{1}{2}\right)/\Gamma(m+1)
=
 (-\kappa_1)
 \Gamma\left( -\kappa_1+\frac{1}{2} \right)/\Gamma(-\kappa_1+1).
\end{equation}
First
we claim that \eqref{kappa_1-eq} admits at most one root.
We set
\[
 g(\lambda) =
 \lambda\Gamma\left( \lambda+\frac{1}{2} \right)/\Gamma(\lambda+1)
\hspace{5mm}(\lambda>0).
\]
By $\Gamma(\lambda+1)=\lambda\Gamma(\lambda)$,
it follows that
\begin{equation}
\label{g-eq}
 g(\lambda) =
 \Gamma\left( \lambda+\frac{1}{2} \right)/\Gamma(\lambda).
\end{equation}
Hence
from the property of the Gamma function (see p.\hspace{1mm}4 in \cite{Iwanami}),
$g(\lambda)$ is strictly increasing function for $\lambda>0$,
which shows the claim.
To assure $\kappa_1=-(m+1)$,
we substitute $\lambda=m+1$ in \eqref{g-eq}:
\begin{eqnarray*}
 g(m+1)
\hspace{-2mm}&=&\hspace{-2mm}
 \Gamma\left( m+\frac{3}{2} \right)/\Gamma(m+1) =
 \left( m+\frac{1}{2} \right)\Gamma\left( m+\frac{1}{2} \right)/\Gamma(m+1)
\\
\hspace{-2mm}&=&\hspace{-2mm}
 qm\Gamma\left( m+\frac{1}{2} \right)/\Gamma(m+1).
\end{eqnarray*}
Therefore
we conclude that $\kappa_1=-(m+1)$.
Next we show that $\kappa_2>0$.
We suppose $\kappa_2\leq0$.
For the case $\kappa_2<0$,
by the same way as above,
the second eigenfunction $I_2(\xi)$ is given by
\[
 I_2(\xi) = c_1'U\left( -\kappa_2,\frac{1}{2},\frac{\xi^2}{4} \right)
\]
for some $c_1'\not=0$.
Hence
it holds that
$I_2(\xi)>0$ on $\R_+$ or $I_2(\xi)<0$ on $\R_+$.
However this contradicts $\int_0^{\infty}I_1(\xi)I_2(\xi)e^{-\xi^2/4}d\xi=0$.
For the case $\kappa_2=0$,
we easily see that
\[
 \left( e^{-\xi^2/4}I_2'(\xi) \right)' = 0.
\]
Integrating over $(0,\xi)$,
we obtain
$I_2'(\xi)=-qB^{q-1}I_2(0)e^{\xi^2/4}$
and
\[
 I_2(\xi) = I_2(0) - qB^{q-1}I_2(0)\int_0^{\xi}e^{t^2/4}dt.
\]
Hence it follows that $I_2\not\in L_{\rho}^2(\R)$,
which is a contradiction.
Thus the proof is completed.
\end{proof}
%%%%%%%%%%%%%%%%%%%%%%%%%%%%%%%%%%%%%%%%%%%%%%%%%%%%%%%%%%%%%%

From the above facts,
the operator $(-A)$ has two negative eigenvalues $-1$ and $-1/2$ and zero eigenvalue.

%%%%%%%%%%%%%%%%%%%%%%%%%%%%%%%%%%%%%%%%%%%%%%%%%%%%%%%%%%%%%%
\subsection{Linear backward heat equation with Robin type boundary conditions}
\label{Linearbackward-sec}
%%%%%%%%%%%%%%%%%%%%%%%%%%%%%%%%%%%%%%%%%%%%%%%%%%%%%%%%%%%%%%

Here we study the following linear parabolic equation:
\begin{equation*}
\begin{cases}
\dis
 \pa_sv = \Delta v - \frac{y}{2}\cdot\nabla v
& \text{in }\R_+^n\times(0,\infty),
\\[2mm] \dis
 \pa_{\nu}v = Kv
& \text{on }\pa\R_+^n\times(0,\infty),
\end{cases}
\end{equation*}
where $K$ is a positive constant.
When $K=qB^{q-1}$,
this equation coincides with the linearized equation of \eqref{(B)-eq} around $\varphi_0$.
Let $b_K(\xi)$ and $\mu_K<0$ be the first positive eigenfunction with $b_K(0)=1$
and the first eigenvalue of
\[
\begin{cases}
\dis
 -\left( b'' - \frac{\xi}{2}b'\right) = \mu b
& \text{in } \R_+,
\\ \dis
 \pa_{\nu}b(0) = Kb(0).
&
\end{cases}
\]
Then
by Lemma \ref{Keep-lem} with \eqref{U(a,b,c)-eq},
we find that $b_K(\xi)$ is written by
\[
\begin{array}{lll}
\dis
 b_K(\xi)
\hspace{-2mm}&=&\hspace{-2mm} \dis
 \left( \frac{\Gamma(-\mu_K+1/2)}{\sqrt{\pi}} \right)
 U(-\mu_K,1/2,\xi^2/4)
\\[4mm]
\hspace{-2mm}&=&\hspace{-2mm} \dis
 c_K\int_0^{\infty}e^{-\xi^2t/4}t^{-\mu_K-1}(1+t)^{-1/2+\mu_K}dt,
\end{array}
\]
where $c_K=\Gamma(-\mu_K+1/2)/\sqrt{\pi}\Gamma(-\mu_K)$.
To estimate the above integral
we change variables $s=\xi^2t$ and integrate by parts,
then we get
\[
\begin{array}{lll}
\dis
 b_K(\xi)
\hspace{-2mm}&=&\hspace{-2mm} \dis
 c_K\xi^{2\mu_K}\int_0^{\infty}e^{-s/4}s^{-\mu_K-1}\left( 1+\xi^{-2}s \right)^{-1/2+\mu_K}ds
\\[4mm] \dis
\hspace{-2mm}&\leq&\hspace{-2mm} \dis
 c_K\xi^{2\mu_K}\int_0^{\infty}e^{-s/4}s^{-\mu_K-1}ds.
\end{array}
\]
Here
we fix $K_0>1$.
Then
there exist two positive constants $C_1<C_2$ such that
$C_1\leq -\mu_K\leq C_2$ for $K_0^{-1}<K<K_0$.
Therefore
there exists $C_0>0$ such that
\[
 \sup_{K_0^{-1}<K<K_0}\int_0^{\infty}e^{-s/4}s^{-\mu_K-1}ds\leq C_0.
\]
As a consequence,
we obtain for $K_0^{-1}\leq K\leq K_0$
\[
\begin{array}{lll}
\dis
 b_K(\xi)
\leq
 C_0c_K\xi^{2\mu_K}.
\end{array}
\]
Furthermore
since $b_K(0)=1$ and $b_K'(0)=-K$,
there exists $C_0'>0$ such that
\begin{equation}\label{Lunch-eq}
 \sup_{K_0^{-1}<K<K_0}\sup_{0<\xi<1}b_K(\xi) \leq C_0'.
\end{equation}
Combining the above estimates,
we obtain the following lemma.

%%%%%%%%%%%%%%%%%%%%%%%%%%%%%%%%%%%%%%%%%%%%%%%%%%%%%%%%%%%
\begin{lem}\label{Pii-lem}
For any $K_0>1$ there exists $c=c(K_0)>0$ such that for $K_0^{-1}<K<K_0$
\[
 b_K(\xi) \leq c(1+\xi)^{2\mu_K}
\hspace{5mm}\mathrm{for}\ \xi\in\R_+.
\]
\end{lem}
%%%%%%%%%%%%%%%%%%%%%%%%%%%%%%%%%%%%%%%%%%%%%%%%%%%%%%%%%%%

Next we compute $b_K'(\xi)/b_K(\xi)$ and $b_K''(\xi)/b_K(\xi)$.
By using integration by parts,
we get
\begin{eqnarray*}
 b_K'(\xi)
\hspace{-2mm}&=&\hspace{-2mm}
 -\frac{c_K\xi}{2}\int_0^\infty e^{-\xi^2t/4}t^{-\mu_K}(1+t)^{-1/2+\mu_K}dt
\\
\hspace{-2mm}&=&\hspace{-2mm}
 -\frac{2c_K}{\xi}\int_0^{\infty} e^{-\xi^2t/4}
 \left( t^{-\mu_K}(1+t)^{-1/2+\mu_K} \right)'dt
\\
\hspace{-2mm}&=&\hspace{-2mm}
 \frac{2c_K}{\xi}\int_0^{\infty} e^{-\xi^2t/4}
 \left( \mu_K+\left(\frac{1}{2}-\mu_K\right)\frac{t}{1+t} \right)
 t^{-\mu_K-1}(1+t)^{-1/2+\mu_K}dt.
\end{eqnarray*}
Hence
it holds that
\[
 \frac{|b_K'(\xi)|}{b_K(\xi)} \leq \frac{1+4|\mu_K|}{\xi}.
\]
Here
we again fix $K_0>1$.
Since $|\mu_K|<c$ for $K\in(0,K_0)$,
there exists $c>0$ such that
\[
 \sup_{0<K<K_0}
 \sup_{\xi\geq1}\left( \frac{|b_K'(\xi)|}{b_K(\xi)} \right) \leq c.
\]
Repeating the above argument,
we see that
\[
 \sup_{0<K<K_0}
 \sup_{\xi\geq1}
 \left( \frac{|b_K''(\xi)|}{b_K(\xi)} \right) \leq c'.
\]
Furthermore
by the same way as \eqref{Lunch-eq},
we obtain
\[
 \sup_{0<K<K_0}
 \sup_{\xi\leq1}\left( \frac{|b_K'(\xi)|+|b_K''(\xi)|}{b_K(\xi)} \right) \leq c''.
\]
Therefore
we conclude the following lemma.

%%%%%%%%%%%%%%%%%%%%%%%%%%%%%%%%%%%%%%%%%%%%%%%%%%%%%%%%%%%
\begin{lem}\label{23-lem}
Let $B_K(\xi)=b_K'(\xi)/b_K(\xi)$.
Then for any $K_0>1$ there exists $c=c(K_0)>0$ such that
\[
 \sup_{0<K<K_0}
 \sup_{\xi\in\R_+}
 \left( B_K(\xi)+|B_K'(\xi)| \right) \leq c.
\]
\end{lem}
%%%%%%%%%%%%%%%%%%%%%%%%%%%%%%%%%%%%%%%%%%%%%%%%%%%%%%%%%%%

We introduce another function:
\[
 w_K(y,s)=e^{\mu_Ks}v(y,s)/b_K(y_n).
\]
Then $w_K(y,s)$ satisfies
\begin{equation}\label{sendai-eq}
\begin{cases}
\dis
 \pa_sw_K =
 \Delta w_K - \frac{y}{2}\cdot\nabla w_K + 2B_K(y_n)\pa_nw_K
& \text{in }\R_+^n\times(0,\infty),
\\[2mm] \dis
 \pa_{\nu}w_K = 0
& \text{on }\pa\R_+^n\times(0,\infty),
\end{cases}
\end{equation}
where the coefficient $B_K(y_n)$ is given by
\[
 B_K(y_n) = \left(\frac{b_K'(y_n)}{b_K(y_n)}\right).
\]
First
we construct a heat kernel $\Gamma_K(y,\xi,s)$ of \eqref{sendai-eq}
and provide time local heat kernel estimates.
Let $\gamma'(y',\xi',s)$ be the heat kernel of
\[
 \pa_s\theta = \Delta'\theta -\frac{y'}{2}\cdot\nabla'\theta
\hspace{5mm}\text{in }\R^{n-1}\times(0,\infty),
\]
which is explicitly given by (c.f. p.\hspace{1mm}141 in \cite{Herrero-V1})
\begin{equation*}
 \gamma'(y',\xi',s) =
 \left( \frac{1}{4\pi} \right)^{(n-1)/2}
 \left( \frac{1}{1-e^{-s}} \right)^{(n-1)/2}
 \exp\left( -\frac{|y'e^{-s/2}-\xi'|}{4(1-e^{-s})} \right).
\end{equation*}
Now we construct a heat kernel $\gamma_K(z,\zeta,s)$ of
\begin{equation}\label{w2K-eq}
\begin{cases}
\dis
 \pa_s\vartheta = \vartheta_{zz} - \frac{z}{2}\vartheta_z + 2B_K(z)\vartheta_z
& \text{in }\R_+\times(0,\infty),
\\ \dis
 \pa_{\nu}\vartheta = 0
& \text{on }z=0,\ s\in(0,\infty).
\end{cases}
\end{equation}
% Now
% we construct the heat kernel $\bar{\gamma}_K$ and
% give the heat kernel estimate of $\bar{\gamma}_K$.
Let $\vartheta_K(z,s)$ be a solution of \eqref{w2K-eq} and put
\[
 U_K(x,t) = \vartheta_K\left( (1-t)^{-1/2}x,-\log(1-t) \right)
\hspace{5mm}(x\in\R_+).
\]
Then $U_K(x,t)$ solves
\begin{equation}\label{sekitome-eq}
\begin{cases}
\dis
 \pa_tU = U_{xx} + 2\left( \frac{{\cal B}_K(x,t)}{\sqrt{1-t}} \right)U_x
& \text{in }\R_+\times(0,1),
\\[2mm] \dis
 \pa_{\nu}U = 0
& \text{on }x=0,\ t\in(0,1),
\end{cases}
\end{equation}
where ${\cal B}_K(x,t)=B_K(x/\sqrt{1-t})$.
Let $G_0(x,\xi,t)$ be the standard heat kernel with the Neumann boundary condition on $\R_+$
given by
\[
 G_0(x,\zeta,t) = \frac{1}{\sqrt{4\pi t}}\left( e^{-|x-\zeta|^2/4t}+e^{-|x+\zeta|^2/4t} \right).
\]
It is known that
by the Levi parametrix method (see pp.\hspace{1mm}356--363 in \cite{Ladyzenskaja-S-U}),
\eqref{sekitome-eq} admits the heat kernel $G_K(x,\zeta,t,\tau)$ which has the following form:
\begin{equation*}
% \label{G_K-eq}
 G_K(x,\zeta,t,\tau) = G_0(x,\zeta,t-\tau) + {\cal G}_K(x,\zeta,t,\tau),
\end{equation*}
where ${\cal G}_K(x,\zeta,t,\tau)$ is given by
\begin{equation}\label{calG_K-eq}
 {\cal G}_K(x,\zeta,t,\tau) =
 \int_\tau^td\sigma\int_0^\infty G_0(x,\eta,t-\sigma)Q_K(\eta,\zeta,\sigma,\tau)d\eta
\end{equation}
and
$Q_K(\eta,\zeta,\sigma,\tau)$ is a unique solution of the integral equation of
\[
 Q_K(\eta,\zeta,\sigma,\tau) = F_K(\eta,\zeta,\sigma,\tau) +
 \int_\tau^td\nu\int_0^\infty F_K(\eta,\mu,\sigma,\nu)Q_K(\mu,\zeta,\nu,\tau)d\mu
\]
with
\[
 F_K(\eta,\zeta,\sigma,\tau) =
 2\left( \frac{{\cal B}_K(\eta,\sigma)}{\sqrt{1-\sigma}} \right)
 \pa_xG_0(\eta,\zeta,\sigma-\tau).
\]
Here we put
\[
 G_0^\epsilon(\eta,\zeta,t) =
 \frac{1}{\sqrt{4\pi t}}
 \left(
 \exp\left( -\frac{(\eta-\zeta)^2}{4(1+\epsilon)t} \right)
 +
 \exp\left( -\frac{(\eta+\zeta)^2}{4(1+\epsilon)t} \right)
 \right)
\]
and fix three constants
\[
 K_0>1,\hspace{7.5mm} \epsilon_0\in(0,1),\hspace{7.5mm} \sigma_*\in(0,1).
\]
Then
by Lemma \ref{23-lem},
for any $a\in(0,1)$
there exists $c=c(a)>0$ such that for $K\in(0,K_0)$ and $\epsilon\geq\epsilon_0$
\[
\begin{array}{c}
\dis
 |F_K(\eta,\zeta,\sigma,\tau)| \leq
 c\left( \frac{G_0^\epsilon(\eta,\zeta,\sigma-\tau)}{\sqrt{\sigma-\tau}} \right),
% \hspace{5mm}\text{for }\eta,\zeta\in\R_+,\ 0<\tau<\sigma<\sigma_*,
\\[4mm] \dis
 |F_K(\eta,\zeta,\sigma,\tau)-F_K(\eta',\zeta,\sigma,\tau)|
\leq
 c\left( \frac{|\eta-\eta'|^a}{(\sigma-\tau)^{a_1}} \right)
% \\[4mm] \hspace{0mm}
 \left(
 G_0^\epsilon(\eta,\zeta,\sigma-\tau)+G_0^\epsilon(\eta',\zeta,\sigma-\tau)
 \right)
\end{array}
\]
for $\eta,\zeta\in\R_+$ and $0<\tau<\sigma<\sigma_*$, where $a_1=(1+a)/2$.
Hence
by the same calculations as in pp.\hspace{1mm}360--363 in \cite{Ladyzenskaja-S-U},
for any $a'\in(0,1)$
there exist $c=c(a')>0$ such that for $K\in(0,K_0)$ and $\epsilon\geq\epsilon_0$
\begin{equation}\label{Q_K-eq}
\begin{array}{c}
\dis
 |Q_K(\eta,\zeta,\sigma,\tau)|
\leq
 \frac{c}{\sqrt{\sigma-\tau}}G_0^{\epsilon}(\eta,\zeta,\sigma-\tau)
\\[4mm] \dis
  |Q_K(\eta,\zeta,\sigma,\tau)-Q_K(\eta',\zeta,\sigma,\tau)|
\leq
 c\left( \frac{|\eta-\eta'|^{a'}}{(\sigma-\tau)^{a_1'}} \right)
 \left(
 G_0^\epsilon(\eta,\zeta,\sigma-\tau)+G_0^\epsilon(\eta',\zeta,\sigma-\tau)
 \right)
\end{array}
\end{equation}
for $\eta,\zeta\in\R_+$ and $0<\tau<\sigma<\sigma_*$,
where $a_1'=(1+a')/2$.
Therefore
by the same way as in pp.\hspace{1mm}376--378 in \cite{Ladyzenskaja-S-U},
we find that
for any $t_*\in(0,1)$
there exists $c=c(t_*)>0$ such that
for any $K\in(0,K_0)$ and $\epsilon\geq\epsilon_0$
\begin{equation}\label{GGG-eq}
\begin{array}{c}
\dis
 |\pa_x^i{\cal G}_K(x,\zeta,t,\tau)|
\leq
 c(t-\tau)^{(1-i)/2}G_0^\epsilon(x,\zeta,t-\tau)
\hspace{5mm}(i=0,1,2),
\\[3mm] \dis
|\pa_t{\cal G}_K(x,\zeta,t,\tau)|
\leq
 c(t-\tau)^{-1/2}G_0^\epsilon(x,\zeta,t-\tau)
\end{array}
\end{equation}
for $x,\zeta\in\R_+$ and $0<\tau<t<t_*$.
Moreover
from \eqref{calG_K-eq} and \eqref{Q_K-eq},
a direct computation shows that
\begin{equation}\label{GGG2-eq}
 \lim_{x\to0+}\pa_x{\cal G}_K(x,\zeta,t,\tau)=0
\hspace{5mm}\text{for } \zeta\in\overline{\R}_+,\ t>\tau\geq0.
\end{equation}
Therefore
we fined that
$G_K(x,\zeta,t,\tau)$ satisfies \eqref{sekitome-eq} for $0<\tau<t<1$.
Moreover
from \eqref{GGG-eq},
it holds that for $f\in BC(\overline{\R}_+)$
\[
 \lim_{t\to\tau}\int_0^\infty G_K(x,\zeta,t,\tau)f(\zeta)d\zeta
=
 \lim_{t\to\tau}\int_0^\infty G_0(x,\zeta,t-\tau)f(\zeta)d\zeta
=
 f(x).
\]
As a consequence,
$G_K(x,\zeta,t,\tau)$ turns out to be the heat kernel of \eqref{sekitome-eq}.
Going back to the variable $(z,s)$,
we obtain the heat kernel of \eqref{w2K-eq}: 
\[
 \gamma_K(z,\zeta,s) = G_K\left( e^{-s/2}z,\zeta,1-e^{-s},0 \right).
\]
Then from \eqref{GGG-eq},
for any $s_*>0$ there exists $c=c(s_*)>0$ such that for $K\in(0,K_0)$ and $\epsilon\geq\epsilon_0$
\begin{equation}\label{tsuyoi-eq}
\begin{array}{l}
\dis
 \gamma_K(z,\zeta,s)
\leq
 cG_0^\epsilon\left( e^{-s/2}z,\zeta,1-e^{-s},0 \right)
\\[3.5mm] \dis \hspace{5mm}
=
 \frac{c}{(1-e^{-s})^{1/2}}
 \left(
 \exp\left( -\frac{\left( ze^{-s/2}-\zeta \right)^2}{4(1+\epsilon)(1-e^{-s})} \right)
%  \right.
% \\[6mm] \dis \hspace{15mm}
%  \left.
 +
 \exp\left( -\frac{\left( ze^{-s/2}+\zeta \right)^2}{4(1+\epsilon)(1-e^{-s})} \right)
 \right)
\end{array}
\end{equation}
for $z,\zeta\in\R_+$ and $s\in(0,s_*)$.
Since the heat kernel $\Gamma_K(y,\xi,s)$ of \eqref{sendai-eq} is given by
\[
 \Gamma_K(y,\xi,s) = \gamma'(y',\xi',s)\gamma_K(y_n,\xi_n,s),
\]
summarizing the above estimates,
we obtain the following lemma.

%%%%%%%%%%%%%%%%%%%%%%%%%%%%%%%%%%%%%%%%%%%%%%%%%%%%%%%%%%%
\begin{lem}\label{24-lem}
For any $K_0>1$, $\epsilon_0>0$ and $s_*>0$
there exists $c=c(K_0,\epsilon_0,s_*)>0$ such that
for $K\in(0,K_0)$ and $\epsilon\geq\epsilon_0$
\[
\begin{array}{lll}
\dis
 \Gamma_K(y,\xi,s)
\hspace{-2mm}&\leq&\hspace{-2mm} \dis
 \frac{c}{(1-e^{-s})^{n/2}}
 \exp\left( -\frac{|y'e^{-s/2}-\xi'|}{4(1-e^{-s})} \right)
 \left(
 \exp\left( -\frac{\left( ze^{-s/2}-\zeta \right)^2}{4(1+\epsilon)(1-e^{-s})} \right)
 \right.
\\[6mm] \hspace{-2mm}&&\hspace{-2mm}\hspace{5mm}\dis
 \left.
 +
 \exp\left( -\frac{\left( ze^{-s/2}+\zeta \right)^2}{4(1+\epsilon)(1-e^{-s})} \right)
 \right)
\hspace{5mm}\mathrm{for}\ y,\xi\in\R_+^n,\ s\in(0,s_*).
\end{array}
\]
\end{lem}
%%%%%%%%%%%%%%%%%%%%%%%%%%%%%%%%%%%%%%%%%%%%%%%%%%%%%%%%%%%

As a consequence of this heat kernel estimate,
we obtain $L^r$-$L^p$ type estimates.
Here we define another weighted $L^p$-space:
\[
 \|w\|_{L_{K,\rho}^p(\R_+^n)}
=
 \int_{\R_+^n} |w(y)|^p b_K(y_n)^2\rho(y)dy.
\]

%%%%%%%%%%%%%%%%%%%        Lemma        %%%%%%%%%%%%%%%%%%%
\begin{lem}\label{25-lem}
For any $K_0>1$, $p,r\in(1,\infty)$, $\delta\in(0,p-1)$ and $s_*>0$
there exists $c=c(K_0,p,r,s_*,\delta)>0$
such that for $K\in(K_0^{-1},K_0)$ and $s\in(\max\{0,\log\{(r-1)/(p-1-\delta)\}\},s_*)$
\[
\begin{array}{c}
\dis
 \left\|
 \int_{\R_+^n}\Gamma_K(y',\xi,s)|w_0(\xi)|d\xi
 \right\|_{L_{\rho}^r(\pa\R_+^n)}
\leq
%  c\left( \frac{\|w_0\|_{L_{K,\rho}^p(\R_+^n)}}{(1-e^{-s})^{n/2p}} \right),
 c(1-e^{-s})^{-n/2p}\|w_0\|_{L_{K,\rho}^p(\R_+^n)},
\\[6mm] \dis
 \left\|
 \int_{\pa\R_+^n}\Gamma_K(y',\xi',s)|w_0(\xi')|d\xi'
 \right\|_{L_{\rho}^r(\pa\R_+^n)}
\leq
 c(1-e^{-s})^{-((n-1)/2p+1/2)}\|w_0\|_{L_{\rho}^p(\pa\R_+^n)}.
\end{array}
\]
\end{lem}
%%%%%%%%%%%%%%%%%%%%%%%%%%%%%%%%%%%%%%%%%%%%%%%%%%%%%%%%%%%

%%%%%%%%%%%%%%%%%%%        Proof        %%%%%%%%%%%%%%%%%%%
\begin{proof}
Let $y'\in\pa\R_+^n$ and $p'>1$ be a dual exponent of $p$.
From the H${\rm\ddot{o}}$lder inequality,
we easily see that
\begin{equation}
\label{Gets1-eq}
\begin{array}{l}
\dis
 \int_{\R_+^n}\Gamma_K(y',\xi,s)|w_0(\xi)|d\xi
\leq
 \|w_0\|_{L_{K,\rho}^p(\R_+^n)}
% \\[-5mm] \dis \hspace{25mm}
%  \times
 \left(
 \int_{\R_+^n}
 \Gamma_K(y',\xi,s)^{p'}
 \left( \frac{e^{|\xi|^2/4p}}{b_K(\xi_n)^{2/p}} \right)^{p'}
 d\xi
 \right)^{1/p'},
\\[6mm] \dis
 \int_{\pa\R_+^n}\Gamma_K(y',\xi',s)|w_0(\xi')|d\xi'
\leq
 \|w_0\|_{L_{\rho}^p(\pa\R_+^n)}
 \left(
 \int_{\R_+^n}
 \left( \Gamma_K(y',\xi',s)e^{|\xi'|^2/4p} \right)^{p'}d\xi'
 \right)^{1/p'}.
\end{array}
\end{equation}
Here we note that
\begin{equation}\label{Gets2-eq}
 \frac{|y'e^{-s/2}-\xi'|}{4(1-e^{-s})}-\frac{|\xi'|^2}{4p}
=
 \frac{p-(1-e^{-s})}{4p(1-e^{-s})}
 \left| \xi'-\frac{pe^{-s/2}y'}{p-(1-e^{-s})} \right|^2
-
 \frac{e^{-s}|y'|^2}{4(p-(1-e^{-s}))}
\end{equation}
and
\[
 \frac{\xi_n^2}{4(1+\epsilon)(1-e^{-s})}-\frac{\xi_n^2}{4p}
=
 \frac{\xi_n^2}{4p(1+\epsilon)(1-e^{-s})}
 \left(
 p-(1+\epsilon)(1-e^{-s})
 \right).
\]
Therefore
we apply Lemma \ref{24-lem} with $\epsilon=(p-1)/2$ to get
\[
\begin{array}{c}
\dis
 \int_{\R_+^n}
 \Gamma_K(y',\xi,s)^{p'}
 \left( \frac{e^{|\xi|^2/4p}}{b_K(\xi_n)^{2/p}} \right)^{p'}
 d\xi
\leq
 \left( \frac{c}{(1-e^{-s})^{n/2p}} \right)^{p'}
 \exp\left( \frac{p'e^{-s}|y'|^2}{4(p-(1-e^{-s}))} \right),
\\[5mm] \dis
 \int_{\pa\R_+^n}
 \left( \Gamma_K(y',\xi',s)e^{|\xi'|^2/4p} \right)^{p'}d\xi'
\leq
 \left( \frac{c}{(1-e^{-s})^{(n-1)/2p+1/2}} \right)^{p'}
 \exp\left( \frac{p'e^{-s}|y'|^2}{4(p-(1-e^{-s}))} \right),
\end{array}
\]
where we used $b_K(y_n)\leq c(1+y_n)^{2\mu_K}$ for $K\in(K_0^{-1},K_0)$
(see Lemma \ref{Pii-lem}).
From these estimates,
we obtain
\[
\begin{array}{l}
\dis
 \int_{\R_+^n}\Gamma_K(y',\xi,s)|w_0(\xi)|d\xi
% \\ \dis \hspace{20mm}
\leq
 \left( \frac{c}{(1-e^{-s})^{n/2p}} \right)
 \exp\left( \frac{e^{-s}|y'|^2}{4(p-(1-e^{-s}))} \right)
 \|w_0\|_{L_{K,\rho}^p(\R_+^n)},
\\[6mm] \dis
 \int_{\pa\R_+^n}\Gamma_K(y',\xi',s)|w_0(\xi')|d\xi'
% \\ \dis \hspace{20mm}
\leq
 \left( \frac{c}{(1-e^{-s})^{(n-1)/2p+1/2}} \right)
  \exp\left( \frac{e^{-s}|y'|^2}{4(p-(1-e^{-s}))} \right)
 \|w_0\|_{L_{\rho}^p(\pa\R_+^n)}.
\end{array}
\]
Moreover
we note that
\[
 \frac{re^{-s}|y'|^2}{4(p-(1-e^{-s}))} -
 \frac{|y'|^2}{4}
=
 -\left( \frac{(p-1)-(r-1)e^{-s}}{4(p-(1-e^{-s}))} \right)|y'|^2.
\]
Then
it holds that for $\delta\in(0,p-1)$
\[
 (p-1)-(r-1)e^{-s} \geq \delta
\hspace{5mm}\text{if } s\geq\log\left( \frac{r-1}{p-1-\delta} \right).
\]
Hence
it follows that
\[
\begin{array}{c}
\dis
 \left\|
 \int_{\R_+^n}\Gamma_K(y',\xi,s)|w_0(\xi)|d\xi
 \right\|_{L_{\rho}^r(\pa\R_+^n)}
\leq
 \left( \frac{c\delta^{-(n-1)/2r}}{(1-e^{-s})^{n/2p}} \right)\|w_0\|_{L_{K,\rho}^p(\R_+^n)},
\\[4mm] \dis
 \left\|
 \int_{\pa\R_+^n}\Gamma_K(y',\xi',s)|w_0(\xi')|d\xi'
 \right\|_{L_{\rho}^r(\pa\R_+^n)}
\leq
 \left( \frac{c\delta^{-(n-1)/2r}}{(1-e^{-s})^{(n-1)/2p+1/2}} \right)
 \|w_0\|_{L_{\rho}^p(\pa\R_+^n)}
\end{array}
\]
for $s\geq\log\{(r-1)/(p-1-\delta)\}$,
which completes the proof.
\end{proof}
%%%%%%%%%%%%%%%%%%%%%%%%%%%%%%%%%%%%%%%%%%%%%%%%%%%%%%%%%%%

Next
we provide time global estimates of $\gamma_K(z,\zeta,s)$
to establish time global estimates of $\Gamma_K(y,\xi,s)$.
We fix $\theta\in(0,1)$ such that $(1-\theta)^2/\theta^2=2$
and set
\[
 \Lambda(z,\zeta,s)
=
 \frac{1}{(1-e^{-s})^{1/2}}
 \left(
 \exp\left( -\frac{|ze^{-s/2}-\theta\zeta|^2}{4(1-e^{-s})} \right)
 +
 \exp\left( -\frac{|ze^{-s/2}+\theta\zeta|^2}{4(1-e^{-s})} \right)
 \right).
\]
Then $\Lambda(z,\zeta,s)$ satisfies for $\zeta\in\R_+$
\begin{equation}\label{Yuka-eq}
\begin{cases}
\dis
 \pa_s\Lambda = \Lambda_{zz} - \frac{z}{2}\Lambda_z
& \text{in }\R_+\times(0,\infty),
\\
 \Lambda_z = 0
& \text{on }z=0,\ s\in(0,\infty).
\end{cases}
\end{equation}

%%%%%%%%%%%%%%%%%%%%%%%%%%%%%%%%%%%%%%%%%%%%%%%%%%%%%%%%%%%
\begin{lem}\label{26-lem}
For any $K_0>1$
there exists $c_0=c_0(K_0)>0$ such that for $K\in(0,K_0)$
\[
 \gamma_K(z,\zeta,s) \leq c_0(1-e^{-s})^{-1/2}
\hspace{5mm}\mathrm{for}\ z,\zeta,\in\R_+,\ s\geq0.
\]
Moreover
for any $K_0>1$
there exists $c_1=c_1(K_0)>0$ such that for $K\in(0,K_0)$
\[
 \gamma_K(z,\zeta,s)\leq c_1\Lambda(z,\zeta,s)
\]
for $z\in(0,(\theta\zeta-2)e^{s/2})$, $\zeta\in(2/\theta,\infty)$ and $s\geq1$.
\end{lem}
%%%%%%%%%%%%%%%%%%%%%%%%%%%%%%%%%%%%%%%%%%%%%%%%%%%%%%%%%%%

%%%%%%%%%%%%%%%%%%%%%%%%%%%%%%%%%%%%%%%%%%%%%%%%%%%%%%%%%%%
\begin{proof}
We fix $\epsilon_0=s_*=1$ in \eqref{tsuyoi-eq}.
Then
there exists $c_0>0$ such that for $K\in(0,K_0)$
\[
 \gamma_K(z,\zeta,s) \leq c_0(1-e^{-s})^{-1/2}
\hspace{5mm}\text{for }s\in(0,1).
\]
Since
a constant function $\bar{\vartheta}(z,s):=c_0(1-e^{-1})^{-1/2}$ is a solution of \eqref{w2K-eq},
from a comparison argument,
we obtain $\gamma_K(z,\zeta,s)\leq\bar{\vartheta}(z,s)$ for $s\geq1$.
Since
$(1-e^{-s})^{-1/2}\geq1$ for $s\geq0$,
it holds that $\gamma_K(z,\zeta,s)\leq c_0(1-e^{-1})^{-1/2}(1-e^{-s})^{-1/2}$ for $s\geq1$,
which show the first statement.
Next
we show the second statement.
Since $(1-\theta)/\theta=\sqrt{2}$,
we note that
\begin{equation}\label{Oppi-eq}
 0<
 \theta\zeta e^{s/2}-z
\leq
 \theta\zeta e^{s/2}
\leq
 \frac{\theta}{1-\theta}(\zeta e^{s/2}-z)
=
 \frac{1}{\sqrt{2}}(\zeta e^{s/2}-z)
\hspace{5mm}\text{for }z\leq\theta\zeta e^{s/2}.
\end{equation}
We fix $\epsilon_0=s_*=1$ in \eqref{tsuyoi-eq} again.
Then
by \eqref{tsuyoi-eq} and \eqref{Oppi-eq},
we see that
\begin{equation}\label{previous-eq}
\begin{array}{lll}
\dis
 \gamma_K(z,\zeta,1)
\hspace{-2mm}&\leq&\hspace{-2mm}\dis
 \frac{2c}{(1-e^{-1})^{1/2}}
 \exp\left( -\frac{|ze^{-1/2}-\zeta|^2}{8(1-e^{-1})} \right)
\\[4.5mm] \hspace{-2mm}&\leq&\hspace{-2mm}\dis
 \frac{2c}{(1-e^{-1})^{1/2}}
 \exp\left( -\frac{|ze^{-1/2}-\theta\zeta|^2}{4(1-e^{-1})} \right)
\\[4.5mm] \hspace{-2mm}&\leq&\hspace{-2mm}\dis
 4c\Lambda(z,\zeta,1)
\hspace{5mm}\text{for }z\leq\theta\zeta e^{1/2}.
\end{array}
\end{equation}
By definition of $\Lambda$,
there exists $a_1>0$ such that
\[
 \Lambda\left( (\theta\zeta-2)e^{s/2},\zeta,s \right)
\geq
 \frac{e^{-1/(1-e^{-s})}}{(1-e^{-s})^{1/2}} \geq a_1
\hspace{5mm}\text{for }\zeta\in(2/\theta,\infty),\ s\geq1.
\]
From the first statement,
we can define
\[
 a_2 = \sup_{s\geq1}\sup_{z,\zeta\in\R_+}\gamma_K(z,\zeta,s).
\]
Hence
we obtain
\begin{equation}\label{yakedo2-eq}
 \gamma_K\left( (\theta\zeta-2)e^{s/2},\zeta,s \right)
\leq
 \left( \frac{a_2}{a_1} \right)
 \Lambda\left( (\theta\zeta-2)e^{s/2},\zeta,s \right)
\hspace{5mm}\text{for }\zeta\in(2/\theta,\infty),\ s\geq1.
\end{equation}
Now
we claim that
\begin{equation}\label{Chako-eq}
 \Lambda_z(z,\zeta,s)\geq0
\end{equation}
for $z\in(0,(\theta\zeta-2)e^{s/2})$ and $\zeta\in(2/\theta,\infty)$.
By definition of $\Lambda$,
we see that
\begin{equation}\label{Lambda_z-eq}
\begin{array}{lll}
\dis
 \Lambda_z(z,\zeta,s)
\hspace{-2mm}&=&\dis\hspace{-2mm}
 \frac{e^{-s/2}}{(1-e^{-s})}
 \left(
 \frac{\theta\zeta-ze^{-s/2}}{2\sqrt{1-e^{-s}}}
 \exp\left( -\frac{|ze^{-s/2}-\theta\zeta|^2}{4(1-e^{-s})} \right)
 \right.
\\[5mm]&&\dis\hspace{35mm}
 -\left.
 \frac{\theta\zeta+ze^{-s/2}}{2\sqrt{1-e^{-s}}}
 \exp\left( -\frac{|ze^{-s/2}+\theta\zeta|^2}{4(1-e^{-s})} \right)
 \right).
\end{array}
\end{equation}
Let
$\zeta_1=\theta\zeta-ze^{-s/2}/2\sqrt{1-e^{-s}}$
and
$\zeta_2=\theta\zeta+ze^{-s/2}/2\sqrt{1-e^{-s}}$,
and set $g(\zeta)=\zeta e^{-\zeta^2}$.
Then
\eqref{Lambda_z-eq} is rewritten by
\[
 \Lambda_z(z,\zeta,s)
=
 \frac{e^{-s/2}}{1-e^{-s}}
 \left( g(\zeta_1)-g(\zeta_2) \right).
\]
Then
we observe that
\[
 \zeta_1
=
 \frac{\theta\zeta-ze^{-s/2}}{2\sqrt{1-e^{-s}}}
\geq 
 \frac{\theta\zeta-(\theta\zeta-2)}{2\sqrt{1-e^{-s}}}
=
 \frac{1}{\sqrt{1-e^{-s}}}
\geq1
\]
for $z\in(0,(\theta\zeta-2)e^{s/2})$ and $\zeta\in(2/\theta,\infty)$.
Since $g(\zeta)=\zeta e^{-\zeta^2}$ is monotone decreasing for $\zeta\geq 1/\sqrt{2}$,
it follows from $\zeta_1<\zeta_2$ that
\[
 \Lambda_z(z,\zeta,s)
=
 \frac{e^{-s/2}}{1-e^{-s}}
 \left( g(\zeta_1)-g(\zeta_2) \right)>0
\]
for $z\in(0,(\theta\zeta-2)e^{s/2})$ and $\zeta\in(2/\theta,\infty)$,
which shows the claim.
Therefore
since $b_K'(z)<0$,
it holds from \eqref{Chako-eq} that
\[
 B_K(z)\Lambda_z(z,\zeta,s) =
 \left( \frac{b_K'(z)}{b_K(z)} \right)\Lambda_z(z,\zeta,s) \leq 0
\]
for $z\in(0,(\theta\zeta-2)e^{s/2})$ and $\zeta\in(2/\theta,\infty)$.
Therefore
since
$\Lambda(z,\zeta,s)$ satisfies \eqref{Yuka-eq} for $\zeta\in\R_+$,
we obtain
\[
 \left( \pa_ s-\pa_{zz}^2+\frac{z}{2}\pa_z-2B_K(z)\pa_z \right)\Lambda(z,\zeta,s)
\geq 0
\]
for
$z\in(0,(\theta\zeta-2)e^{s/2})$ and $\zeta\in(2/\theta,\infty)$.
Here
we apply a comparison argument in an expanding domain
$Q=\{(z,s);\in\R_+\times(1,\infty);z\in(0,(\theta\zeta-2)e^{s/2})\}$
for $\zeta\in(2/\theta,\infty)$.
From 
\eqref{previous-eq} and \eqref{yakedo2-eq},
we conclude that
\[
 \gamma_K(z,\zeta,s)
\leq
 \max\left\{ 4c,a_2/a_1 \right\}\Lambda(z,\zeta,s)
\hspace{5mm}\text{for }(z,s)\in Q,
\]
hence the proof is completed.
\end{proof}
%%%%%%%%%%%%%%%%%%%%%%%%%%%%%%%%%%%%%%%%%%%%%%%%%%%%%%%%%%%

Finally
we provide time global $L^\infty$-$L_{K,\rho}^2$ estimates of solutions of \eqref{sendai-eq}.

%%%%%%%%%%%%%%%%%%%%%%%%%%%%%%%%%%%%%%%%%%%%%%%%%%%%%%%%%%%
\begin{lem}\label{27-lem}
For any $K_0,R>1$
there exists $c=c(K_0,R)>0$ such that for $K\in(0,K_0)$
\[
 \int_{\R_+^n}\Gamma_K(y,\xi,s)|w_0(\xi)|d\xi
\leq
 c\exp\left( \frac{e^{-s}|y'|^2}{4(1+e^{-s})} \right)
 \|w_0\|_{L_{K,\rho}^2(\R_+^n)}
\]
for $y=(y',y_n)\in\R^{n-1}\times(0,R)$ and $s\geq1$.
\end{lem}
%%%%%%%%%%%%%%%%%%%%%%%%%%%%%%%%%%%%%%%%%%%%%%%%%%%%%%%%%%%

%%%%%%%%%%%%%%%%%%%        Proof        %%%%%%%%%%%%%%%%%%%
\begin{proof}
We note from \eqref{Gets1-eq} that
\[
\begin{array}{l}
\dis
 \int_{\R_+^n}\Gamma_K(y,\xi,s)|w_0(\xi)|d\xi
\leq
 \|w_0\|_{L_{K,\rho}^2(\R_+^n)}
 \left(
 \int_{\R_+^n}\gamma(y',\xi',s)^2
 \right.
\\ \dis \hspace{60mm}
 \left.
\times
 \gamma_K(y_n,\xi_n,s)^2
 \left( \frac{e^{|\xi|^2/8}}{b_K(\xi_n)} \right)^2
 d\xi
 \right)^{1/2}.
\end{array}
\]
Then by using \eqref{Gets2-eq},
we get
\[
\begin{array}{l}
\dis
 \int_{\R_+^n}
 \gamma(y',\xi',s)^2\gamma_K(y_n,\xi_n,s)^2
 \left( \frac{e^{|\xi|^2/8}}{b_K(\xi_n)} \right)^2d\xi
\leq
 c\left( \frac{1}{1-e^{-s}} \right)^{(n-1)/2}
\\[4mm] \dis \hspace{40mm}
\times
 \exp\left( \frac{e^{-s}|y'|^2}{2(1+e^{-s})} \right)
 \int_0^{\infty}
 \gamma_K(y_n,\zeta,s)^2
 \left( \frac{e^{\zeta^2/4}}{b_K(\zeta)^2} \right)
 d\zeta.
\end{array}
\]
Fix $K_0>1$ and
let $\theta\in(0,1)$ be as in Lemma \ref{26-lem},
which is defined by $(1-\theta)/\theta=\sqrt{2}$.
Then since $b_K(\zeta)$ is positive on $\R_+^n$ with $b_K(0)=1$,
there exists $c=c(K_0)>0$ such that
% for $K\in(0,K_0)$
\[
 \sup_{K\in(0,K_0)}\sup_{\zeta\in(0,3/\theta)}b_K(\zeta)^{-1}\leq c.
\]
Therefore
from the first statement of Lemma \ref{26-lem},
we obtain
\[
 \int_0^{3/\theta}
 \gamma_K(y_n,\zeta,s)^2\left( \frac{e^{\zeta^2/4}}{b_K(\zeta)^2} \right)d\zeta
\leq
 \frac{c}{(1-e^{-s})}.
\]
Next
we estimate the rest of integral over $(3/\theta,\infty)$.
Here
we note that
\[
 |y_ne^{-s/2}-\theta\zeta| = \theta\zeta-y_ne^{-s/2} \geq \theta\zeta-\frac{\theta\zeta}{3}
= \frac{2}{3}\theta\zeta
\]
for $y_n\in(0,e^{s/2})$ and $\zeta\in(3/\theta,\infty)$.
Then
we apply the second statement of Lemma \ref{26-lem} to obtain
\begin{equation*}
\begin{array}{lll}
\dis
 \int_{3/\theta}^{\infty}
 \gamma_K(y_n,\zeta,s)^2
 \left( \frac{e^{|\zeta|^2/4}}{b_K(\zeta)^2} \right)
 d\zeta
\hspace{-2mm}&=&\hspace{-2mm} \dis
 \frac{c}{(1-e^{-s})^{1/2}}
 \int_{3/\theta}^{\infty}\gamma_K(y_n,\zeta,s)
\\[2mm]
\hspace{-2mm}&&\hspace{6mm} \dis
 \times
 \exp\left( -\frac{(\theta\zeta)^2}{9(1-e^{-s})} \right)
 \left( \frac{e^{\zeta^2/4}}{b_K(\zeta)^2} \right)
 d\zeta
\\[5mm]
\hspace{-2mm}&&\hspace{-50mm} \dis
\leq
 \frac{c}{(1-e^{-s})^{1/2}}
 \int_0^{\infty}\gamma_K(y_n,\zeta,s)
 \left(
 e^{-(\theta\zeta)^2/9}
 \left( \frac{e^{\zeta^2/4}}{b_K(\zeta)^2} \right)
 \right)
 d\zeta
\end{array}
\end{equation*}
for $y_n\in(0,e^{s/2})$ and $s\geq1$.
Here applying Lemma \ref{28-lem} with
$w_0(\zeta)=e^{-(\theta\zeta)^2/9}e^{\zeta^2/4}/b_K(\zeta)^2$
and combining the above estimates,
we obtain the desired estimate.
\end{proof}
%%%%%%%%%%%%%%%%%%%%%%%%%%%%%%%%%%%%%%%%%%%%%%%%%%%%%%%%%%%

%%%%%%%%%%%%%%%%%%%%%%%%%%%%%%%%%%%%%%%%%%%%%%%%%%%%%%%%%%%
\begin{lem}\label{28-lem}
For any $M,R,K_0>1$
there exists $c=c(M,R,K_0)>0$ such that
if
\[
 \int_0^{\infty}|w_0(\zeta)|b_K(\zeta)^2e^{-\zeta^2/4}d\zeta\leq M,
\]
then it holds that
\[
 \int_0^{\infty}\gamma_K(z,\zeta,s)|w_0(\zeta)|d\zeta \leq c
\hspace{5mm}\mathrm{for}\ z\in(0,R),\ s\geq1.
\]
\end{lem}
%%%%%%%%%%%%%%%%%%%%%%%%%%%%%%%%%%%%%%%%%%%%%%%%%%%%%%%%%%%

%%%%%%%%%%%%%%%%%%%%%%%%%%%%%%%%%%%%%%%%%%%%%%%%%%%%%%%%%%%
\begin{proof}
We set
\[
 \vartheta_K(z,s)
=
 \int_0^{\infty}\gamma_K(z,\zeta,s)|w_0(\zeta)|d\zeta.
\]
Then
$\vartheta_K(z,s)$ is a solution of \eqref{w2K-eq} with the initial data $|w_0(z)|$.
Multiplying \eqref{w2K-eq} by $b_K(z)^2e^{-z^2/4}$ and integrating over $(0,\infty)$,
we obtain
\[
 \pa_s
 \int_0^{\infty}\vartheta_K(z,s)b_K(z)^2e^{-z^2/4}dz = 0.
\]
Hence
by assumption and $\vartheta_K\geq0$,
it follows that
\[
 \int_0^{\infty}|\vartheta_K(z,s)|b_K(z)^2e^{-z^2/4}dz
=
 \int_0^{\infty}|w_0(z)|b_K(z)^2e^{-z^2/4}dz
\leq M.
\]
As a consequence,
for any $R>1$
there exists $c=c(R)>0$ such that
\[
 \sup_{s\in(0,\infty)}\int_0^{R+1}|\vartheta_K(z,s)|dz \leq cM.
\]
Hence by a local $L^{\infty}$-estimate for parabolic equations,
we obtain
\[
 \sup_{K\in(0,K_0)}\sup_{z\in(0,R)}|\vartheta_K(z,s)|
\leq
 c\sup_{K\in(0,K_0)}\int_{s-1/2}^{s+1/2}d\mu\int_0^{R+1}|\vartheta_K(z,\mu))|dz
\leq cM
\]
for $s\geq1$,
which completes the proof.
\end{proof}
%%%%%%%%%%%%%%%%%%%%%%%%%%%%%%%%%%%%%%%%%%%%%%%%%%%%%%%%%%%

%%%%%%%%%%%%%%%%%%%%%%%%%%%%%%%%%%%%%%%%%%%%%%%%%%%%%%%%%%%
\subsection{Representation formula}
%%%%%%%%%%%%%%%%%%%%%%%%%%%%%%%%%%%%%%%%%%%%%%%%%%%%%%%%%%%

Here
we provide a representation formula of solutions of
\begin{equation}\label{v12-eq}
\begin{cases}
\dis
 \pa_sv = \Delta v - \frac{y}{2}\cdot\nabla v - mv + g_1(y,s)
& \text{in }\R_+^n\times(0,\infty),
\\ \dis
 \pa_{\nu}v = Kv + g_2(y',s)
& \text{on }\pa\R_+^n\times(0,\infty),
\\ \dis
 v(y,0) = v_0(y)
& \text{in }\R_+^n,
\end{cases}
\end{equation}
where $K$ is a positive constant.
Let $b_K(y_n)$ and $\mu_K $ be introduced in Section \ref{Linearbackward-sec}
and set $w_K(y,s)=v(y,s)/b_K(y_n)$.
Then since $b_K(0)=1$,
we easily see that $w_K(y,s)$ solves
\begin{equation}\label{w_K-eq}
\begin{cases}
\dis
 \pa_sw_K
=
 \Delta w_K - \frac{y}{2}\cdot\nabla w_K + 2B_K(y_n)\pa_nw_K
&
\\[-2mm] \dis \hspace{50mm}
 -(m+\mu_K)w_K + \frac{g_1(y,s)}{b_K(y_n)}
& \text{in }\R_+^n\times(0,\infty),
\\[-1mm] \dis
 \pa_{\nu}w_K = g_2(y',s)
& \text{on }\pa\R_+^n\times(0,\infty),
\\ \dis
 w_K(y,0) = v_0(y)/b_K(y_n)
& \text{in }\R_+^n.
\end{cases}
\end{equation}
Define
\[
\begin{array}{c}
\dis
 P_1(y,s) =
 \int_0^se^{-(m+\mu_K)\mu}
 G_0(e^{-\mu/2}y_n,0,\mu)
 \left( \int_{\R^{n-1}}\gamma(y',\xi',\mu)g_2(\xi',s-\mu)d\xi' \right)
 d\mu,
\\[5mm] \dis
 P_2(y,s) =
 \int_0^se^{-(m+\mu_K)\mu}
 {\cal G}_K(e^{-\mu/2}y_n,0,\mu,0)
 \left( \int_{\R^{n-1}}\gamma(y',\xi',\mu)g_2(\xi',s-\mu)d\xi' \right)
 d\mu.
\end{array}
\]
Then it is known that (cf. p.\hspace{1mm}173 in \cite{Deng-F-L})
\[
 \lim_{y_n\to0}\pa_nP_1(y,s) = -g_2(y',s).
\]
On the other hand,
from \eqref{GGG-eq} and \eqref{GGG2-eq},
we see that
\[
 \lim_{y_n\to0}\pa_nP_2(y,s) = 0.
\]
Hence
by definition of $\Gamma_K(y,\xi,s,\mu)$,
we find that
\[
\begin{array}{l}
\dis\hspace{-10mm}
 \lim_{y_n\to0}
 \pa_n\left(
 \int_0^se^{-(m+\mu_K)(s-\mu)}d\mu\int_{\pa\R_+^n}
 \Gamma_K(y,\xi',s-\mu)g_2(\xi',\mu)d\xi'
 \right)
\\[5mm] \dis \dis\hspace{50mm}
=
 \lim_{y_n\to0}\pa_n( P_1(y,s)+P_2(y,s) )
=
 -g_2(y',s).
\end{array}
\]
Due to this fact,
we obtain a representation formula of a solution of \eqref{w_K-eq}.
\[
\begin{array}{l}
\dis
 w_K(y,s)
=
 e^{-(m+\mu_K)s}\int_{\R_+^n}\Gamma_K(y,\xi,s)
 \left( \frac{v_0(\xi)}{b_K(\xi_n)} \right)
 d\xi
\\[4mm] \dis \hspace{25mm}
+
 \int_0^se^{-(m+\mu_K)(s-\mu)}d\mu\int_{\R_+^n}
 \Gamma_K(y,\xi,s-\mu)
 \left( \frac{g_1(\xi,\mu)}{b_K(\xi_n)} \right)d\xi
\\[4mm] \dis \hspace{45mm}
+
 \int_0^se^{-(m+\mu_K)(s-\mu)}d\mu\int_{\pa\R_+^n}
 \Gamma_K(y,\xi',s-\mu)g_2(\xi',\mu)d\xi'.
\end{array}
\]
For simplicity of notations,
we define
\[
\begin{array}{c}
\dis
 {\cal S}_K(s)w_0
=
\int_{\R_+^n}\Gamma_K(\cdot,\xi,s)w_0(\xi)d\xi,
\\ \dis
 {\cal T}_K(s,\mu)h
=
 \int_{\pa\R_+^n}
 \Gamma_K(\cdot,\xi',s-\mu)h(\xi')d\xi'.
\end{array}
\]
Since $w_K(y,s)=v(y,s)/b_K(y_n)$,
a representation formula of a solution of \eqref{v12-eq} is given by
\begin{equation}\label{representation-eq}
\begin{array}{l}
\dis
 \left( \frac{v(s)}{b_K} \right)
=
 e^{-(m+\mu_K)s}{\cal S}_K(s)\left( \frac{v_0}{b_K} \right)
+
 \int_0^se^{-(m+\mu_K)(s-\mu)}
\\[5mm] \dis \hspace{30mm}
\times
 \left(
 {\cal S}_K(s-\mu)\left( \frac{g_1(\mu)}{b_K} \right)
 +
 {\cal T}_K(s,\mu)g_2(\mu)
 \right)
 d\mu.
\end{array}
\end{equation}

%%%%%%%%%%%%%%%%%%%%%%%%%%%%%%%%%%%%%%%%%%%%%%%%%%%%%%%%%%%%%%%%
\subsubsection{Comparison argument}
%%%%%%%%%%%%%%%%%%%%%%%%%%%%%%%%%%%%%%%%%%%%%%%%%%%%%%%%%%%%%%%%

Here
we provide a pointwise estimate of solutions of \eqref{v123-eq}
by using a representation formula \eqref{representation-eq}.
\begin{equation}\label{v123-eq}
\begin{cases}
\dis
 \pa_sv \leq \Delta v - \frac{y}{2}\cdot\nabla v - mv + g_1(y,s)
& \text{in }\R_+^n\times(0,\infty),
\\ \dis
 \pa_{\nu}v \leq K(y',s)v + g_2(y',s)
& \text{on }\pa\R_+^n\times(0,\infty),
\\ \dis
 v(y,0) \leq v_0(y)
& \text{in }\R_+^n.
\end{cases}
\end{equation}
For the case $K(y',s)\equiv K_0$ with some positive constant $K_0$,
by using a solution of \eqref{v12-eq} as a comparison function,
we obtain from \eqref{representation-eq}
\begin{equation}\label{representation2-eq}
\begin{array}{l}
\dis
 \left( \frac{v(s)}{b_{K_0}} \right)
\leq
 e^{-(m+\mu_{K_0})s}
 {\cal S}_{K_0}(s)\left( \frac{v_0}{b_{K_0}} \right)
 +
 \int_0^se^{-(m+\mu_{K_0})(s-\mu)}
\\[4mm] \dis \hspace{30mm}
\times
 \left(
 {\cal S}_{K_0}(s-\mu)\left( \frac{g_1(\mu)}{b_{K_0}} \right)
 +
 {\cal T}_{K_0}(s,\mu)g_2(\mu)
 \right)
 d\mu.
\end{array}
\end{equation}
For the case $K(y',s)\in L^{\infty}(\pa\R_+^n\times(0,\infty))$,
we use the following solution as a comparison function.
\[
\begin{cases}
\dis
 \pa_s\bar{v} = \Delta\bar{v} - \frac{y}{2}\cdot\nabla\bar{v} - m\bar{v} + |g_1(y,s)|
& \text{in }\R_+^n\times(0,\infty),
\\ \dis
 \pa_{\nu}\bar{v} = K_0\bar{v} + |g_2(y',s)|
& \text{on }\pa\R_+^n\times(0,\infty),
\\ \dis
 \bar{v}(y,0) = |v_0(y)|
& \text{in }\R_+^n,
\end{cases}
\]
where $K_0=\|K(y',s)\|_{L^{\infty}(\pa\R_+^n\times(0,\infty))}$.
By a comparison argument,
it is easily shown that $\bar{v}\geq0$.
Hence 
it follows that
\[
 \pa_{\nu}(\bar{v}-v)
=
 K(y',s)(\bar{v}-v) + (K_0-K(y'))\bar{v}
\geq
 K(y',s)(\bar{v}-v).
\]
Then
by using a comparison argument again,
we get
\[
 v(y,s) \leq \bar{v}(y,s).
\]
Therefore
we apply \eqref{representation-eq} to obtain
\begin{equation}
\label{representation3-eq}
\begin{array}{l}
\dis
 \left( \frac{v(s)}{b_{K_0}} \right)
\leq
 e^{-(m+\mu_{K_0})s}
 {\cal S}_{K_0}(s)\left( \frac{|v_0|}{b_{K_0}} \right)
+
 \int_0^se^{-(m+\mu_{K_0})(s-\mu)}
\\[4mm] \dis \hspace{30mm}
\times
 \left(
 {\cal S}_{K_0}(s-\mu)\left( \frac{|g_1(\mu)|}{b_{K_0}} \right)
 +
 {\cal T}_{K_0}(s,\mu)|g_2(\mu)|
 \right)
 d\mu.
\end{array}
\end{equation}
%
%
%%%%%%%%%%%%%%%%%%%%%%%%%%%%%%%%%%%%%%%%%%%%%%%%%%%%%%%%%%%%%%
\section{Dynamical system approach}
\label{DynamicalSystemApproach-sec}
%%%%%%%%%%%%%%%%%%%%%%%%%%%%%%%%%%%%%%%%%%%%%%%%%%%%%%%%%%%%%%
%
%
In this section,
we study the asymptotic behavior of solutions of \eqref{(B2)-eq}.
Our argument is based on the argument in \cite{Filippas-K}.
By using Lemma \ref{Compact-lem},
we slightly simplify their arguments.
Let $v(y,s)$ be a bounded solution of \eqref{(B2)-eq}
converging to zero in $L_{\rho}^2(\R_+^n)$ as $s\to\infty$.
We expand $v(y,s)$ by eigenfunctions of \eqref{eigene-eq}:
\[
 v(s) = \sum_{\alpha}a_{\alpha}(s)E_\alpha
\hspace{5mm}\text{in } L_{\rho}^2(\R_+^n),
\]
where $a_{\alpha}(s)=(v(s),E_\alpha)_{\rho}$.
We denote by $\pi_{\sf s}$, $\pi_{\sf n}$ and $\pi_{\sf u}$
projection operators onto the subspace spanned by
eigenfunctions of ($-A$) with the positive eigenvalue,
the zero eigenvalue and the negative eigenvalue
and set
\[
 v_{\sf s}=\pi_{\sf s}v,
\hspace{5mm}
 v_{\sf n}=\pi_{\sf n}v,
\hspace{5mm}
 v_{\sf u}=\pi_{\sf u}v.
\]
The subspace $\pi_{\sf s}L_{\rho}^2(\R_+^n)$ is infinite dimension
and
$\pi_{\sf u}L_{\rho}^2(\R_+^n)$, $\pi_{\sf n}L_{\rho}^2(\R_+^n)$
are finite dimension.
First
we prepare the following elementary lemma,
which is useful in this section.
We put
\[
 \Omega = \{y=(y',y_n)\in\R^{n-1}\times(0,1)\}.
\]

%%%%%%%%%%%%%%%%%%%        Lemma        %%%%%%%%%%%%%%%%%%%
\begin{lem}\label{31-lem}
There exists $c>0$ such that
\[
 \int_{\pa\R_+^n}g(y')v^2\rho\hspace{0.5mm}dy'
\leq
 c\int_{\Omega}
 g(y')\left( v^2+v|\nabla v| \right)\rho\hspace{0.5mm}dy
\hspace{5mm}\mathrm{for}\ v\in H_{\rho}^1(\R_+^n).
\]
\end{lem}
%%%%%%%%%%%%%%%%%%%%%%%%%%%%%%%%%%%%%%%%%%%%%%%%%%%%%%%%%%%

%%%%%%%%%%%%%%%%%%%        Proof        %%%%%%%%%%%%%%%%%%%
\begin{proof}
Let $\eta_1(y_n)$ be a cut off function such that
$\eta_1(y_n)=1$ if $y_n\in(0,1/2)$ and $\eta_1(y_n)=0$ if $y_n\geq1$.
Then
it is verified that
\begin{eqnarray*}
 \int_{\pa\R_+^n}g(y')v^2\rho\hspace{0.5mm}dy'
\hspace{-2mm}&=&\hspace{-2mm}
 -\int_{\R^{n-1}}g(y')\rho\hspace{0.5mm}dy'
 \int_0^1\pa_n\left( v^2\eta_1(y_n) \right)dy_n
\\
\hspace{-2mm}&\leq&\hspace{-2mm}
 c\int_0^1dy_n\int_{\R^{n-1}}
 g(y')\left( v^2+v|\nabla v| \right)\rho\hspace{0.5mm}dy'.
\end{eqnarray*}
Since $\rho(y')|_{\pa\R_+^n}=e^{-|y'|^2/4}$,
it is clear that $\rho(y)=e^{-y_n^2/4}\rho(y')$ for $y\in\R_+^n$.
This implies that $\rho(y')\leq e^{1/4}\rho(y)$ for $y_n\leq1$.
Therefore we complete the proof.
\end{proof}
%%%%%%%%%%%%%%%%%%%%%%%%%%%%%%%%%%%%%%%%%%%%%%%%%%%%%%%%%%%

Now we state a main result in this section.

%%%%%%%%%%%%%%%%%%%%%%%%%%%%%%%%%%%%%%%%%%%%%%%%%%%%%%%%%%%%%%
\begin{pro}\label{31-pro}
One of the following two cases holds.
\\[4mm]
$({\rm I})$
% It holds that
$\dis\lim_{s\to\infty}
 \left( \frac{\|v_{\sf s}(s)\|_{\rho}+\|v_{\sf u}(s)\|_{\rho}}{\|v_{\sf n}(s)\|_{\rho}} \right)=0$,
\\[4mm]
$({\rm II})$ $\|v(s)\|_{\rho}$ decays to zero exponentially.
\end{pro}
%%%%%%%%%%%%%%%%%%%%%%%%%%%%%%%%%%%%%%%%%%%%%%%%%%%%%%%%%%%%%%

\begin{proof}
Multiplying \eqref{(B2)-eq} by $v_{\sf u}$, $v_{\sf n}$ and $v_{\sf s}$
respectively,
then
we verify that
\[
\begin{array}{c}
\dis
 \frac{1}{2}\pa_s\|v_{\sf a}\|_{\rho}^2
=
 - m\|v_{\sf a}\|_{\rho}^2 + \int_{\R_+^n}v\nabla(\rho\nabla v_{\sf a})dy
+
 \int_{\pa\R_+^n}f(v)v_{\sf a}\rho\h dy'
\hspace{5mm}
 ({\sf a}\in\{\sf n,u\})
\hspace{-10mm}
\\[6mm]
\begin{array}{lll}
\dis
 \frac{1}{2}\pa_s\|v_{\sf s}\|_{\rho}^2
\hspace{-2mm}&=&\hspace{-2mm} \dis
 -m\|v_{\sf s}\|_{\rho}^2
-
 \int_{\R_+^n}\nabla v\cdot\nabla v_{\sf s}\rho\hspace{0.5mm}dy
+
 \int_{\pa\R_+^n}
 \left( qB^{q-1}vv_{\sf s}+f(v)v_{\sf s} \right)\rho\h dy'
\\
\hspace{-2mm}&=&\hspace{-2mm} \dis
 -m\|v_{\sf s}\|_{\rho}^2 -  \|\nabla v_{\sf s}\|_{\rho}^2
+
 qB^{q-1}\|v_{\sf s}\|_{L_{\rho}^2(\pa\R_+^n)}^2
+
 \int_{\pa\R_+^n}f(v)v_{\sf s}\rho\h dy'. 
\end{array}
\end{array}
\]
Hence we obtain an ODE system.
\begin{equation}\label{Kei-eq}
\begin{cases}
\dis
 \frac{1}{2}\pa_s\|v_{\sf u}\|_{\rho}^2
\geq
 \frac{1}{2}\|v_{\sf u}\|_{\rho}^2
+
 \int_{\pa\R_+^n}f(v)v_{\sf u}\rho\h dy',
\\ \dis
\frac{1}{2}\pa_s\|v_{\sf n}\|_{\rho}^2
=
 \int_{\pa\R_+^n}f(v)v_{\sf n}\rho\h dy',
\\ \dis
\frac{1}{2}\pa_s\|v_{\sf s}\|_{\rho}^2
=
  -\|\nabla v_{\sf s}\|_{\rho}^2 -m\|v_{\sf s}\|_{\rho}^2
+
 qB^{q-1}\|v_{\sf s}\|_{L_{\rho}^2(\pa\R_+^n)}^2
+
 \int_{\pa\R_+^n}f(v)v_{\sf s}\rho\h dy'.
\end{cases}
\end{equation}
By the Schwarz inequality,
we see that for ${\sf a}\in\{\sf s,n,u\}$
\[
 \int_{\pa\R_+^n}f(v)|v_{\sf a}|\rho\h dy'
\leq
 \int_{\pa\R_+^n}
 \left( \epsilon(1+|y'|^2)^{1/2}v_{\sf a}^2+\frac{1}{\epsilon}(1+|y'|^2)^{-1/2}f(v)^2 \right)
 \rho\h dy'.
\]
From Lemma \ref{31-lem} and Lemma \ref{Compact-lem},
it holds that
\begin{eqnarray*}
 \int_{\pa\R_+^n}(1+|y'|^2)^{1/2}v_{\sf a}^2\rho\h dy'
\hspace{-2mm}&\leq&\hspace{-2mm}
 c\int_{\Omega}\left( (1+|y'|^2)v_{\sf a}^2+|\nabla v_{\sf a}|^2 \right)\rho\h dy
\leq
 c\|v_{\sf a}\|_{H_{\rho}^1(\R_+^n)}^2.
\end{eqnarray*}
Here
we recall from \eqref{f(v)regularity-eq} that $|f(v)|\leq cv^2$.
Then
from Lemma \ref{31-lem} and Lemma \ref{21-lem},
we get
\begin{equation}\label{f(v)-eq}
\begin{array}{lll}
\dis
 \int_{\pa\R_+^n}(1+|y'|^2)^{-1/2}f(v)^2\rho\h dy'
\hspace{-2mm}&\leq&\hspace{-2mm} \dis
 c\int_{\pa\R_+^n}(1+|y'|^2)^{-1/2}v^4\rho\h dy'
\\
\hspace{-2mm}&\leq&\hspace{-2mm} \dis
 c\int_{\Omega}
 (1+|y'|^2)^{-1/2}\left( |v|^3|\pa_nv|+v^4 \right)
 \rho\h dy
\\[3.5mm]
\hspace{-2mm}&\leq&\hspace{-2mm} \dis
 c\int_{\Omega}(1+|y'|^2)^{-1/2}|v|^3\rho\h dy
% \\[3.5mm]
% \hspace{-2mm}&\leq&\hspace{-2mm} \dis
\leq
 cM(s)\|v\|_{\rho}^2,
\end{array}
\end{equation}
where $M(s)$ is given by
\[
 M(s)
=
 \sup_{y\in\Omega}(1+|y'|^2)^{-1/2}|v(y,s)|.
\]
As a consequence,
there exists $\epsilon_0>0$
such that for $\epsilon\in(0,\epsilon_0)$
\begin{equation}
\label{sasou-eq}
\begin{cases}
\dis
\frac{1}{2}\pa_s\|v_{\sf u}\|_{\rho}^2
\geq
 \frac{1}{4}
 \|v_{\sf u}\|_{\rho}^2 - \frac{c}{\epsilon}M(s)\|v\|_{\rho}^2,
\\[2mm] \dis
 \frac{1}{2}\left| \pa_s\|v_{\sf n}\|_{\rho}^2 \right|
\leq
 c\epsilon\|v_{\sf n}\|_{\rho}^2 + \frac{c}{\epsilon}M(s)\|v\|_{\rho}^2
\end{cases}
\end{equation}
and
\[
\frac{1}{2}\pa_s\|v_{\sf s}\|_{\rho}^2
\leq
 -(1-c\epsilon)\left(
 \|\nabla v_{\sf s}\|_{\rho}^2+m\|v_{\sf s}\|_{\rho}^2
 \right)
+
 qB^{q-1}
 \|v_{\sf s}\|_{L_{\rho}^2(\pa\R_+^n)}^2
+
 \frac{c}{\epsilon}M(s)\|v\|_{\rho}^2.
\]
Let $\lambda_*>0$ be the smallest positive eigenvalue of $(-A)$,
which is given by
\[
 \lambda_* =
 \inf_{E\in{\cal H}_{\sf s}}
 \frac{\dis\|\nabla E\|_{\rho}^2+m\|E\|_{\rho}^2-
 qB^{q-1}\|E\|_{L_{\rho}^2(\pa\R_+^n)}^2}
%  \int_{\pa\R_+^n}e^2\rho\hspace{0.5mm}dy'}
 {\|E\|_{\rho}^2},
\]
where
${\cal H}_{\sf s}=H_{\rho}^1(\R_+^n)\cap\pi_{\sf s}L_{\rho}^2(\R_+^n)$.
By continuity,
there exists $\epsilon_*>0$ such that
for $\epsilon\in(0,\epsilon_*)$
\[
 \inf_{E\in{\cal H}_{\sf s}}
 \frac{\dis(1-\epsilon)
 \left( \|\nabla E\|_{\rho}^2+m\|E\|_{\rho}^2 \right)
 -qB^{q-1}\|E\|_{L_{\rho}^2(\pa\R_+^n)}^2}
 {\|E\|_{\rho}^2}
\geq
 \frac{\lambda_*}{2}.
\]
Hence
there exists $\epsilon_1\in(0,\epsilon_0)$
such that for $\epsilon\in(0,\epsilon_1)$
\begin{equation}
\label{nakajima-eq}
\frac{1}{2}\pa_s\|v_{\sf s}\|_{\rho}^2
\leq
 -\frac{\lambda_*}{2}\|v_{\sf s}\|_{\rho}^2
+
 \frac{c}{\epsilon}M(s)\|v\|_{\rho}^2.
\end{equation}
Since
$\lim_{s\to\infty}M(s) = 0$,
applying Lemma 3.1 in \cite{Filippas-K} to \eqref{sasou-eq} and \eqref{nakajima-eq},
we obtain the conclusion.
\end{proof}
%%%%%%%%%%%%%%%%%%%%%%%%%%%%%%%%%%%%%%%%%%%%%%%%%%%%%%%%%%%

Next
we study the asymptotic behavior of $\nabla' v(y,s)$.
We set
\[
 \|\hspace{-0.3mm}|v\|\hspace{-0.3mm}|_{\rho}^2
=
 \|v\|_{\rho}^2 + \|\nabla'v\|_{\rho}^2.
\]

%%%%%%%%%%%%%%%%%%%        Proposition        %%%%%%%%%%%%%%%%%%%
\begin{pro}\label{32-pro}
One of the following two cases holds.
\\[2mm]
$({\rm I})$
$\dis
 \lim_{s\to\infty}
 \left(
 \frac{\|\hspace{-0.3mm}|v_{\sf s}\|\hspace{-0.3mm}|_{\rho}+
 \|\hspace{-0.3mm}|v_{\sf u}\|\hspace{-0.3mm}|_{\rho}}
 {\|\hspace{-0.3mm}|v_{\sf n}\|\hspace{-0.3mm}|_{\rho}}
 \right)=0$,
\\[4mm]
$({\rm II})$ $\|\hspace{-0.3mm}|v(s)\|\hspace{-0.3mm}|_{\rho}$ decays to zero exponentially,
\end{pro}
%%%%%%%%%%%%%%%%%%%%%%%%%%%%%%%%%%%%%%%%%%%%%%%%%%%%%%%%%%%

%%%%%%%%%%%%%%%%%%%        Proof        %%%%%%%%%%%%%%%%%%%
\begin{proof}
We repeat the proof of Proposition \ref{31-pro}. 
We set $V(y,s)=\pa_iv(y,s)$ ($i=1,\cdots,n-1$).
Then $V(y,s)$ solves
\[
 \begin{cases}
 \dis
 \pa_sV = \Delta V -\frac{y}{2}\cdot V - \left(m+\frac{1}{2}\right)V
 & \text{in }\R_+^n\times(s_T,\infty),
 \\ \dis
 \pa_{\nu}V = qB^{q-1}V + f'(v)V
 & \text{on }\pa\R_+^n\times(s_T,\infty).
 \end{cases}
\]
The main linear part (neglecting $f'(v)V$) is written by
\[
 \pa_s V = \tilde{A}V := \left( A-\frac{1}{2} \right)V,
\hspace{5mm} D(\tilde{A})=D(A).
\]
Let $\tilde{\pi}_{\sf s}$, $\tilde{\pi}_{\sf n}$
and $\tilde{\pi}_{\sf u}$
be projection operators  onto the subspace spanned by
eigenfunctions of ($-\tilde{A}$) with the positive eigenvalue,
the zero eigenvalue and the negative eigenvalue
and set
\[
 V_{\sf s}=\tilde{\pi}_{\sf s}V,
\hspace{5mm}
 V_{\sf n}=\tilde{\pi}_{\sf n}V,
\hspace{5mm}
 V_{\sf u}=\tilde{\pi}_{\sf u}V.
\]
Then
we obtain
\[
\begin{cases}
\dis
\frac{1}{2}\pa_s\|V_{\sf u}\|_{\rho}^2
=
 \frac{1}{2}\|V_{\sf u}\|_{\rho}^2
+
 \int_{\pa\R_+^n}f'(v)VV_{\sf u}\rho\h dy',
\\ \dis
\frac{1}{2}\pa_s\|V_{\sf n}\|_{\rho}^2
=
 \int_{\pa\R_+^n}f'(v)VV_{\sf n}\rho\h dy',
\\ \dis
\frac{1}{2}\pa_s\|V_{\sf s}\|_{\rho}^2
=
  -\|\nabla V_{\sf s}\|_{\rho}^2 -
 \left( m+\frac{1}{2} \right)\|V_{\sf s}\|_{\rho}^2
+
 qB^{q-1}\|V_{\sf s}\|_{L_{\rho}^2(\pa\R_+^n)}^2
+
 \int_{\pa\R_+^n}f'(v)VV_{\sf s}\rho\h dy'.
\end{cases}
\]
Here
we recall from \eqref{f(v)regularity-eq} that $|f'(v)|\leq c|v|$.
Therefore
by the same calculation as \eqref{f(v)-eq},
we verify that 
\begin{eqnarray*}
 \int_{\pa\R_+^n}|f'(v)VV_{\sf s}|\rho\h dy'
\hspace{-2mm}&\leq&\hspace{-2mm}
 c\int_{\pa\R_+^n}
 \left(
 \frac{1}{\epsilon}(1+|y'|^2)^{-1/2}|vV|^2
 +
 \epsilon(1+|y'|^2)^{1/2}|V_{\sf s}|^2
 \right)\rho\h dy'
\\
\hspace{-2mm}&\leq&\hspace{-2mm}
 \frac{c}{\epsilon}\int_{\pa\R_+^n}
 (1+|y'|^2)^{-1/2}|\left(  |v|^4+|V|^4 \right)\rho\h dy'
+
 c\epsilon\|V_{\sf s}\|_{H_{\rho}^1(\R_+^n)}^2 
\\
\hspace{-2mm}&\leq&\hspace{-2mm}
 \frac{c}{\epsilon}
 \left( M(s)\|v\|_{\rho}^2+M_V(s)\|V\|_{\rho}^2 \right)
+
 c\epsilon\|V_{\sf s}\|_{H_{\rho}^1(\R_+^n)}^2,
\end{eqnarray*}
where
\[
 M_V(s) = 
 \sup_{y\in\Omega}(1+|y'|^2)^{-1/2}|V(y,s)|.
\]
Hence
there exists $\epsilon_0>0$ such that for $\epsilon\in(0,\epsilon_0)$
\begin{equation}
\label{Shouko-eq}
\begin{cases}
\dis
 \frac{1}{2}\pa_s\|V_{\sf u}\|_{\rho}^2
\geq
 \frac{1}{4}\|V_{\sf u}\|_{\rho}^2
-
 \frac{c}{\epsilon}
 \left( M(s)\|v\|_{\rho}^2+M_V(s)\|V\|_{\rho}^2 \right),
\\[3mm] \dis
 \frac{1}{2}\left| \pa_s\|V_{\sf n}\|_{\rho}^2 \right|
\leq
 c\epsilon\|V_{\sf n}\|_{\rho}^2
+
 \frac{c}{\epsilon}
 \left( M(s)\|v\|_{\rho}^2+M_V(s)\|V\|_{\rho}^2 \right),
\\[3mm] \dis
 \frac{1}{2}\pa_s\|V_{\sf s}\|_{\rho}^2
\leq
 -\frac{\lambda_*}{2}\|\nabla V_{\sf s}\|_{\rho}^2
+
 \frac{c}{\epsilon}
 \left( M(s)\|v\|_{\rho}^2+M_V(s)\|V\|_{\rho}^2 \right).
\end{cases}
\end{equation}
Therefore
since $\tilde{\pi}_{\sf a}(\pa_iv) = \pa_i(\pi_{\sf a}v)$ ($i=1,\cdots,n-1$)
for ${\sf a}\in\{\sf s,n,u\}$,
by \eqref{sasou-eq}, \eqref{nakajima-eq} and \eqref{Shouko-eq},
we obtain
\[
\begin{cases}
\dis
 \frac{1}{2}\pa_s\|\hspace{-0.3mm}|v_{\sf u}\|\hspace{-0.3mm}|_{\rho}^2
\geq
 \frac{1}{4}\|\hspace{-0.3mm}|v_{\sf u}\|\hspace{-0.3mm}|_{\rho}^2
-
 \frac{c}{\epsilon}(M(s)+M_V(s))\|\hspace{-0.3mm}|v\|\hspace{-0.3mm}|_{\rho}^2,
\\[3mm] \dis
 \frac{1}{2}\left| \pa_s\|\hspace{-0.3mm}|v_{\sf n}\|\hspace{-0.3mm}|_{\rho}^2 \right|
\leq
 c\epsilon\|\hspace{-0.3mm}|v_{\sf n}\|\hspace{-0.3mm}|_{\rho}^2
+
 \frac{c}{\epsilon}(M(s)+M_V(s))\|\hspace{-0.3mm}|v\|\hspace{-0.3mm}|_{\rho}^2,
\\[3mm] \dis
 \frac{1}{2}\pa_s\|\hspace{-0.3mm}|v_{\sf s}\|\hspace{-0.3mm}|_{\rho}^2
\leq
 -\frac{\lambda_*}{2}\|\hspace{-0.3mm}|\nabla v_{\sf s}\|\hspace{-0.3mm}|_{\rho}^2
+
 \frac{c}{\epsilon}(M(s)+M_V(s))\|\hspace{-0.3mm}|v\|\hspace{-0.3mm}|_{\rho}^2.
\end{cases}
\]
Since $\lim_{s\to\infty}M_V(s)=0$,
applying Lemma 3.1 in \cite{Filippas-K},
we complete the proof.
\end{proof}
%%%%%%%%%%%%%%%%%%%%%%%%%%%%%%%%%%%%%%%%%%%%%%%%%%%%%%%%%%%

%%%%%%%%%%%%%%%%%%%        Subsection        %%%%%%%%%%%%%%%%%%%
\subsection{Case (I)}
%%%%%%%%%%%%%%%%%%%%%%%%%%%%%%%%%%%%%%%%%%%%%%%%%%%%%%%%%%%

In this subsection,
we study a precise asymptotic behavior for the case (I).

%%%%%%%%%%%%%%%%%        Definition        %%%%%%%%%%%%%%%%%%%
\begin{df}
We call a function $v(y)$ $y_n$-axial symmetric,
if the function $v(y)$ can be expressed by $v(y)=v(|y'|,y_n)$.
\end{df}
%%%%%%%%%%%%%%%%%%%%%%%%%%%%%%%%%%%%%%%%%%%%%%%%%%%%%%%%%

For the rest of this paper,
solutions are always assumed to be $y_n$-axial symmetric.
Then
the kernel of $A$ ($=\pi_{\sf n}L_\rho^2(\R_+^n)$) turns out to be one dimension
under a symmetric assumption.
In fact,
let
\begin{equation}\label{calE-eq}
 {\cal E}(y) = c(H_2(y_1)+\cdots H_2(y_{n-1}))I_1(y_n),
\end{equation}
where $c>0$ is a normalization constant such that $\|{\cal E}\|_{\rho}=1$.
Then
it holds that ker$A=$span$\{{\cal E}\}$.

%%%%%%%%%%%%%%%%%%%        Proposition        %%%%%%%%%%%%%%%%%%%
\begin{pro}\label{33-pro}
Let $v(y,s)$ be $y_n$-axial symmetric and behave as the case $({\rm I})$
in Proposition {\rm\ref{31-pro}}.
Then
it follows that
\[
 \left\|\hspace{-0.3mm}\left|v(s)+\frac{\nu_q}{s}{\cal E}\right\|\hspace{-0.3mm}\right|_{\rho} = o(s^{-1}),
\]
where $\nu_q$ is given by
\[
 \nu_q^{-1} =
 \left( \frac{q(q-1)B^{q-2}}{2} \right)
\int_{\pa\R_+^n}{\cal E}^3\rho\h dy'>0.
\]
\end{pro}
%%%%%%%%%%%%%%%%%%%%%%%%%%%%%%%%%%%%%%%%%%%%%%%%%%%%%%%%%%%

%%%%%%%%%%%%%%%%%%%        Proof        %%%%%%%%%%%%%%%%%%%
\begin{proof}[\bf Proof of Theorem \ref{1-thm}]
Theorem \ref{1-thm} follows from Proposition \ref{31-pro} and Proposition \ref{33-pro}.
\end{proof}
%%%%%%%%%%%%%%%%%%%%%%%%%%%%%%%%%%%%%%%%%%%%%%%%%%%%%%%%%%%

First
we prepare two lemmas.

%%%%%%%%%%%%%%%%%%%        Lemma        %%%%%%%%%%%%%%%%%%%
\begin{lem}\label{32-lem}
Let $v(y,s)$ be as in Proposition {\rm\ref{33-pro}}.
Then
for any $\delta>0$
there exist $c_1,c_2>0$ such that
\[
 c_1\|v_{\sf n}(s-\delta)\|_{\rho}
\leq
 \|v_{\sf n}(s)\|_{\rho}
\leq
 c_2\|v_{\sf n}(s+\delta)\|_{\rho}.
\]
\end{lem}
%%%%%%%%%%%%%%%%%%%%%%%%%%%%%%%%%%%%%%%%%%%%%%%%%%%%%%%%%%%

%%%%%%%%%%%%%%%%%%%        Proof        %%%%%%%%%%%%%%%%%%%
\begin{proof}
Since $|f(v)|\leq cv^2$ (see \eqref{f(v)regularity-eq}),
by Lemma \ref{31-lem},
the second equation in \eqref{Kei-eq} is estimated by
\[
\begin{array}{lll}
\dis
 \left| \pa_s\|v_{\sf n}\|_{\rho}^2 \right|
\hspace{-2mm}&\leq&\hspace{-2mm} \dis
 c\int_{\pa\R_+^n}v^2|v_{\sf n}|\rho\h dy'
\leq
 c\left( \int_{\pa\R_+^n}v^4\rho\h dy'+\|v_{\sf n}\|_\rho^2 \right)
\\[4mm]
\hspace{-2mm}&\leq&\hspace{-2mm} \dis
 c\left( \int_{\R_+^n}\left( |\nabla v||v|^3+v^4 \right)\rho\h dy+\|v_{\sf n}\|_{\rho}^2 \right).
\dis
\end{array}
\]
Therefore
since $v(y,s)$ and $|\nabla v(y,s)|$ are uniformly bounded (see Theorem \ref{Ua-thm}),
from Proposition \ref{31-pro},
we obtain
\[
 \left| \pa_s\|v_{\sf n}\|_{\rho}^2 \right|
\leq
 c\left( \|v\|_{\rho}^2+\|v_{\sf n}\|_{\rho}^2 \right)
\leq
 c\|v_{\sf n}\|_{\rho}^2.
\]
Then
this implies
\[
 -c \leq \pa_s\left( \log\|v_{\sf n}\|_{\rho}^2\right) \leq c.
\]
Integrating both sides,
we obtain the conclusion.
\end{proof}
%%%%%%%%%%%%%%%%%%%%%%%%%%%%%%%%%%%%%%%%%%%%%%%%%%%%%%%%%%%

%%%%%%%%%%%%%%%%%%%        Lemma        %%%%%%%%%%%%%%%%%%%
\begin{lem}\label{33-lem}
Let $v(y,s)$ be as in Proposition {\rm\ref{33-pro}}.
Then
for $r\in(1,\infty)$
there exists a positive continuous function $\nu(s)$
satisfying $\lim_{s\to\infty}\nu(s)=0$
such that
\[
 \|(v-v_{\sf n})(s)\|_{L_{\rho}^r(\pa\R_+^n)} \leq \nu(s)\|v_{\sf n}(s)\|_{\rho}.
\]
\end{lem}
%%%%%%%%%%%%%%%%%%%%%%%%%%%%%%%%%%%%%%%%%%%%%%%%%%%%%%%%%%%

%%%%%%%%%%%%%%%%%%%        Proof        %%%%%%%%%%%%%%%%%%%
\begin{proof}
We set $z(y,s)=v(y,s)-v_{\sf n}(y,s)$.
Then
$z(y,s)$ solves
\begin{equation}\label{zZz-eq}
\begin{cases}
\dis
 \pa_sz =
 \Delta z - \frac{y}{2}\cdot\nabla z - mz - X(s){\cal E}
& \text{in }\R_+^n\times(s_T,\infty),
\\[2mm] \dis
 \pa_{\nu}z = qB^{q-1}z + f(v)
& \text{on }\pa\R_+^n\times(s_T,\infty),
\end{cases}
\end{equation}
where
$X(s)$ is given by
\[
 X(s) = \int_{\pa\R_+^n}f(v(s)){\cal E}\rho\h dy'.
\]
Since $|f(v)|\leq cv^2$ (see \eqref{f(v)regularity-eq}),
we easily see that
\[
\begin{array}{c}
\dis
 \pa_{\nu}z \leq qB^{q-1}z+cv^2,
\hspace{5mm}
 \pa_{\nu}z \geq qB^{q-1}z-cv^2
\hspace{5mm}\text{on }\pa\R_+^n\times(s_T,\infty).
\end{array}
\]
Then
we apply \eqref{representation3-eq} in \eqref{zZz-eq} with $K_0:=qB^{q-1}$ to obtain
\[
\begin{array}{lll}
\dis
 |z(s)|
\hspace{-2mm}&\leq&\hspace{-2mm}\dis
 e^{-(m+\mu_{K_0})(s-s_0)}
 \left| {\cal S}_{K_0}(s-s_0)\left( \frac{z(s_0)}{b_{K_0}} \right) \right|
+
 \int_{s_0}^se^{-(m+\mu_{K_0})(s-\mu)}
\\[5mm] && \dis \hspace{0mm}
\times
 \left(
 |X(\mu)|\left| {\cal S}_{K_0}(s-\mu)\left( \frac{{\cal E}}{b_{K_0}} \right) \right|
 +
 c\left| {\cal T}_{K_0}(s,\mu)v^2 \right|
 \right)
 d\mu
\\[6mm]
\hspace{-2mm}&=:&\hspace{-2mm}\dis
 J_1+J_2+J_3
\hspace{5mm}\text{for }y\in\pa\R_+^n,\ s\geq s_0.
\end{array}
\]
To apply Lemma \ref{25-lem},
we fix $p>\max\{2,n-1\}$ and $\delta\in(0,1)$ and set
\[
 s_1 := s_0 + \max\left\{ 1,\log\left( \frac{r-1}{1-\delta} \right) \right\}.
\]
Then
from Lemma \ref{25-lem} with $p=2$,
we see that
\[
\begin{array}{lll}
\dis
 \|J_1(s_1)\|_{L_{\rho}^r(\pa\R_+^n)}
\hspace{-2mm}&\leq&\hspace{-2mm} \dis
 \frac{ce^{-(m+\mu_{K_0})(s_1-s_0)}}{(1-e^{-(s_1-s_0)})^{n/4}}\|z(s_0)\|_{\rho}
\leq
 c\|z(s_0)\|_{\rho}
\\[4mm]
\hspace{-2mm}&=&\hspace{-2mm} \dis
 c\|(v-v_{\sf n})(s_0)\|_{\rho}.
\end{array}
\]
Next
we estimate $J_2$.
From the Schwarz inequality,
we get
\[
 |X(\mu)|
\leq
 c\int_{\pa\R_+^n}v(\mu)^2|{\cal E}|\rho\h dy'
\leq
 c\left( \int_{\pa\R_+^n}(1+|y'|^2)^{-1/2}v(\mu)^4\rho\h dy' \right)^{1/2}.
\]
Then by the same calculation as \eqref{f(v)-eq},
we see that
\[
 \int_{\pa\R_+^n}(1+|y'|^2)^{-1/2}v(\mu)^4\rho\h dy'
\leq
 cM(\mu)\|v\|_{\rho}^2,
\]
where $M(s)$ is the same as in \eqref{f(v)-eq}.
Furthermore
since $A{\cal E}=0$,
we find that
${\cal S}_{K_0}(s-\mu)({\cal E}/b_{K_0})=e^{\mu_{K_0}(s-\mu)}{\cal E}/b_{K_0}$.
Therefore
since $b_{K_0}\equiv1$ on $\pa\R_+^n$,
we obtain
\[
\begin{array}{lll}
\dis
 \|J_2(s_1)\|_{L_{\rho}^r(\pa\R_+^n)}
\hspace{-2mm}&\leq&\hspace{-2mm} \dis
 c\left(
 \int_{s_0}^{s_1}M(\mu)^{1/2}\|v(\mu)\|_{\rho}d\mu
 \right)
 \left\| {\cal E} \right\|_{L_\rho^r(\pa\R_+^n)}.
\\[4mm]
\hspace{-2mm}&\leq&\hspace{-2mm} \dis
 c\left( \sup_{\mu\in(s_0,s_1)}M(\mu)^{1/2} \right)
 \left( \sup_{\mu\in(s_0,s_1)}\|v(\mu)\|_\rho \right).
\end{array}
\]
Finally
we compute $J_3$.
Since $p>2$,
by definition of $s_1$,
we easily see that
\[
 s_1-s_0 \geq \log\left( \frac{r-1}{1-\delta} \right)
>
 \log\left( \frac{r-1}{p-1-\delta} \right).
\]
Hence
we can apply Lemma \ref{25-lem} and obtain
\[
\begin{array}{l}
\dis
 \|J_3(s_1)\|_{L_{\rho}^r(\pa\R_+^n)}
\leq
 ce^{-(m+\mu_{K_0})(s_1-s_0)}
 \int_{s_0}^{s_1}
 \frac{\|v(\mu)^2\|_{L_{K_0,\rho}^p(\R_+^n)}}{(1-e^{-(s_1-\mu)})^{(n-1)/2p+1/2}}
 d\mu.
\end{array}
\]
Then
by the Schwarz inequality and $|b_{K_0}(\xi_n)|\leq K_0$ for $\xi_n\in\R_+$,
we compute
\begin{eqnarray*}
 \|v(\mu)^2\|_{L_{K_0,\rho}^p(\R_+^n)}^p
\hspace{-2mm}&=&\hspace{-2mm}
 \int_{\R_+^n}|v(\xi,\mu)|^{2p}b_{K_0}(\xi_n)^2\rho(\xi)d\xi
\\
\hspace{-2mm}&\leq&\hspace{-2mm}
 c\left(
 \int_{\R_+^n}(1+|\xi|^2)^{-1}|v(\xi,\mu)|^{4p}\rho(\xi)d\xi
 \right)^{1/2}
\\[2mm]
\hspace{-2mm}&\leq&\hspace{-2mm}
 cM_1(\mu)\|v(\mu)\|_\rho,
\end{eqnarray*}
where $M_1(\mu)=\sup_{\xi\in\R_+^n}(1+|\xi|^2)^{-1/2}|v(\xi,\mu)|^{2p-1}$.
Since
$v(y,\mu)$ is uniformly bounded and $v(y,\mu)\to0$ uniformly on any compact set in $\overline{\R_+^n}$
as $\mu\to\infty$,
we easily see that $M_1(\mu)\to0$ as $\mu\to\infty$.
Therefore
since $p>n-1$,
we obtain
\[
\begin{array}{lll}
\dis
 \|J_3(s_1)\|_{L_{\rho}^r(\pa\R_+^n)}
\hspace{-2mm}&\leq&\hspace{-2mm} \dis
 c\int_{s_0}^{s_1}
 \frac{M_1(\mu)\|v(\mu)\|_\rho}{(1-e^{-(s_1-\mu)})^{(n-1)/2p+1/2}}
 d\mu
\\[4mm]
\hspace{-2mm}&\leq&\hspace{-2mm} \dis
 c\left( \sup_{\mu\in(s_0,s_1)}M_1(\mu) \right)
 \left( \sup_{\mu\in(s_0,s_1)}\|v(\mu)\|_\rho \right).
\end{array}
\]
Put
$\nu_1(s)=M(s)^{1/2}+M_1(s)$.
Then
combining the above estimates and applying
Proposition \ref{31-pro} and Lemma \ref{32-lem},
we conclude that
\[
\begin{array}{lll}
\dis
 \|z(s_1)\|_{L_{\rho}^r(\pa\R_+^n)}
\hspace{-2mm}&\leq&\hspace{-2mm} \dis
 c\left(
 \|(v-v_{\sf n})(s_0)\|_{\rho}+
 \left( \sup_{\mu\in(s_0,s_1)}\nu_1(\mu) \right)
 \left( \sup_{\mu\in(s_0,s_1)}\|v_{\sf n}(\mu)\|_\rho \right)
 \right)
\\[6mm]
\hspace{-2mm}&\leq&\hspace{-2mm} \dis
 c\left(
 \frac{\|(v-v_{\sf n})(s_0)\|_{\rho}}{\|v_{\sf n}(s_0)\|}
 +
 \left( \sup_{\mu\in(s_0,s_1)}\nu_1(\mu) \right)
 \right)
 \|v_{\sf n}(s_1)\|_\rho.
\end{array}
\]
Since $\nu_1(\mu)\to0$ as $\mu\to0$,
the proof is completed.
\end{proof}
%%%%%%%%%%%%%%%%%%%%%%%%%%%%%%%%%%%%%%%%%%%%%%%%%%%%%%%%%%%

%%%%%%%%%%%%%%%%%%%        Proof        %%%%%%%%%%%%%%%%%%%
\begin{proof}
[Proof of Proposition {\rm\ref{33-pro}}]
Set $a_0(s)=(v(s),{\cal E})_{\rho}$.
Then
it is verified that
\[
 \dot{a}_0 = \int_{\pa\R_+^n}f(v){\cal E}\rho\h dy'.
\]
Then
since $f(v)=k_qv^2+O(v^3)$ with $k_q=q(q-1)B^{q-2}/2$ (see \eqref{f(v)regularity-eq}),
we get
\[
 \dot{a}_0
=
 k_q\int_{\pa\R_+^n}v_{\sf n}^2{\cal E}\rho\h dy'
+
 \int_{\pa\R_+^n}
 \left( k_q(v^2-v_{\sf n}^2)+O(v^3) \right){\cal E}\rho\h dy'.
\]
Since $v(y,s)$ is uniformly bounded on $\R_+^n\times(s_T,\infty)$,
the second integral on the right-hand side is estimates as follows.
\[
\begin{array}{c}
\begin{array}{lll}
 \dis
 \int_{\pa\R_+^n}|v^2-v_{\sf n}^2||{\cal E}|\rho\h dy'
\hspace{-2mm}&\leq&\hspace{-2mm} \dis
 c\int_{\pa\R_+^n}|v+v_{\sf n}||v-v_{\sf n}||{\cal E}|\rho\h dy'
\\
\hspace{-2mm}&\leq&\hspace{-2mm} \dis
 c\|v-v_{\sf n}\|_{L_{\rho}^4(\pa\R_+^n)}\|v+v_{\sf n}\|_{L_{\rho}^2(\pa\R_+^n)},
\end{array}
\\[8mm]
\begin{array}{lll}
 \dis
 \int_{\pa\R_+^n}|v^3||{\cal E}|\rho\h dy'
\hspace{-2mm}&\leq&\hspace{-2mm} \dis
 c\int_{\R_+^n}\left( |v(v^2-v_{\sf n}^2)|+|vv_{\sf n}^2| \right)
 |{\cal E}|\rho\h dy'
\\[4mm]
\hspace{-2mm}&\leq&\hspace{-2mm} \dis
 c\int_{\R_+^n}\left( |v^2-v_{\sf n}^2|+|vv_{\sf n}^2| \right)
 |{\cal E}|\rho\h dy'
\\[4mm]
\hspace{-2mm}&\leq&\hspace{-2mm} \dis
 c\left(
 \|v-v_{\sf n}\|_{L_{\rho}^4(\pa\R_+^n)}\|v+v_{\sf n}\|_{L_{\rho}^2(\pa\R_+^n)}
 +
 \|v_{\sf n}\|_{L_{\rho}^4(\pa\R_+^n)}^2\|v\|_{L_{\rho}^2(\pa\R_+^n)}
 \right).
\end{array}
\end{array}
\]
Hence
we get
\[
\begin{array}{lll}
 \left| \dot{a}_0-\nu_q^{-1}a_0^2 \right|
\hspace{-2mm}&\leq&\hspace{-2mm} \dis
 c\left(
 \frac{\|v-v_{\sf n}\|_{L_{\rho}^4(\pa\R_+^n)}\|v+v_{\sf n}\|_{L_{\rho}^2(\pa\R_+^n)}}{\|v_{\sf n}^2\|_\rho^2}
 +
 \|v\|_{L_{\rho}^2(\pa\R_+^n)}
 \right)a_0^2
=:
 c\nu(s)a_0^2,
\end{array}
\]
where $\nu_q^{-1}=k_q\int_{\pa\R_+^n}{\cal E}^3\rho dy'>0$ (see p.\h 164 in \cite{Herrero-V1}).
Then
Proposition \ref{32-pro} and Lemma \ref{33-lem} implies
\begin{equation}\label{SundayLunch-eq}
\lim_{s\to\infty}\nu(s)=0.
\end{equation}
From the above differential inequality,
we get
\[
 \left| \frac{1}{a_0(s)}-\frac{1}{a_0(s_T)} + \frac{s-s_T}{\nu_q} \right|
\leq
 c\int_{s_T}^s\nu(\tau)d\tau.
\]
Hence
it holds that
\[
 \left| \frac{1}{a_0(s)s}+ \frac{1}{\nu_q} \right|
\leq
 \frac{s_T}{\nu_qs}-\frac{1}{a_0(s_T)s} + \frac{c}{s}\int_{s_T}^s\nu(\tau)d\tau.
\]
Therefore
combining \eqref{SundayLunch-eq},
we obtain
\[
 \lim_{s\to\infty}sa_0(s) = -\nu_q.
\]
As a consequence,
since  $v_{\sf n}=a_0(s){\cal E}$,
it follows that
\[
\begin{array}{c}
\dis
 \left\| v_{\sf n}(s)+\frac{\nu_q}{s}{\cal E} \right\|_{\rho}
=
 \left( a_0(s)+\frac{\nu_q}{s} \right)\|{\cal E}\|_\rho
=
 o(s^{-1}),
\\[4mm] \dis
 \left\| \nabla'\left( v_{\sf n}(s)+\frac{\nu_q}{s}{\cal E} \right) \right\|_{\rho}
=
 \left( a_0(s)+\frac{\nu_q}{s} \right)\|\nabla'{\cal E}\|_{\rho}
=o(s^{-1}).
\end{array}
\]
Thus
by Proposition \ref{32-pro},
we conclude
\[
 \left\| v(s)+\frac{\nu_q}{s}{\cal E} \right\|_{\rho} = o(s^{-1}),
\hspace{7.5mm}
 \left\| \nabla'\left( v(s)+\frac{\nu_q}{s}{\cal E} \right) \right\|_{\rho} = o(s^{-1}),
\]
which completes the proof.
\end{proof}

%%%%%%%%%%%%%%%%%%%%%%%%%%%%%%%%%%%%%%%%%%%%%%%%%%%%%%%%%%%%%%%%
\section{Estimate for a large range}
\label{LongRange-sec}
%%%%%%%%%%%%%%%%%%%%%%%%%%%%%%%%%%%%%%%%%%%%%%%%%%%%%%%%%%%%%%%%

Throughout this section,
we assume that
$v(y,s)$ is $y_n$-axial symmetric and behaves as the case (I) in Proposition \ref{31-pro}.
Additionally,
we assume that $v(y,s)$ satisfies a monotonicity condition:
\begin{equation}\label{Naname-eq}
 y'\cdot\nabla'v(y,s)\leq0
\hspace{5mm}\text{for }(y,s)\in\R_+^n\times(s_T,\infty).
\end{equation}
In this section,
following the arguments in \cite{Herrero-V1} and \cite{Velazquez1},
we derive pointwise estimate of $v(y',s)$ along $|y'|\sim s^{1/2}$ on $\pa\R_+^n$
from the asymptotic behavior $v(y,s)\sim-\nu_qs^{-1}{\cal E}(y)$
(Proposition \ref{33-pro}) with global heat kernel estimates given in
Section \ref{Linearbackward-sec}.
For simplicity of notations,
we set
\[
 B_{\pa\R_+^n}(r) = \{y'\in\pa\R_+^n;|y'|<r\}.
\]
First
we show that
condition \eqref{Naname-eq} is assured
if the original initial data $u_0(x)$ satisfies \eqref{Naname-eq}.

%%%%%%%%%%%%%%%%%%%        Lemma        %%%%%%%%%%%%%%%%%%%
\begin{lem}\label{41-lem}
Let $u_0(x)$ be $x_n$-axial symmetric and satisfy $x'\cdot\nabla'u_0(x)\leq0$ for $x\in\R_+^n$.
Then
$v(y,s)$ satisfies \eqref{Naname-eq} for $s\in(s_T,\infty)$.
\end{lem}
%%%%%%%%%%%%%%%%%%%%%%%%%%%%%%%%%%%%%%%%%%%%%%%%%%%%%%%%%%%%%%

%%%%%%%%%%%%%%%%%%%        Proof        %%%%%%%%%%%%%%%%%%%
\begin{proof}
Since $x=(T-t)^{1/2}y$,
it is clear that
$y'\cdot\nabla'v(y,s)\leq0$ is equivalent to $x'\cdot\nabla'u(x,t)\leq0$.
From Lemma 2.1 in \cite{Harada},
it holds that $x'\cdot\nabla'u(x,t)\leq0$ if $x'\cdot\nabla'u_0(x)\leq0$.
Hence the proof is completed.
\end{proof}
%%%%%%%%%%%%%%%%%%%%%%%%%%%%%%%%%%%%%%%%%%%%%%%%%%%%%%%%%%%%%%

As a consequence of assumption \eqref{Naname-eq},
we obtain the following lemma immediately.

%%%%%%%%%%%%%%%%%%%        Lemma        %%%%%%%%%%%%%%%%%%%
\begin{lem}\label{42-lem}
There exists $c>0$ such that
\[
 \varphi(y',s) \leq B + cs^{-1}\chi_{|y'|<\sqrt{2n}}
\hspace{5mm}\mathrm{for}\ y'\in\pa\R_+^n,\ s\gg1.
\]
\end{lem}
%%%%%%%%%%%%%%%%%%%%%%%%%%%%%%%%%%%%%%%%%%%%%%%%%%%%%%%%%%%

%%%%%%%%%%%%%%%%%%%        Proof        %%%%%%%%%%%%%%%%%%%
\begin{proof}
Since ${\cal E}(y)=c(H_2(y_1)+\cdots+H_2(y_{n-1}))I_1(y_n)$ and $H_2(\xi)=c'(\xi^2-2)$ with $c,c'>0$,
${\cal E}(y)$ is explicitly expressed by
\[
 {\cal E}(y') = cc'\left( |y'|^2-2(n-1) \right)I_1(0)
\hspace{5mm}\text{for }y'\in\pa\R_+^n.
\]
Hence
it follows that ${\cal E}(y')<0$ for $y'\in\R_+^n$ with $|y'|>\sqrt{2(n-1)}$.
By assumption,
we recall that
$v(y,s)$ behaves $v(y,s)\sim -\nu_qs^{-1}{\cal E}$ as $s\to\infty$.
Therefore
we find that $v(y',s)|_{\pa\R_+^n}<0$ for $|y'|=\sqrt{2n}$ and large $s>s_T$.
As a consequence,
by using \eqref{Naname-eq},
we see that $v(y',s)|_{\pa\R_+^n}<0$ for $|y'|>\sqrt{2n}$ and large $s>s_T$.
Thus
we complete the proof.
\end{proof}
%%%%%%%%%%%%%%%%%%%%%%%%%%%%%%%%%%%%%%%%%%%%%%%%%%%%%%%%%%%

A goal of this section is
to show the following pointwise estimate along $|y'|\sim s^{1/2}$ on $\pa\R_+^n$.

%%%%%%%%%%%%%%%%%%%        Proposition        %%%%%%%%%%%%%%%%%%%
\begin{pro}\label{41-pro}
There exist $\theta\in(0,1)$ and $0<k_1<k_2<1$ such that
\[
 k_1B \leq \varphi(y',s)|_{\pa\R_+^n} \leq  k_2B
\hspace{5mm}\mathrm{for}\ |y'|=\theta\sqrt{s}.
\]
\end{pro}
%%%%%%%%%%%%%%%%%%%%%%%%%%%%%%%%%%%%%%%%%%%%%%%%%%%%%%%%%%%

%%%%%%%%%%%%%%%%%%%        Lemma        %%%%%%%%%%%%%%%%%%%
\begin{lem}\label{43-lem}
There exists a positive continuous function $\nu(s)$ satisfying
$\lim_{s\to\infty}\nu(s)=0$ such that
\[
 v(y',s)|_{\pa\R_+^n} \geq -\frac{\nu_q}{s}{\cal E}(y') - \nu(s)
\hspace{5mm}\mathrm{for}\ y'\in B_{\pa\R_+^n}(\sqrt{s}).
\]
\end{lem}
%%%%%%%%%%%%%%%%%%%%%%%%%%%%%%%%%%%%%%%%%%%%%%%%%%%%%%%%%%%%%%

%%%%%%%%%%%%%%%%%%%        Proof        %%%%%%%%%%%%%%%%%%%
\begin{proof}
We set $K=qB^{q-1}$ and ${\cal D}(y)={\cal E}(y)/b_K(y_n)$.
Then
since $f(v)\geq0$ in \eqref{(B2)-eq},
by using \eqref{representation2-eq},
we get
\begin{eqnarray*}
 v(s)
\hspace{-2mm}&\geq&\hspace{-2mm}
 e^{-(m+\mu_K)(s-s_0)}{\cal S}_K(s-s_0)\left( \frac{v(s_0)}{b_K} \right)
\\[1mm]
\hspace{-2mm}&=&\hspace{-2mm}
 e^{-(m+\mu_{K})(s-s_0)}
 {\cal S}_K(s-s_0)\left( -\frac{\nu_q}{s_0}{\cal D} \right)
+
 e^{-(m+\mu_K)(s-s_0)}
\\
&& \hspace{0mm}
\times
 {\cal S}_K(s-s_0)\left( \frac{v(s_0)+\nu_qs_0^{-1}{\cal E}}{b_K} \right)
\hspace{5mm}\text{for }y'\in\pa\R_+^n,\ s>s_0.
\end{eqnarray*}
Here
since
\begin{equation*}
\begin{cases}
\dis
 -\left(
 \Delta-\frac{y}{2}\cdot\nabla+\left(\frac{2b_{K}'}{b_{K}}\right)\pa_n
 \right){\cal D}
=
 -(m+\mu_K){\cal D}
& \text{in }\R_+^n,
\\ \dis
 \pa_{\nu}{\cal D} = 0
& \text{on } \pa\R_+^n,
\end{cases}
\end{equation*}
we note that
\[
 {\cal S}_K(s-s_0){\cal D} = e^{(m+\mu_K)(s-s_0)}{\cal D}.
\]
Therefore
since ${\cal D}(y')={\cal E}(y')$ on $\pa\R_+^n$,
it follows that
\[
 e^{-(m+\mu_K)}{\cal S}_K(s-s_0){\cal D}
 \left( \frac{\nu_q}{s_0}{\cal D} \right)
=
 \frac{\nu_q}{s_0}{\cal E}
\hspace{5mm}\text{on }\pa\R_+^n.
\]
Furthermore
from Lemma \ref{27-lem},
we observe that
\[
\begin{array}{lll}
\dis
 \left|
 {\cal S}_K(s-s_0)
 \left( \frac{v(s_0)+\nu_qs_0^{-1}{\cal E}}{b_K} \right)
 \right|
\hspace{-2mm}&\leq&\hspace{-2mm} \dis
 c\exp\left( \frac{e^{-(s-s_0)}|y'|^2}{4(1+e^{-(s-s_0)})} \right)
% \\[6mm] && \hspace{-40mm}\dis
% \times
 \left\| v(s_0)+\frac{\nu_q}{s_0}{\cal E} \right\|_{\rho}
\end{array}
\]
for $y'\in\pa\R_+^n$ and $s\geq s_0+1$.
By the same way as in \cite{Herrero-V1},
we choose $s>s_0$ such that $s=e^{s-s_0}$.
This choice is equivalent to
\[
 1-\frac{s_0}{s} = \frac{\log s}{s}.
\]
Here
we recall from Lemma \ref{22-lem} that $\mu_K=-(m+1)$.
Therefore
since $e^{-(s-s_0)}|y'|^2\leq1$ for $|y'|\leq\sqrt{s}$,
Proposition \ref{33-pro} implies
\[
\begin{array}{lll}
\dis
 e^{-(m+\mu_{K})(s-s_0)}
 \left|
 {\cal S}_{K}(s-s_0)
 \left( \frac{v(s_0)+\nu_qs_0^{-1}{\cal E}}{b_{K}} \right)
 \right|=
 so(s_0^{-1}) 
=
 o(1)
\hspace{5mm}\text{for }y'\in B_{\pa\R_+^n}(\sqrt{s}).
\end{array}
\]
Thus the proof is completed.
\end{proof}
%%%%%%%%%%%%%%%%%%%%%%%%%%%%%%%%%%%%%%%%%%%%%%%%%%%%%%%%%%%

%%%%%%%%%%%%%%%%%%%        Lemma        %%%%%%%%%%%%%%%%%%%
\begin{lem}\label{44-lem}
For any $R>0$ there exists $c_R>0$ such that
\[
 |\nabla'v(y,s)|\leq c_Rs^{-1/2}
\hspace{5mm}\mathrm{for}\ |y'|\leq R\sqrt{s},\ y_n\in(0,R).
\]
\end{lem}
%%%%%%%%%%%%%%%%%%%%%%%%%%%%%%%%%%%%%%%%%%%%%%%%%%%%%%%%%%%%%%

%%%%%%%%%%%%%%%%%%%%%%%%%%%%%%%%%%%%%%%%%%%%%%%%%%%%%%%%%%%
\begin{proof}
We set $V(y,s)=\pa_iv(y,s)$ ($i=1,\cdots,n-1$),
then $V(y,s)$ solves
\[
\begin{cases}
 \dis
 \pa_sV = \Delta V - \frac{y}{2}\cdot\nabla V - \left( m + \frac{1}{2} \right)V
& \text{in }\R_+^n\times(s_T,\infty),
\\ \dis
 \pa_{\nu}V = q\varphi^{q-1}V
& \text{on }\pa\R_+^n\times(s_T,\infty).
\end{cases}
\]
Since $v(,s)$ behaves $v(y,s)\sim-\nu_qs^{-1}{\cal E}$ as $s\to\infty$,
we note from Proposition \ref{33-pro} that $|V(y',s)|\leq cs^{-1}$
for $y'\in B_{\pa\R_+^n}(\sqrt{2n})$.
Hence
by Lemma \ref{42-lem},
we observe that
\[
 q\varphi^{q-1}V
=
 qB^{q-1}V + q(\varphi^{q-1}-B^{q-1})V
\leq
 qB^{q-1}V + cs^{-2}
\hspace{5mm}\text{for }y'\in B_{\pa\R_+^n}(\sqrt{2n}).
\]
We set
\[
 K(y',s) =
 \begin{cases}
 qB^{q-1} & \text{if } |y'|\leq\sqrt{2n},
 \\
 q\varphi(y',s)^{q-1} & \text{if } |y'|\geq\sqrt{2n}.
 \end{cases}
\]
Then
we see that
\[
 \pa_{\nu}V \leq K(y',s)V + cs^{-2}
\hspace{3mm}\text{on } \pa\R_+^n.
\]
We note from Lemma \ref{42-lem} with \eqref{Naname-eq} that
\[
 |K(y',s)| \leq qB^{q-1} =: K_0.
\]
Therefore
we 
apply the estimate \eqref{representation3-eq} to $\pm V(y,s)$ and obtain
\[
 \left( \frac{|V(s)|}{b_{K_0}} \right)
\leq
 e^{(s-s_0)/2}{\cal S}_{K_0}(s-s_0)
 \left( \frac{|V(s_0)|}{b_{K_0}} \right)
+
 c\int_{s_0}^s
 e^{(s-\mu)/2}\mu^{-2}{\cal T}_{K_0}(s,\mu)d\mu.
\]
Here
we used $\mu_{K_0}=-(m+1)$ (Lemma \ref{22-lem}).
By Lemma \ref{27-lem},
for any $R>0$ there exists $c=c(R)>0$ such that
\begin{eqnarray*}
 \left( \frac{|V(y,s)|}{b_{K_0}(y_n)} \right)
\hspace{-2mm}&\leq&\hspace{-2mm}
 ce^{(s-s_0)/2}
 \left(
 \exp\left( \frac{e^{-(s-s_0)}|y'|^2}{4(1+e^{-(s-s_0)})} \right)
 \|V(s_0)\|_{\rho}
 +
 \int_{s_0}^s\frac{\mu^{-2}d\mu}{\sqrt{1-e^{-(s-\mu)}}}
 \right)
\end{eqnarray*}
for $y'\in\R^{n-1}$, $y_n\leq R$.
As in the proof of Lemma \ref{43-lem},
we choose $s>s_0$ such that $s=e^{s-s_0}$.
Then
from Proposition \ref{33-pro},
we obtain
\[
 |V(y,s)| \leq cb_{K_0}(y_n)s^{-1/2}
\hspace{5mm}\text{for } |y'|\leq R\sqrt{s},\ y_n\leq R,
\]
which completes the proof.
\end{proof}
%%%%%%%%%%%%%%%%%%%%%%%%%%%%%%%%%%%%%%%%%%%%%%%%%%%%%%%%%%%

%%%%%%%%%%%%%%%%%%%        Lemma        %%%%%%%%%%%%%%%%%%%
\begin{lem}\label{45-lem}
There exists $c>0$ such that
\[
 \sup_{y'\in B_{\pa\R_+^n}(\sqrt{s})}|\pa_s\varphi(y',s)|\leq cs^{-1}.
\]
\end{lem}
%%%%%%%%%%%%%%%%%%%%%%%%%%%%%%%%%%%%%%%%%%%%%%%%%%%%%%%%%%%%%%

We set
$Y(y,s)=\pa_sv(y,s)$,
then $Y(y,s)$ solves
\[
\begin{cases}
\dis
 \pa_sY = \Delta Y - \frac{y}{2}\cdot\nabla Y - mY
& \text{in }\R_+^n\times(s_T,\infty),
\\ \dis
 \pa_{\nu}Y = q\varphi^{q-1}Y
& \text{on }\pa\R_+^n\times(s_T,\infty).
\end{cases}
\]
We expand $Y(y,s)$ by using eigenfunctions
$\{E_\alpha\}_{\alpha\in{\cal A}}$ of
\eqref{eigene-eq}:
\[
 Y(s) = \sum_{\alpha\in{\cal A}}a_\alpha(s)E_\alpha,
\hspace{5mm}\text{in } H_{\rho}^1(\R_+^n),
\]
where $a_\alpha(s)=(v(s),E_\alpha)_{\rho}$.
Since $v\in C^1((s_T,\infty);L_{\rho}^2(\R_+^n))$,
it is verified that $a_\alpha(s)=\pa_s(v(s),E_\alpha)_{\rho}$.
Let $\lambda_{\alpha}$ be the eigenvalue corresponding to $E_\alpha$.
Then
it holds that
\begin{equation}
\label{b_alpha-eq}
 a_\alpha'
=
 -(m+\lambda_{\alpha})a_\alpha
+
 q\int_{\pa\R_+^n}
 (\varphi^{q-1}-B^{q-1})YE_\alpha\rho\h dy'.
\end{equation}
The proof of Lemma \ref{45-lem} follows immediately from
the following two lemmas.

%%%%%%%%%%%%%%%%%%%        Lemma        %%%%%%%%%%%%%%%%%%%
\begin{lem}\label{46-lem}
For any $\alpha\in{\cal A}$
there exists $c_{\alpha}>0$ such that
\[
 |a_\alpha(s)| \leq c_{\alpha}s^{-2}.
\]
\end{lem}
%%%%%%%%%%%%%%%%%%%%%%%%%%%%%%%%%%%%%%%%%%%%%%%%%%%%%%%%%%%

%%%%%%%%%%%%%%%%%%%        Proof        %%%%%%%%%%%%%%%%%%%
\begin{proof}
Since $Y(y,s)=\pa_sv(y,s)$,
a direct computation shows that
\begin{eqnarray*}
 q\int_{\pa\R_+^n}(\varphi^{q-1}-B^{q-1})YE_\alpha\rho\h dy'
\hspace{-2mm}&=&\hspace{-2mm}
 q\int_{\pa\R_+^n}((v+B)^{q-1}-B^{q-1})
 (\pa_sv)E_\alpha\rho\h dy'
\\
\hspace{-2mm}&=&\hspace{-2mm}
 \pa_s\int_{\pa\R_+^n}f(v)E_\alpha\rho\h dy'.
\\[-4mm]
\end{eqnarray*}
For simplicity,
we set
\[
 X_{\alpha} = \int_{\pa\R_+^n}f(v)E_\alpha\rho\h dy'.
\vspace{-3mm}
\]
Then
\eqref{b_alpha-eq} is written by
\begin{eqnarray*}
\dis
 \pa_s(e^{(m+\lambda_{\alpha})s}a_\alpha)
\hspace{-2mm}&=&\hspace{-2mm}
 e^{(m+\lambda_{\alpha})s}\pa_sX_{\alpha}
\\
\hspace{-2mm}&=&\hspace{-2mm}
 \pa_s\left(
 e^{(m+\lambda_{\alpha})s}X_{\alpha}
 \right)
 -
 (m+\lambda_{\alpha})e^{(m+\lambda_{\alpha})s}X_{\alpha}.
\end{eqnarray*}
% Since $v(s)\to0$ in $L_\rho^2(\R_+^n)$ as $s\to\infty$,
% it is clear that $a_\alpha(s)=(v(s),E_\alpha)_{\rho}\to0$ as $s\to\infty$.
% Therefore
% there exists a sequence $\{s_k\}_{k\in\N}$ such that
% $s_k\to\infty$ and $a_\alpha(s_k)\to0$ as $k\to\infty$.
We fix $s_1>1$.
Then
we integrate both sides over $(s,s_1)$ to obtain
\begin{equation}
\label{grape-eq}
\begin{array}{l}
\dis
 e^{(m+\lambda_{\alpha})s}(a_\alpha(s)-X_{\alpha}(s))
-
 e^{(m+\lambda_{\alpha})s_1}(a_\alpha(s_1)-X_{\alpha}(s_1))
\\[2mm] \dis \hspace{40mm}
=
 -(m+\lambda_{\alpha})
 \int_{s_1}^se^{(m+\lambda_{\alpha})\mu}X_{\alpha}(\mu)d\mu.
\end{array}
\end{equation}
Here
we recall from \eqref{f(v)regularity-eq} that $|f(v)|\leq cv^2$.
Then by the H${\rm\ddot{o}}$lder inequality,
we see that
\begin{eqnarray*}
 |X_{\alpha}|
\hspace{-2mm}&\leq&\hspace{-2mm}
 c\int_{\pa\R_+^n}v^2|E_\alpha|\rho\h dy'
\leq
 c\int_{\pa\R_+^n}\left( |vv_{\sf n}|+|v(v-v_{\sf n})| \right)|E_\alpha|\rho\h dy'
\\[-2mm]
\hspace{-2mm}&\leq&\hspace{-2mm}
 c\|v\|_{L_{\rho}^2(\pa\R_+^n)}
 \left(
 \left( \int_{\pa\R_+^n}v_{\sf n}^2|E_\alpha|^2\rho\h dy' \right)^{1/2}
 +
 \|E_\alpha\|_{L_{\rho}^4(\pa\R_+^n)}
 \|v-v_{\sf n}\|_{L_{\rho}^4(\pa\R_+^n)}
 \right)
\\
\hspace{-2mm}&\leq&\hspace{-2mm}
 c\|v\|_{L_{\rho}^2(\pa\R_+^n)}
 \left(
 \|v_{\sf n}\|_{\rho}
 +
 \|v-v_{\sf n}\|_{L_{\rho}^4(\pa\R_+^n)}
 \right).
\end{eqnarray*}
Since $\|v\|_{L_{\rho}^2(\pa\R_+^n)}\leq\|v_{\sf n}\|_{L_\rho^2(\pa\R_+^n)}+\|v-v_{\sf n}\|_{L_\rho^2(\pa\R_+^n)}$,
it follows from Lemma \ref{33-lem} and Proposition \ref{33-pro} that
\[
 |X_{\alpha}(s)| \leq c\|v_{\sf n}(s)\|_\rho^2 \leq \frac{c}{s^2}.
\]
Then
the right-hand side on \eqref{grape-eq} is estimated by
\[
\begin{array}{lll}
\dis
 \int_{s_1}^se^{(m+\lambda_\alpha)\mu}X_\alpha(\mu)d\mu
\hspace{-2mm}&\leq&\hspace{-2mm} \dis
 \left( \int_{s_1}^{(s_1+s)/2}+\int_{(s_1+s)/2}^s \right)
 e^{(m+\lambda_\alpha)\mu}X_\alpha(\mu)d\mu
\\[6mm]
\hspace{-2mm}&\leq&\hspace{-2mm} \dis
 c\left( e^{(m+\lambda_\alpha)(s+s_1)/2}+e^{(m+\lambda_\alpha)s}(s_1+s)^{-2} \right).
\end{array}
\]
Therefore
applying this estimate in \eqref{grape-eq},
we obtain the conclusion.
\end{proof}
%%%%%%%%%%%%%%%%%%%%%%%%%%%%%%%%%%%%%%%%%%%%%%%%%%%%%%%%%%%

We define ${\cal A}_{\lambda}\subset{\cal A}=\N_0^{n-1}\times\N$ by
\[
 {\cal A}_{\lambda} = \{\alpha\in{\cal A};\lambda_{\alpha}<\lambda\}.
\]

%%%%%%%%%%%%%%%%%%%        Lemma        %%%%%%%%%%%%%%%%%%%
\begin{lem}\label{47-lem}
There exists $\lambda_*>0$ such that
\[
 \left\|
 Y(s)-\sum_{\alpha\in{\cal A}_{\lambda_*}}a_\alpha(s)E_\alpha
 \right\|_{\rho}
\leq
 cs^{-2}.
\]
\end{lem}
%%%%%%%%%%%%%%%%%%%%%%%%%%%%%%%%%%%%%%%%%%%%%%%%%%%%%%%%%%%%%%

%%%%%%%%%%%%%%%%%%%        Proof        %%%%%%%%%%%%%%%%%%%
\begin{proof}
We set
$P_{\lambda}(y,s)=Y(y,s)
 -\sum_{\alpha\in{\cal A}_{\lambda}}a_\alpha(s)E_\alpha(y)$,
then $P_{\lambda}(y,s)$ is a solution of
\[
\begin{cases}
\dis
 \pa_sP_{\lambda}
=
 \Delta P_{\lambda} - \frac{y}{2}\cdot\nabla P_{\lambda} -mP_{\lambda}
 -
 \sum_{\alpha\in{\cal A}_{\lambda}}Q_{\alpha}E_\alpha
& \text{in }\R_+^n\times(s_T,\infty),
\\[5mm] \dis
 \pa_{\nu}P_{\lambda}
=
 q\varphi^{q-1}P_{\lambda} + q(\varphi^{q-1}-B^{q-1})
 \sum_{\alpha\in{\cal A}_{\lambda}}a_\alpha E_\alpha
& \text{on }\pa\R_+^n\times(s_T,\infty),
\end{cases}
\]
where $Q_{\alpha}$ is given by
\[
 Q_{\alpha}
=
 q\left(
 \int_{\pa\R_+^n}(\varphi^{q-1}-B^{q-1})YE_\alpha\rho\h dy'
 \right).
\]
Multiplying by $P_{\lambda}\rho$ and integrating over $\R_+^n$,
we get
\begin{equation}\label{Tame-eq}
\begin{array}{l}
\dis
 \frac{1}{2}\pa_s\|P_{\lambda}\|_{\rho}^2
=
 -\|\nabla P_{\lambda}\|_{\rho}^2 - m\|P_{\lambda}\|_{\rho}^2
-
 \sum_{\alpha\in{\cal A}_{\lambda}}Q_{\alpha}(P_{\lambda},E_\alpha)_{\rho}
\\ \dis \hspace{15mm}
 +q\int_{\pa\R_+^n}\varphi^{q-1}P_{\lambda}^2\rho\h dy'
+
 q\sum_{\alpha\in{\cal A}_{\lambda}}a_\alpha
 \int_{\pa\R_+^n}(\varphi^{q-1}-B^{q-1})P_{\lambda}E_\alpha\rho\h dy'.
\end{array}
\end{equation}
Since $\varphi(y,s)$ is positive and uniformly bounded,
the third term and the last term on the right-hand side are estimated by
\[
\begin{array}{c}
\begin{array}{lll}
\dis
 \frac{Q_{\alpha}}{q}
\hspace{-2mm}&=&\hspace{-2mm} \dis
 \int_{\pa\R_+^n}(\varphi^{q-1}-B^{q-1})
 \left( P_{\lambda}+\left( \sum_{\beta\in{\cal A}_{\lambda}}a_{\beta}E_{\beta} \right) \right)
 E_\alpha\rho\h dy'
\\
\hspace{-2mm}&\leq&\hspace{-2mm} \dis
 c\|E_\alpha\|_{L_{\rho}^4(\pa\R_+^n)}
 \left(
 \|v\|_{L_{\rho}^4(\pa\R_+^n)}
 \|P_{\lambda}\|_{L_{\rho}^2(\pa\R_+^n)}
 +
 \sum_{\beta\in{\cal A}_{\lambda}}a_{\beta}
 \|E_{\beta}\|_{L_{\rho}^4(\pa\R_+^n)}
 \|v\|_{L_{\rho}^2}
 \right),
\end{array}
\\[14mm] \dis
 \sum_{\alpha\in{\cal A}_{\lambda}}a_\alpha
 \int_{\pa\R_+^n}\left| \varphi^{q-1}-B^{q-1} \right|P_{\lambda}E_\alpha
 \rho\h dy'
\leq
 \sum_{\alpha\in{\cal A}_{\lambda}}a_\alpha
 \|E_\alpha\|_{L_{\rho}^4(\R_+^n)}\|v\|_{L_{\rho}^4(\R_+^n)}
 \|P_{\lambda}\|_{L_{\rho}^2(\R_+^n)}.
\end{array}
\]
Here
we apply Lemma \ref{33-lem} and Proposition \ref{33-pro} to obtain
\[
 \|v\|_{L_{\rho}^4(\pa\R_+^n)}
\leq
 \|v_{\sf n}\|_{L_{\rho}^4(\pa\R_+^n)}
 +
 \|v-v_{\sf n}\|_{L_{\rho}^4(\pa\R_+^n)}
\leq
 c\|v_{\sf n}\|_{\rho}
\leq
 cs^{-1}.
\]
Therefore
there exists $c_{\lambda}>0$ such that
\[
\begin{array}{l}
\dis
 \sum_{\alpha\in{\cal A}_{\lambda}}
 |Q_{\alpha}(P_{\lambda},E_\alpha)_{\rho}| +q\sum_{\alpha\in{\cal A}_{\lambda}}a_\alpha
 \int_{\pa\R_+^n}|\varphi^{q-1}-B^{q-1}|P_{\lambda}E_\alpha\rho\h dy'
\\ \dis \hspace{10mm}
\leq
 c_{\lambda}
 \left(
 s^{-1}\|P_{\lambda}\|_{H_{\rho}^1(\R_+^n)}^2
 +
 s^{-1}\sum_{\alpha\in{\cal A}_{\lambda}}a_\alpha^2
 \right)
\leq
 c_{\lambda}
 \left(
 s^{-1}\|P_{\lambda}\|_{H_{\rho}^1(\R_+^n)}^2
 +
 s^{-5}
 \right),
\end{array}
\]
where we used Lemma \ref{46-lem} in the last inequality.
Substituting this estimate into \eqref{Tame-eq}
and noting that $\varphi^{q-1}|_{\pa\R_+^n}\leq B^{q-1}(1+cs^{-1})$ from Lemma \ref{42-lem},
we obtain
\[
 \frac{1}{2}\pa_s\|P_{\lambda}\|_{\rho}^2
\leq
 -\left( 1-c_{\lambda}s^{-1} \right)
 \left(
 \|\nabla P_{\lambda}\|_{\rho}^2
 +
 m\|P_{\lambda}\|_{\rho}^2
 \right)
 +(1+cs^{-1})qB^{q-1}\|P_{\lambda}\|_{L_{\rho}^2(\pa\R_+^n)}^2
+
 c_{\lambda}s^{-5}.
\]
Let $\Pi_{\lambda}$ be a subspace of $H_{\rho}^1(\R_+^n)$  defined by
\[
 \Pi_{\lambda} = \{E\in H_1^{\rho}(\R_+^n);(E,E_\alpha)=0
 \text{ for any } \alpha\in{\cal A}_{\lambda}\}.
\]
Then
it holds that
\[
 \inf_{E\in \Pi_{\lambda}}
 \frac{\|\nabla E\|_{\rho}^2-qB^{q-1}\|E\|_{L_{\rho}^2(\pa\R_+^n)}^2}
 {\|E\|_{\rho}^2}
\geq \lambda.
\]
By continuity,
there exists $\epsilon_0>0$ such that
for $\epsilon\in(0,\epsilon_0)$
\[
 \inf_{E\in\Pi_{\lambda}}
 \frac{(1-\epsilon)\|\nabla E\|_{\rho}^2-
 (1+\epsilon)qB^{q-1}\|E\|_{L_{\rho}^2(\pa\R_+^n)}^2}
 {\|E\|_{\rho}^2}
\geq \lambda/2.
\]
Hence
there exist $\lambda_*>0$ and $s_1>s_T$ such that
for $\lambda\geq\lambda_*$
\[
 \frac{1}{2}\pa_s\|P_{\lambda}\|_{\rho}^2
\leq
 -\frac{m}{2}\|P_{\lambda}\|_{\rho}^2 + c_{\lambda}s^{-5}
\hspace{5mm}\text{for } s>s_1,
\]
which implies
\[
\begin{array}{lll}
\dis
 \|P_{\lambda}(s)\|_{\rho}^2
\hspace{-2mm}&\leq&\hspace{-2mm}\dis
 e^{-m(s-s_T)}\|P_{\lambda}(s_T)\|_{\rho}^2
+
 c_{\lambda}\int_{s_T}^se^{-m(s-\mu)}(1+\mu)^{-5}d\mu
\\[4mm] \dis
\hspace{-2mm}&\leq&\hspace{-2mm}\dis
 e^{-m(s-s_T)}\|P_{\lambda}(s_T)\|_{\rho}^2
+
 c_{\lambda}s^{-4}
\hspace{5mm}\text{for } s>s_1.
\end{array}
\]
Therefore
the proof is completed.
\end{proof}
%%%%%%%%%%%%%%%%%%%%%%%%%%%%%%%%%%%%%%%%%%%%%%%%%%%%%%%%%%%%%%

%%%%%%%%%%%%%%%%%%%        Proof        %%%%%%%%%%%%%%%%%%%
\begin{proof}[Proof of Lemma {\rm\ref{45-lem}}]
From Lemma \ref{46-lem} and Lemma \ref{47-lem},
we obtain
\[
 \|Y(s)\|_{\rho} \leq cs^{-2}.
\]
Repeating the argument in the proof of Lemma \ref{44-lem},
we obtain the conclusion.
\end{proof}
%%%%%%%%%%%%%%%%%%%%%%%%%%%%%%%%%%%%%%%%%%%%%%%%%%%%%%%%%%%

%%%%%%%%%%%%%%%%%%%        Lemma        %%%%%%%%%%%%%%%%%%%
\begin{lem}\label{48-lem}
For any $\theta>0$
there exist $l_1\in(0,1)$ and $s_*>0$ such that
\[
 v(y',s)|_{\pa\R_+^n} \leq - l_1B
\hspace{5mm}\mathrm{for}\ |y'|=\theta\sqrt{s},\ s\geq s_*.
\]
\end{lem}
%%%%%%%%%%%%%%%%%%%%%%%%%%%%%%%%%%%%%%%%%%%%%%%%%%%%%%%%%%%%%%

%%%%%%%%%%%%%%%%%%%        Proof        %%%%%%%%%%%%%%%%%%%
\begin{proof}
We fix $\theta=1$.
The proof for the other case $\theta\not=1$ follows from the quite same argument as $\theta=1$.
Since $v(y,s)$ behaves as the case (I),
we note from Lemma \ref{42-lem} that
\begin{equation}\label{Bach-eq}
 v(y',s)|_{\pa\R_+^n} \leq cs^{-1}\chi_{|y'|<\sqrt{2n}}
\hspace{5mm}\text{for }s\gg1.
\end{equation}
To derive a contradiction,
we suppose that
there exist sequences
$\{\tau_j\}_{j\in\N}$ and $\{y_j'\}_{k\in\N}\subset\pa\R_+^n$
with $|y_j'|=\sqrt{\tau_j}$
such that
$\tau_j\to\infty$ and
\[
 \lim_{j\to\infty}v(y_j',\tau_j) = 0.
\]
Then
it follows from \eqref{Naname-eq} and \eqref{Bach-eq} that
\begin{equation}\label{Jap-eq}
 \lim_{j\to\infty}\sup_{y'\in B_{\pa\R_+^n}(\sqrt{\tau_j})}|v(y',\tau_j)| = 0.
\end{equation}
Now we define $s_j<\tau_j$ by 
\[
 e^{\tau_j-s_j} = \tau_j.
\]
From this definition,
we find that
\[
 1 - \left( \frac{s_j}{\tau_j} \right)
=
 \left( \frac{\log\tau_j}{\tau_j} \right).
\]
Since $B_{\pa\R_+^n}(\sqrt{s})\subset B_{\pa\R_+^n}(\sqrt{\mu})$ if $s<\mu$,
we apply Lemma \ref{45-lem} to get
\begin{equation}\label{vbars_k-eq}
 |v(y',s) - v(y',\tau_j)|
\leq
 \int_s^{\tau_j}|\pa_sv(y',\mu)|d\mu
\leq
 c\log\left( \frac{\tau_j}{s_j} \right)
\end{equation}
for $y'\in B_{\pa\R_+^n}(\sqrt{s})$ and $s\in(s_j,\tau_j)$.
Now we set
\[
 \epsilon_j =
 \sup_{y'\in B_{\pa\R_+^n}(\sqrt{s}),\ s\in(s_j,\tau_j)}|v(y',s)|.
\]
Then
\eqref{Jap-eq} and \eqref{vbars_k-eq} imply
\[
 \lim_{j\to\infty}\epsilon_j=0.
\]
Furthermore
by definition of $\epsilon_j$ and \eqref{Bach-eq},
we get
% we verify that
\begin{eqnarray*}
 \pa_{\nu}v
\hspace{-2mm}&=&\hspace{-2mm}
 q\left( \int_0^1(B+\theta v)^{q-1}d\theta \right)v
\\[2mm]
\hspace{-2mm}&\leq&\hspace{-2mm}
 \left( \int_0^1\left( B+cs^{-1} \right)^{q-1}d\theta \right)v_+
-
 \left( \int_0^1\left( B-c\epsilon_j \right)^{q-1}d\theta \right)v_-
\\[2mm]
\hspace{-2mm}&\leq&\hspace{-2mm}
 (qB^{q-1}+cs^{-1})v_+
-
 qB^{q-1}(1-c\epsilon_j)v_-
\hspace{5mm}\text{for } y'\in B_{\pa\R_+^n}(\sqrt{s}),\ s\in(s_j,\tau_j).
\end{eqnarray*}
Here
we note from \eqref{Bach-eq} that $v_+(y,s)\leq cs^{-1}$.
Therefore
we obtain
\[
 \pa_{\nu}v
\leq
 qB^{q-1}(1-c\epsilon_j)v
+
 c(\epsilon_j+s^{-1})s^{-1}
+
 c\chi_{|y'|>\sqrt{s}}
\hspace{5mm}\text{for } y'\in\pa\R_+^n,\ s\in(s_j,\tau_j).
\]
We set $K_j=qB^{q-1}(1-c\epsilon_j)$
and
\[
 \mu_j:=\mu_{K_j},
\hspace{5mm}
 {\cal S}_j(s) := {\cal S}_{K_j}(s),
\hspace{5mm}
 {\cal T}_j(s,\mu) := {\cal T}_{K_j}(s,\mu),
\hspace{5mm}
 b_j(y_n) := b_{K_j}(y_n). 
\]
Then
we apply \eqref{representation2-eq} to obtain
\begin{eqnarray*}
 v(s)|_{\pa\R_+^n}
\hspace{-2mm}&\leq&\hspace{-2mm}
 e^{-(m+\mu_k)(s-s_k)}
 {\cal S}_j(s-s_j)\left( \frac{v(s_j)}{b_j} \right)
 +
 c\int_{s_j}^s
 e^{-(m+\mu_j)(s-\mu)}
\\
\hspace{-2mm}&&\hspace{0mm}
\times
 {\cal T}_j(s,\mu)
 \left( (\epsilon_j+\mu^{-1})\mu^{-1}+\chi_{\{|\xi'|>\sqrt{\mu}\}} \right)
 d\mu
\hspace{5mm}\text{for }s\in(s_j,\tau_j).
\end{eqnarray*}
Since $e^{-(\tau_j-s_j)}|y'|^2\leq1$ for $y'\in B_{\pa\R_+^n}(\sqrt{\tau_j})$,
it follows from Lemma \ref{27-lem} and Proposition \ref{33-pro} that
\[
\begin{array}{lll}
\dis
 \left|
 {\cal S}_j(\tau_j-s_j)
 \left( \frac{v(s_j)+\nu_qs_j^{-1}{\cal E}}{b_j} \right)
 \right|
\hspace{-2mm}&\leq&\hspace{-2mm} \dis
 \exp\left(
 \frac{e^{-(\tau_j-s_j)}|y'|^2}{4(1+e^{-(\tau_j-s_j)})}
 \right)
 o(s_j^{-1})
\\[5mm]
\hspace{-2mm}&=&\hspace{-2mm} \dis
 o(s_j^{-1})
\hspace{5mm}\text{for }y'\in B_{\pa\R_+^n}(\sqrt{\tau_j}).
\end{array}
\]
We set ${\cal D}_j={\cal E}/b_j$.
Then
by the same way as in the proof of Lemma \ref{43-lem},
we see that
\[
 {\cal S}_j(\tau_j-s_j){\cal D}_j = e^{(m+\mu_j)(\tau_j-s_j)}{\cal D}_j.
\]
Since $K_j=qB^{q-1}(1-c\epsilon_j)<qB^{q-1}$,
we note from Lemma \ref{22-lem} that
\[
 \mu_j > \mu|_{K=qB^{q-1}} = -(m+1).
\]
Hence
by using $e^{\tau_j-s_j}=\tau_j$ and ${\cal D}_j={\cal E}$ on $\pa\R_+^n$,
we obtain
\[
\begin{array}{l}
\dis
 e^{-(m+\mu_j)(\tau_j-s_j)}{\cal S}_j(\tau_j-s_j)\left( \frac{v(s_j)}{b_j} \right)
=
 e^{-(m+\mu_j)(\tau_j-s_j)}{\cal S}_j(\tau_j-s_j)
\\[2mm] \dis \hspace{60mm}
\times
 \left( -\frac{\nu_q}{s_j}{\cal D}_j+
 \left( \frac{v(s_j)+\nu_qs_j^{-1}{\cal E}}{b_j} \right) \right)
\\[6mm] \dis \hspace{5mm}
=
 -\frac{\nu_q}{s_j}{\cal D}_j + e^{-(m+\mu_j)(\tau_j-s_j)}o(s_j^{-1})
% \\[2mm] \dis \hspace{5mm}
=
 -\frac{\nu_q}{s_j}{\cal E} + \tau_j^{-(m+\mu_j)}o(s_j^{-1})
\\[4mm] \dis \hspace{5mm}
=
 -\frac{\nu_q}{s_j}{\cal E} + o(1)
\hspace{5mm}\text{for }y'\in B_{\pa\R_+^n}(\sqrt{\tau_j}).
\end{array}
\]
Furthermore
by Lemma \ref{26-lem},
we compute
\[
\begin{array}{l}
\dis
 \left|
 {\cal T}_j
 \left( (\epsilon_j+\mu^{-1})\mu^{-1}+\chi_{\{|\xi'|>\sqrt{\mu}\}} \right)
 \right|
\leq
 c\left(
 \frac{(\epsilon_j+\mu^{-1})\mu^{-1}}{\sqrt{1-e^{-(s-\mu)}}}
 \right.
\\[4mm] \dis \hspace{10mm}
\left.
+
 \int_{\R^{n-1}}
 \frac{\gamma(y',\xi',s-\mu)}{\sqrt{1-e^{-(s-\mu)}}}
 \chi_{\{|\xi'|>\sqrt{\mu}\}}d\xi'
 \right)
\hspace{5mm}\text{for }y'\in\pa\R_+^n.
\end{array}
\]
Here
since $\tau_j\leq9s_j/4$ for large $j\in\N$,
we observe that
\[
 |y'|
 \leq \frac{1}{2}\sqrt{\tau_j}
 \leq \frac{3}{4}\sqrt{s_j}
 \leq \frac{3}{4}|\xi'|
\hspace{5mm}\text{for }|y'|<\frac{\sqrt{\tau_j}}{2},\ |\xi'|>\sqrt{s_j}.
\]
This implies
\[
 |y'e^{-(s-\mu)/2}-\xi'| \geq |\xi'|-|y'| \geq \frac{|\xi'|^2}{4}
\hspace{5mm}\text{for }|y'|<\frac{\sqrt{\tau_j}}{2},\ |\xi'|>\sqrt{s_j}.
\]
Therefore
we get
\vspace{-2mm}
\[
\begin{array}{l}
\dis
 \int_{\R^{n-1}}
 \gamma(y',\xi',s-\mu)\chi_{\{|\xi'|>\sqrt{\mu}\}}d\xi'
\leq
 c\int_{|\xi'|\geq\sqrt{\mu}}
 \frac{\dis\exp\left( -\frac{|\xi'|^2}{64(1-e^{-(s-\mu)})} \right)}
 {(1-e^{-(s-\mu)})^{(n-1)/2}}
 d\xi'
\\[6mm] \dis \hspace{5mm}
=
 c\int_{|\xi'|\geq\sqrt{\mu/(1-e^{-(s-\mu)})}}
 e^{-|\xi'|^2/64}
 d\xi'
\leq
 c\int_{|\xi'|\geq\sqrt{\mu}}
 e^{-|\xi'|^2/64}
 d\xi'
\\[7mm] \dis \hspace{5mm}
\leq
 c\mu^{(n-3)/2}e^{-\mu/64}
\hspace{5mm}\text{for }y'\in B_{\pa\R_+^n}\left( \frac{\sqrt{\tau_j}}{2} \right),\
 \mu>s_j.
\end{array}
\]
As a consequence,
it follows that
\[
\begin{array}{l}
\dis
 \int_{s_j}^{\tau_j}
 e^{-(m+\mu_j)(\tau_j-\mu)}
 {\cal T}_j(s,\mu)
 \left( (\epsilon_j+\mu^{-1})\mu^{-1}+\chi_{\{|\xi'|>\sqrt{\mu}\}} \right)
 d\mu
\\[4mm] \dis \hspace{10mm}
\leq
 c\left(
 \left( \epsilon_j+s_j^{-1} \right)s_j^{-1}
 +
 s_j^{(n-3)/2}e^{-s_j/64}
 \right)
 \left( 1+e^{-(m+\mu_j)(\tau_j-s_j)} \right)
\\[3mm] \dis \hspace{10mm}
=
 c\left( \epsilon_j+O(s_j^{-1}) \right)
\hspace{5mm}\text{for }y'\in B_{\pa\R_+^n}\left( \frac{\sqrt{\tau_j}}{2} \right).
\end{array}
\]
Thus finally
we obtain
\[
 v(y',\tau_j) \leq -\frac{\nu_q}{s_j}{\cal E}(y') + o(1)
\hspace{5mm}\text{for }y'\in B_{\pa\R_+^n}\left( \frac{\sqrt{\tau_j}}{2} \right).
\]
Noting that ${\cal E}(y')=c(|y'|^2-2(n-1))$ on $\pa\R_+^n$ $(c>0)$,
then we find that
\[
 v(y',\tau_j)|_{\pa\R_+^n}
\leq
 -\frac{c\nu_q}{4}\left( \frac{\tau_j}{s_j} \right) + o(1)
\hspace{5mm}\text{for }|y'|=\frac{\sqrt{\tau_j}}{2}.
\]
However
since $\tau_j/s_j\to1$ as $j\to\infty$,
this contradicts \eqref{Jap-eq},
which completes the proof.
\end{proof}
%%%%%%%%%%%%%%%%%%%%%%%%%%%%%%%%%%%%%%%%%%%%%%%%%%%%%%%%%%%

%%%%%%%%%%%%%%%%%%%        Proof        %%%%%%%%%%%%%%%%%%%
\begin{proof}[\bf Proof of Proposition {\rm\ref{41-pro}}]
We recall from Lemma \ref{43-lem} that
\[
 \varphi(y',s)|_{\pa\R_+^n}\geq B -\nu_qs^{-1}{\cal E} + o(1)
\hspace{5mm}\text{for }y'\in B_{\pa\R_+^n}(\sqrt{s}).
\]
Since ${\cal E}(y')=c(|y'|^2-2(n-1))$ on $\pa\R_+^n$ ($c>0$),
there exists $\theta\in(0,1)$ such that
% and large $s>s_T$
\[
 \varphi(y',s)|_{\pa\R_+^n} \geq \frac{B}{2}
\hspace{5mm}\text{for }|y'|=\theta\sqrt{s},\ s\gg1.
\]
On the other hand,
from Lemma \ref{48-lem},
there exists $l_1\in(0,1)$ such that
\[
 \varphi(y',s)|_{\pa\R_+^n} = v(y',s)|_{\pa\R_+^n}+B \leq (1-l_1)B
\hspace{5mm}\text{for }|y'|=\theta\sqrt{s},\ s\gg1.
\]
Thus the proof is completed.
%%%%%%%%%%%%%%%%%%%%%%%%%%%%%%%%%%%%%%%%%%%%%%%%%%%%%%%%%%%%%%
\end{proof}
%%%%%%%%%%%%%%%%%%%%%%%%%%%%%%%%%%%%%%%%%%%%%%%%%%%%%%%%%%%%%%

%%%%%%%%%%%%%%%%%%%%%%%%%%%%%%%%%%%%%%%%%%%%%%%%%%%%%%%%%%%%%%
\section{Spacial singularities for a blow-up profile}
\label{Spacial-sec}
%%%%%%%%%%%%%%%%%%%%%%%%%%%%%%%%%%%%%%%%%%%%%%%%%%%%%%%%%%%%%%

Here we apply methods given in \cite{Herrero-V3,Merle-Z}
to investigate spacial singularities of blow-up profiles. 
As in Section \ref{LongRange-sec},
$v(y,s)$ is assumed to be $y_n$-axial symmetric,
behave as the case (I) in Proposition \ref{31-pro}
and satisfy \eqref{Naname-eq}.
Let $\theta\in(0,1)$ be given in Proposition \ref{41-pro} and
set $\vec{e}=(1,0,\cdots,0)$.
We introduce
\[
 v_s(x,t) = e^{-ms}u(e^{-s/2}x+\theta\sqrt{s}e^{-s/2}\vec{e},T+(t-1)e^{-s}),
\]
where
$s>s_T$ is a parameter.
Then $v_s(x,t)$ satisfies
\[
\begin{cases}
 \pa_tv = \Delta v & \text{in }\R_+^n\times(0,1),
\\ \dis
 \pa_{\nu}v = v^q & \text{on }\pa\R_+^n\times(0,1).
\end{cases}
\]
The following proposition gives a key estimate,
whose proof is given in the next subsection.

%%%%%%%%%%%%%%%%%%%        Proposition        %%%%%%%%%%%%%%%%%%%
\begin{pro}\label{51-pro}
There exist $0<c_-<c_+$, $s^*\gg1$ and $t_1\in(0,1)$ such that
\[
 c_- \leq v_s(0,t) \leq c_+
\hspace{5mm}\mathrm{for}\ s>s^*,\ t\in(t_1,1).
\]
\end{pro}
%%%%%%%%%%%%%%%%%%%%%%%%%%%%%%%%%%%%%%%%%%%%%%%%%%%%%%%%%%%

The proof of Theorem \ref{2-thm} follows directly from Proposition \ref{51-pro}.

%%%%%%%%%%%%%%%%%%%        Proof        %%%%%%%%%%%%%%%%%%%
\begin{proof}
[\bf Proof of Theorem \ref{2-thm}]
Now
we define an inverse function ${\sf s}(r)$ by
$\sqrt{\sf s}e^{-{\sf s}/2}=r/\theta$ (${\sf s}>1$).
Then from Proposition \ref{51-pro},
there exist $r_*\in(0,1)$ and $t_1\in(0,1)$ such that
\[
 c_-e^{m{\sf s}(r)}
\leq
 u(r\vec{e},T+(t-1)e^{-{\sf s}(r)})
\leq
 c_+e^{m{\sf s}(r)}
\hspace{5mm}\text{for }r\in(0,r_*),\ t\in(t_1,1).
\]
Here
we take $t=1$ to obtain
\[
 c_-e^{m{\sf s}(r)}
\leq
 u(r\vec{e},T)
\leq
 c_+e^{m{\sf s}(r)}
\hspace{5mm}\text{for }r\in(0,r_*).
\]
From definition of ${\sf s}(r)$,
we compute
\[
\begin{array}{c}
\dis
 {\sf s} = \log{\sf s} + 2|\log r| + 2\log\theta,
\\[2mm] \dis
 e^{-{\sf s}}{\sf s}
=
 e^{-{\sf s}}
 \left( \log {\sf s}+2|\log r|+2\log{\theta} \right) = (r/\theta)^2.
\end{array}
\]
Hence it follows that
\[
 e^{-{\sf s}}
=
 \frac{r^2}{2\theta^2|\log r|}
 \left( \frac{2|\log r|}{\log {\sf s}+2|\log r|+2\log\theta} \right).
\]
Since ${\sf s}(r)\leq c|\log r|$ for small $r>0$,
it is clear that
\[
 \left( \frac{2|\log r|}{\log {\sf s}+2|\log r|+2\log\theta} \right)
=
 1 + o(1).
\]
Therefore
there exists $0<c_-'<c_+'$ such that
\[
 c_-'
 \left( \frac{|\log r|}{r^2} \right)^m(1+o(1))
\leq
 u(r\vec{e},T)
\leq
 c_+'
 \left( \frac{|\log r|}{r^2} \right)^m(1+o(1)).
\]
Since $u(x,t)$ is $x_n$-axial symmetric,
we obtain the conclusion.
\end{proof}
%%%%%%%%%%%%%%%%%%%%%%%%%%%%%%%%%%%%%%%%%%%%%%%%%%%%%%%%%%%%%%

%%%%%%%%%%%%%%%%%%%        Proof        %%%%%%%%%%%%%%%%%%%
\subsection{Proof of Proposition \ref{51-pro} (upper bound)}
\label{upper-sec}
%%%%%%%%%%%%%%%%%%%%%%%%%%%%%%%%%%%%%%%%%%%%%%%%%%%%%%%%%%%%%%

We consider a rescaled solution $w_s(y,\tau)$ defined by
($y\in\R_+^n$, $\tau\in\R_+$)
\begin{eqnarray*}
\dis
 w_s(y,\tau)
\hspace{-2mm}&=&\hspace{-2mm}
 e^{-m\tau}v_s(e^{-\tau/2}y,1-e^{-\tau})
\\[1mm]
\hspace{-2mm}&=&\hspace{-2mm}
 e^{-m(\tau+s)}u(e^{-(\tau+s)/2}y+\theta \sqrt{s}e^{-s/2}\vec{e},T-e^{-(\tau+s)})
\\[1mm]
\hspace{-2mm}&=&\hspace{-2mm}
 \varphi(y+\theta\sqrt{s}e^{\tau/2}\vec{e},\tau+s).
\end{eqnarray*}
Then
$w_s(y,\tau)$ satisfies
\begin{equation}\label{ppooii-eq}
\begin{cases}
\dis
 \pa_{\tau}w_s = \Delta w_s - \frac{y}{2}\cdot\nabla w_s - mw_s
& \text{in }\R_+^n\times(0,\infty),
\\ \dis
 \pa_{\nu}w_s = w_s^q
& \text{on }\pa\R_+^n\times(0,\infty)
\end{cases}
\end{equation}
and
\[
 w_s(y,0) =  \varphi(y+\theta\sqrt{s}\vec{e},s).
\]
Moreover
since
$\varphi(y,s)$ is uniformly bounded on $\R_+^n\times(s_T,\infty)$,
it follows that
\begin{equation}\label{M_0-eq}
 M_0 :=
 \sup_{s\in(s_T,\infty)}\sup_{y\in\R_+^n,\tau\in(0,\infty)}w_s(y,\tau)
< \infty.
\end{equation}
From Proposition \ref{41-pro},
we note that $k_1B\leq w_s(0,0)\leq k_2B$ for some $k_1,k_2\in(0,1)$.
Hence
since $|\nabla\varphi(y,s)|$ is uniformly bounded (Lemma \ref{21-lem}) and $B=\varphi_0(0)$,
there exists $\delta_0>0$ such that
\begin{equation}\label{Kenchan1-eq}
 \left( \frac{k_1}{2} \right)\varphi_0(y_n)
\leq
 w_s(y_n,0)|_{y'=0}
\leq
 \left( \frac{1+k_2}{2} \right)\varphi_0(y_n)
\hspace{5mm}\text{for }y_n\in(0,\delta_0).
\end{equation}
On the other hand,
since $\varphi(y,s)\to\varphi_0(y_n)$
in $C_{\loc}(\overline{\R_+^n})$ as $s\to\infty$,
from \eqref{Naname-eq},
we obtain
\begin{equation}\label{Kenchan2-eq}
 \limsup_{s\to\infty}w_s(y_n,0)|_{y'=0}
\leq
 \lim_{s\to\infty}\varphi(y_n,s)|_{y'=0} = \varphi_0(y_n)
\end{equation}
uniformly on $y_n\in[0,R]$ for any $R>0$.
We fix a function $w_*(\xi)\in C^2(\overline{\R}_+)$ such that
\begin{equation}\label{Kenchan3-eq}
\begin{array}{cl}
\dis
 \left( \frac{1+k_2}{2} \right)\varphi_0(\xi)\leq w_*(\xi)<\varphi_0(\xi)
& \text{if } \xi\in(0,\delta_0),
\\ \dis
 w_*(\xi)=\varphi_0(\xi)
& \text{if } \xi\in(\delta_0,\infty).
\end{array}
\end{equation}
Let $\bar{w}(\xi,\tau)$ be a solution of
\begin{equation}\label{wbar-eq}
\begin{cases}
\dis
 \pa_{\tau}\bar{w} = \bar{w}_{\xi\xi} - \frac{\xi}{2}\bar{w}_\xi - m\bar{w}
& \text{in }\R_+\times(0,\infty),
\\ \dis
 \pa_{\nu}\bar{w} = \bar{w}^q
& \text{on }\xi=0,\ \tau\in(0,\infty),
\\ \dis
 \bar{w} = w_*
& \text{in }\R_+.
\end{cases}
\end{equation}

%%%%%%%%%%%%%%%%%%%%%%%%%%%%%%%%%%%%%%%%%%%%%%%%%%%%%%%%%%%
\begin{lem}\label{51-lem}
Let $\bar{w}(\xi,\tau)$ be given above.
Then
$\bar{w}(\xi,\tau)$ is uniformly bounded on $\R_+\times(0,\infty)$
and
converges to zero uniformly on $\R_+$ as $\tau\to\infty$.
\end{lem}
%%%%%%%%%%%%%%%%%%%%%%%%%%%%%%%%%%%%%%%%%%%%%%%%%%%%%%%%%%%

%%%%%%%%%%%%%%%%%%%        Proof        %%%%%%%%%%%%%%%%%%%
\begin{proof}
Let $\epsilon\in[0,1)$ and $h_\epsilon(\xi)$ be the unique solution of
\[
\begin{cases}
\dis
 h_\epsilon'' - \frac{\xi}{2}h_\epsilon' = mh_\epsilon
\hspace{5mm}\text{in } \R_+,
\\ \dis
 h_\epsilon(0) = 1,
\hspace{2mm}
 \pa_{\nu}h_\epsilon(0) = (1-\epsilon)qB^{q-1}.
\end{cases}
\]
First
we claim that $h_{\epsilon=0}(\xi)=h_0(\xi)$ has at least one zero on $\R_+$.
To derive a contradiction,
we suppose $h_0(\xi)>0$ on $\R_+$.
We set $h_*(\xi)=\varphi_0(\xi)/B$,
then $h_*(\xi)$ solves
\[
\begin{cases}
\dis
 h_*'' - \frac{\xi}{2}h_*' = mh_*
\hspace{5mm}\text{in } \R_+,
\\ \dis
 h_*(0) = 1,
\hspace{2mm}
 \pa_{\nu}h_*(0) = B^{q-1}.
\end{cases}
\]
By a boundary condition of $h_0(\xi)$ and $h_*(\xi)$,
we find that $h_0(\xi)<h_*(\xi)$ for small $\xi>0$.
Then there are two possibilities:
(i) there exists $\xi_1\in\R_+$ such that $h_0(\xi_1)=h_*(\xi_1)$ and $h_0(\xi)<h_*(\xi)$
for $\xi\in(0,\xi_1)$,
(ii) $0<h_0(\xi)<h_*(\xi)$ for $\xi\in\R_+$.
We introduce
\[
 g(\xi) = e^{-\xi^2/4}
 \left( h_0'(\xi)h_*(\xi)-h_0(\xi)h_*'(\xi) \right).
\]
Then
we easily see that
\[
\begin{array}{c}
\dis
 g'(\xi)=0 \hspace{5mm}\text{for } \xi\in\R_+,
\hspace{7.5mm}
  g(0) = -(q-1)B^{q-1}<0.
\end{array}
\]
For the case (i),
by definition of $\xi_1$,
it holds that $h_0'(\xi_1)-h_*'(\xi_1)=e^{\xi^2/4}g(\xi_1)/h_*(\xi_1)=g(0)/h_*(\xi_1)<0$.
However
this contradicts the definition of $\xi_1$.
For the case (ii),
since
$\varphi_0(\xi)$ is uniformly bounded on $\R_+$,
it is verified that
$h(\xi)$, $h_0(\xi)$ and their derivatives
are uniformly bounded on $\R_+$.
Hence
it follows that $\lim_{\xi\to\infty}g(\xi)=0$.
However
since $g(\xi)\equiv g(0)<0$,
this is a contradiction.
Therefore
the claim is proved.
We denote by $\xi_0$ the first zero of $h_0(\xi)$.
Since $h_0'(\xi_0)<0$,
by continuity,
$h_\epsilon(\xi)$ has a unique zero near $\xi=\xi_0$ for small $\epsilon\in(0,1)$,
which is denoted by $\xi_\epsilon$.
We fix $\epsilon=\epsilon_0$ small enough.
Now
we construct a super-solution by using $h_{\epsilon_0}(\xi)$.
We set
\[
 \psi_a(\xi) =
 \begin{cases}
 \varphi_0(\xi) - ah_{\epsilon_0}(\xi)
 & \text{if } \xi\in(0,\xi_0),
 \\
 \varphi_0(\xi)
 & \text{if } \xi\in(\xi_0,\infty),
 \end{cases}
\]
where $a\in(0,1)$ is a parameter.
Then
since $(B-a)^q=B^q-qB^{q-1}a+O(a^2)$,
we get
\[
\begin{array}{lll}
 \pa_{\nu}\psi_a
\hspace{-2mm}&=&\hspace{-2mm} \dis
 B^q - (1-\epsilon_0)qB^{q-1}a
\\[2mm]
\hspace{-2mm}&=&\hspace{-2mm} \dis
 (B-a)^q +\epsilon_0qB^{q-1}a + O(a^2)
\\[2mm]
\hspace{-2mm}&\geq&\hspace{-2mm} \dis
 (B-a)^q = \psi_a^q
\hspace{5mm}\text{for }\xi=0,\ 0<a\ll1.
\end{array}
\]
Since $w_*(\xi)\leq\varphi_0(\xi)$,
a comparison argument implies $\bar{w}(\xi,s)\leq\varphi_0(\xi)$.
Hence, by a strong maximum principle,
there exists $a_1\in(0,1)$ such that for $a\in(0,a_1)$
\[
 \bar{w}(\xi,1) \leq \psi_a(\xi)
\hspace{5mm}\text{for }\xi\in\R_+.
\]
Let $W_a(\xi,s)$ be a solution of \eqref{wbar-eq} with the initial data $\psi_a(\xi)$.
Then a comparison argument implies that
\[
 \bar{w}(\xi,s+1) \leq W_a(\xi,s).
\]
Now
we claim that $W_a(\xi,s)$ converges to zero uniformly on $\R_+$
as $s\to\infty$.
Since $\psi_a(\xi)$ is a super-solution,
it holds that $W_a(\xi,s)\leq\psi_b(\xi)$ for $s>0$.
By the unique solvability of solutions of \eqref{wbar-eq}
and a comparison argument,
we see that
\[
 W_a(\xi,s+s') \leq W_a(\xi,s')
\hspace{5mm}\text{for }s,s'>0.
\]
Hence
it follows that $\pa_sW_a(\xi,s)\leq0$ for $s>0$.
As a consequence,
$W_a(\xi,s)$ converges to some function $W_{\infty}(\xi)$
satisfying $0\leq W_{\infty}(\xi)<\varphi_0(\xi)$
uniformly on $\R_+$ as $s\to\infty$.
By a standard argument,
we find that
$W_{\infty}(\xi)$ is one of stationary solutions of \eqref{wbar-eq}.
Since $\varphi_0(\xi)$ is the unique bounded positive solution of \eqref{wbar-eq}
(see Lemma 3.1 \cite{Fila-Q}),
$W_{\infty}(\xi)$ must be zero,
which shows the claim.
Therefore from $\bar{w}(\xi,s+1) \leq W_a(\xi,s)$,
$\bar{w}(\xi,s)$ also converges to zero uniformly on $\R_+$ as $s\to\infty$,
which completes the proof.
\end{proof}
%%%%%%%%%%%%%%%%%%%%%%%%%%%%%%%%%%%%%%%%%%%%%%%%%%%%%%%%%%%

The function $\bar{w}(\xi,\tau)$ is naturally extended to
a function $\hat{w}(y,\tau)$ defined on $\R_+^n\times(0,\infty)$ by
\[
 \hat{w}(y,s) = \bar{w}(y_n,s).
\]
From Lemma \ref{51-lem},
it is clear that $\lim_{\tau\to\infty}\|\hat{w}(\tau)\|_{\rho}=0$. 
Let $\tau_{\epsilon}>0$ be the first time of
\[
 \|\hat{w}(\tau_{\epsilon})\|_{\rho}=\epsilon.
\]

%%%%%%%%%%%%%%%%%%%        Lemma        %%%%%%%%%%%%%%%%%%%
\begin{lem}\label{52-lem}
For any $\epsilon>0$
there exist $s_{\epsilon}>0$ such that
\[
 \|w_s(\tau_{\epsilon})\|_{\rho} \leq 2\epsilon
\hspace{5mm}\mathrm{for}\ s\geq s_{\epsilon}.
\]
\end{lem}
%%%%%%%%%%%%%%%%%%%%%%%%%%%%%%%%%%%%%%%%%%%%%%%%%%%%%%%%%%%%%%

%%%%%%%%%%%%%%%%%%%%%%%%%%%%%%%%%%%%%%%%%%%%%%%%%%%%%%%%%%%
\begin{proof}
Suppose that
there exist $\epsilon_0>0$ and a sequence $\{s_k\}_{k\in\N}$
($s_k\to\infty$)
such that
\begin{equation}
\label{contradiction-eq}
 \|w_{s_k}(\tau_{\epsilon_0})\|_{\rho}< 2\epsilon_0.
\end{equation}
Since
$w_s(y,\tau)$ is defined by
$w_s(y,\tau)=\varphi(y+{\theta}e^{\tau/2}s^{1/2}\vec{e},\tau+s)$,
we apply Lemma \ref{21-lem} to obtain
\begin{equation}\label{Bacon-eq}
 \sup_{s\in(s_T,\infty)}\sup_{(y,\tau)\in\R_+^n\times(0,\infty)}
 \left(
 \sum_{|\alpha|=0}^2|D_y^{\alpha}w_s(y,\tau)|
 \right)
< \infty.
\end{equation}
Furthermore
since $w_s(y,\tau)$ satisfies
\[
 \pa_{\tau}w_s =
 \Delta_y w_s - \frac{y}{2}\cdot\nabla_yw_s - mw_s
\hh \text{in }\R_+^n\times(0,\infty),
\]
it follows from \eqref{Bacon-eq} that
\[
 \sup_{s\in(s_T,\infty)}\sup_{(y,\tau)\in\R_+^n\times(0,\infty)}
 \left( (1+|y|^2)^{-1/2}|\pa_{\tau}w_s(y,\tau)| \right)
< \infty.
\]
Hence
there exist a limiting function
$w_{\infty}(y,\tau)\in
 C^{2,1}(\overline{\R_+^n}\times[0,\infty))\cap L^\infty(\R_+^n\times(0,\infty))$
and a subsequence
$\{w_{s_k}(y,\tau)\}_{k\in\N}$ which is denoted by the same symbol such that
\[
 w_{s_k}(y,\tau) \to w_{\infty}(y,\tau)
\]
in
$C_{\loc}(\overline{\R_+^n}\times[0,\infty))\cap
 C([0,\tau');L_{\rho}^2(\R_+^n))$
for any $\tau'>0$.
Then
from \eqref{Kenchan1-eq}, \eqref{Kenchan2-eq} and \eqref{Kenchan3-eq},
we see that
\[
 0\leq w_{\infty}(y_n,0)|_{y'=0} \leq w_*(y_n).
\]
Moreover
from Lemma \ref{44-lem},
it follows that 
$\nabla'w_{\infty}(y,0) = 0$.
Hence
we obtain
\[
 0\leq w_{\infty}(y,0) \leq w_*(y_n).
\]
As a consequence,
since $w_\infty(y,\tau)$ satisfies \eqref{ppooii-eq},
a comparison argument shows that for $\tau\geq0$ 
\[
 0\leq w_{\infty}(y,\tau) \leq \hat{w}(y,\tau).
\]
Hence
by definition of $\tau_{\epsilon_0}>0$,
it follows that $\|w_{\infty}(\tau_{\epsilon_0})\|_{\rho}\leq\epsilon_0$.
Since $w_{s_k}\to w_{\infty}$
in $C([0,\tau');L_{\rho}^2(\R_+^n))$ for any $\tau'>0$,
it follows that
$\|w_{s_k}(\tau_{\epsilon_0})\|_{\rho}<2\epsilon_0$
for large $k\in\N$.
However
this contradicts \eqref{contradiction-eq},
which completes the proof.
\end{proof}
%%%%%%%%%%%%%%%%%%%%%%%%%%%%%%%%%%%%%%%%%%%%%%%%%%%%%%%%%%%

We prepare a local $L^{\infty}$-estimate which is directly derived
from a standard linear parabolic theory.

%%%%%%%%%%%%%%%%%%%        Lemma        %%%%%%%%%%%%%%%%%%%
\begin{lem}\label{53-lem}
For any $R>0$ there exists $c_R>0$ such that
\[
 \sup_{|y|<R}w_s(y,\tau)
\leq
 c_R\sup_{\tau'\in(\tau-4R^2,\tau)}\|w_s(\tau')\|_{\rho}
\hspace{5mm}\mathrm{for}\ \tau>4R^2.
\]
\end{lem}
%%%%%%%%%%%%%%%%%%%%%%%%%%%%%%%%%%%%%%%%%%%%%%%%%%%%%%%%%%%

%%%%%%%%%%%%%%%%%%%        Proof        %%%%%%%%%%%%%%%%%%%
\begin{proof}
From \eqref{M_0-eq},
applying a local $L^{\infty}$-estimate for a linear parabolic equation to \eqref{ppooii-eq}
(see Theorem 6.17 in \cite{Lieberman}),
we obtain
\begin{eqnarray*}
 \left(
 \sup_{|y|<R,\tau-R^2<\tau'<\tau}w_s(y,\tau')
 \right)^2
\hspace{-2mm}&\leq&\hspace{-2mm}
 c_R^2\int_{\tau-4R^2}^{\tau}d\tau'\int_{|y|<2R}w_s(y,\tau')^2dy
\\
\hspace{-2mm}&\leq&\hspace{-2mm}
 c_R^2e^{R^2}\int_{\tau-4R^2}^{\tau}d\tau'\int_{|y|<2R}
 w_s(y,\tau)^2e^{-|y|^2/4}dy
\\
\hspace{-2mm}&\leq&\hspace{-2mm}
 c_R^2R^2e^{R^2}
 \left(
 \sup_{\tau'\in(\tau-4R^2,\tau)}\|w_s(\tau')\|_{\rho}
 \right)^2,
\end{eqnarray*}
which completes the proof.
\end{proof}
%%%%%%%%%%%%%%%%%%%%%%%%%%%%%%%%%%%%%%%%%%%%%%%%%%%%%%%%%%%

%%%%%%%%%%%%%%%%%%%        Lemma        %%%%%%%%%%%%%%%%%%%
\begin{lem}\label{54-lem}
For any $\delta_0>0$
there exists $\delta_1\in(0,\delta_0)$ such that
if $\|w_s(\tau_0)\|_{\rho}\leq\delta_1$ holds for some $\tau_0>0$ and $s>s_T$,
then it holds that
\[
 \|w_s(\tau)\|_{\rho} \leq \delta_0
\hspace{5mm}\mathrm{for}\ \tau\geq\tau_0.
\]
\end{lem}
%%%%%%%%%%%%%%%%%%%%%%%%%%%%%%%%%%%%%%%%%%%%%%%%%%%%%%%%%%%%%%
%
%
\begin{proof}
Multiplying \eqref{ppooii-eq} by $w_s\rho$ and integrating over $\R_+^n$,
we get
\[
 \frac{1}{2}\pa_{\tau}\|w_s\|_{\rho}^2
=
 -\|\nabla w_s\|_{\rho}^2 -m\|w_s\|_{\rho}^2
+
 \|w_s\|_{L_{\rho}^{q+1}(\pa\R_+^n)}^{q+1}.
\]
From Lemma \ref{31-lem} and Lemma \ref{Compact-lem},
we verify that
\begin{eqnarray*}
 \|w_s\|_{L_{\rho}^{q+1}(\pa\R_+^n)}^{q+1}
\hspace{-2mm}&\leq&\hspace{-2mm}
 \left( \int_{\pa\R_+^n}(1+|y'|^2)^{1/2}w_s^2\rho\h dy' \right)^{1/2}
 \left( \int_{\pa\R_+^n}(1+|y'|^2)^{-1/2}w_s^{2q}\rho\h dy' \right)^{1/2}
\\[1mm]
\hspace{-2mm}&\leq&\hspace{-2mm}
 cK_s(\tau)\left(
 \|w_s\|_{\rho}^2+\|w_s\|_{\rho}\|\nabla w_s\|_{\rho}
 \right),
\end{eqnarray*}
where $K_s(\tau)$ is given by
\vspace{-2mm}
\[
 K_s(\tau)
=
 \left(
 \sup_{y'\in\pa\R_+^n}
 (1+|y'|^2)^{-1/2}w_s(y',\tau)^{2(q-1)}
 \right)^{1/2}.
\]
Hence
it holds that
\begin{equation}\label{Check-eq}
 \frac{1}{2}\pa_{\tau}\|w_s\|_{\rho}^2
\leq
 -\|\nabla w_s\|_{\rho}^2 -m\|w_s\|_{\rho}^2
+
 cK_s(\tau)
 \left(
 \|w_s\|_{\rho}^2+\|w_s\|_{\rho}\|\nabla w_s\|_{\rho}
 \right).
\end{equation}
We note from \eqref{M_0-eq} that
$K_s(\tau)$ is uniformly bounded on $s\in(s_T,\infty)$ and $\tau\in(0,\infty)$.
Therefore
there exists $\alpha_0>0$ independent of $s>s_T$ such that
\[
 \pa_{\tau}\|w_s\|_{\rho}^2 \leq \alpha_0^2\|w_s\|_{\rho}^2,
\]
which implies
\begin{equation}\label{alpha_0-eq}
 \|w_s(\tau)\|_{\rho} \leq e^{\alpha_0(\tau-\tau_0)}\|w_s(\tau_0)\|_{\rho}
\hspace{5mm}\text{for }\tau>\tau_0.
\end{equation}
Let $\delta_0>0$ be a constant given in this lemma and $\epsilon_0>0$ be a small constant.
Then
we can fix $0<\delta_1<\delta_2<\delta_0$ and $R_1>0$ such that
\begin{equation}\label{Choice-eq}
 (1+R_1^2)^{-1/2}M_0^{2(q-1)}
+
 (c_{R_1}\delta_2)^{2(q-1)} < \epsilon_0^2,
\hspace{10mm}
 \delta_1e^{4\alpha_0R_1^2}< \delta_2,
\end{equation}
where $c_{R_1}>0$ is given in Lemma \ref{53-lem}
and $M_0>0$ is given in \eqref{M_0-eq}.
Here
we assume that
$\|w_s(\tau_0)\|_{\rho}\leq\delta_1$ for some $\tau_0>0$ and $s>s_T$.
Then
we will see that $\|w_s(\tau)\|_\rho<\delta_2$ for $\tau>\tau_0$.
In fact,
we first define
\[
 \bar{\tau} = \inf\{\tau>\tau_0;\|w_s(\tau)\|_{\rho}=\delta_2\}.
\]
To derive a contradiction,
we suppose $\bar{\tau}<\infty$.
Then
it follows from \eqref{alpha_0-eq} and \eqref{Choice-eq} that
\[
 \bar{\tau} > \tau_0 + 4R_1^2.
\]
Furthermore
we apply Lemma \ref{53-lem} with \eqref{Choice-eq} to obtain
\[
\begin{array}{lll}
\dis
 \sup_{|y|<R_1}w_s(y,\tau)
\hspace{-2mm}&\leq&\hspace{-2mm} \dis
 c_{R_1}\sup_{\tau'\in(\tau-4R_1^2,\tau)}\|w_s(\tau')\|_{\rho}
\leq
 c_{R_1}\sup_{\tau'\in(\tau_0,\bar{\tau})}\|w_s(\tau')\|_{\rho}
\\[5mm]
\hspace{-2mm}&\leq&\hspace{-2mm} \dis
 c_{R_1}\delta_2
\hspace{5mm}\text{for }\tau\in(\tau_0+4R_1^2,\bar{\tau}).
\end{array}
\]
Hence
by using this estimate and \eqref{Choice-eq},
we get
\[
\begin{array}{lll}
\dis
 K_s(\tau)^2
\hspace{-2mm}&=&\hspace{-2mm} \dis
 \sup_{y'\in\pa\R_+^n}(1+|y'|^2)^{-1/2}w_s(y',\tau)^{2(q-1)}
\\[4mm]
\hspace{-2mm}&\leq&\hspace{-2mm} \dis
 \max\left\{ (c_{R_1}\delta_2)^{2(q-1)},(1+R_1^2)^{-1/2}M_0^{2(q-1)} \right\}
\\[4mm]
\hspace{-2mm}&\leq&\hspace{-2mm} \dis
 \epsilon_0^2
\hspace{5mm}\text{for }\tau\in(\tau_0+4R_1^2,\bar{\tau}).
\end{array}
\]
Therefore
substituting this estimate into \eqref{Check-eq}
and
taking $\epsilon_0>0$ small enough,
we obtain
\[
 \pa_{\tau}\|w_s(\tau)\|_{\rho}^2 < 0
\hspace{5mm}\text{for }\tau\in(\tau_0+4R^2,\bar{\tau}),
\]
which implies
$\|w_s(\bar{\tau})\|_\rho < \|w_s(\tau_0+4R_1^2)\|_{\rho}$.
Furthermore
by \eqref{alpha_0-eq}, \eqref{Choice-eq} and $\|w_s(\tau_0)\|_\rho\leq\delta_1$,
we see that
\[
 \|w_s(\tau_0+4R_1^2)\|_{\rho} \leq e^{4\alpha_0^2R_1^2}\|w_s(\tau_0)\| < \delta_2.
\]
Therefore we obtain
$\|w_s(\bar{\tau})\|_\rho<\delta_2$.
However
this contradicts definition of $\bar{\tau}$,
which assures $\bar{\tau}=\infty$.
Thus the proof is completed.
\end{proof}
%%%%%%%%%%%%%%%%%%%%%%%%%%%%%%%%%%%%%%%%%%%%%%%%%%%%%%%%%%%

Next
we provide uniform decay estimates.

%%%%%%%%%%%%%%%%%%%        Lemma        %%%%%%%%%%%%%%%%%%%
\begin{lem}\label{55-lem}
For any $\nu>0$
there exist
$c,s_1^*,\tau_1>0$ depending only on $\nu>0$
such that for $s\geq s_1^*$
\[
 \|w_s(\tau)\|_{\rho} \leq ce^{-(m-\nu)(\tau-\tau_1)}
\hspace{5mm}\mathrm{for}\ \tau\geq\tau_1.
\]
\end{lem}
%%%%%%%%%%%%%%%%%%%%%%%%%%%%%%%%%%%%%%%%%%%%%%%%%%%%%%%%%%%

%%%%%%%%%%%%%%%%%%%        Proof        %%%%%%%%%%%%%%%%%%%
\begin{proof}
From
Lemma \ref{52-lem}\hspace{0.2mm}--\hspace{0.2mm}Lemma \ref{54-lem},
for any $\nu\in(0,1)$
there exists $s_1^*,\tau_1>0$ depending only on $\nu>0$
such that
\[
 \sup_{s>s_1^*,\ \tau>\tau_1}K_s(\tau)\leq \nu.
\]
Hence
substituting this estimate into \eqref{Check-eq},
we get for $s>s_1^*$
\[
 \pa_{\tau}\|w_s(\tau)\|_{\rho}^2
\leq
 -2(m-c\nu)\|w_s(\tau)\|_{\rho}^2
\hspace{5mm}\text{for }\tau\geq\tau_1.
\]
As a consequence,
it holds that for $s\geq s_1^*$
\[
  \|w_s(\tau)\|_{\rho}^2
\leq
 e^{-2(m-c\nu)(\tau-\tau_1)}\|w_s(\tau_1)\|_{\rho}^2
\hspace{5mm}\text{for }\tau\geq\tau_1.
\]
Here
we note from \eqref{M_0-eq} that
$\|w_s(\tau)\|_\rho$ is uniformly bounded on $s>s_T$ and $\tau>0$.
Thus we completes the proof.
\end{proof}
%%%%%%%%%%%%%%%%%%%%%%%%%%%%%%%%%%%%%%%%%%%%%%%%%%%%%%%%%%%

Let ${\sf E}_0$ be a positive constant such that $\|{\sf E}_0\|_{\rho}=1$.
Here
we decompose $w_s(y,\tau)$ by
\[
 w_s(\tau) = w_{s0}(\tau) + w_{s-}(\tau),
\]
where $w_{s0}(\tau)=(w_s(\tau),{\sf E}_0)_{\rho}{\sf E}_0$.

%%%%%%%%%%%%%%%%%%%        Lemma        %%%%%%%%%%%%%%%%%%%
\begin{lem}\label{56-lem}
There exists $c,\nu_0,s_2^*,\tau_2>0$
such that for $s\geq s_2^*$
\[
 \|w_{s-}(\tau)\|_{\rho} \leq ce^{-(1+\nu_0)m\tau},
\hspace{7.5mm}
 \int_{\tau_1}^{\infty}
 e^{2(1+\nu_0)m\tau}\|\nabla w_{s-}(\tau)\|_{\rho}^2d\tau \leq c
\hspace{5mm}\mathrm{for}\ \tau\geq\tau_2.
\]
\end{lem}
%%%%%%%%%%%%%%%%%%%%%%%%%%%%%%%%%%%%%%%%%%%%%%%%%%%%%%%%%%%

%%%%%%%%%%%%%%%%%%%        Proof        %%%%%%%%%%%%%%%%%%%
\begin{proof}
Since $(w_{s-},w_{s0})_\rho=0$,
we easily see that
\[
 \frac{1}{2}\pa_{\tau}\|w_{s-}\|_{\rho}^2
=
 -\|\nabla w_{s-}\|_{\rho}^2 - m\|w_{s-}\|_{\rho}^2
+
 \int_{\pa\R_+^n}w_s^qw_{s-}\rho\h dy'.
\]
Now we estimate the last term on the right-hand side.
\[
 \int_{\pa\R_+^n}w_s^q|w_{s-}|\rho\h dy'
=
 \int_{\pa\R_+^n}w_{s0}^q|w_{s-}|\rho\h dy'
+
 \int_{\pa\R_+^n}
 \left( w_s^q-w_{s0}^q \right)|w_{s-}|
 \rho\h dy'.
\]
Then by Lemma \ref{31-lem},
it holds that
\[
\begin{array}{lll}
\dis
 \int_{\pa\R_+^n}w_{s0}^q|w_{s-}|\rho\h dy'
\hspace{-2mm}&\leq&\hspace{-2mm} \dis
 \epsilon\|w_{s-}\|_{L_{\rho}^2(\pa\R_+^n)}^2
+
 \frac{c}{\epsilon}\|w_{s0}\|_{\rho}^{2q}
\\ \dis
\hspace{-2mm}&\leq&\hspace{-2mm} \dis
 c\epsilon\|w_{s-}\|_{H_{\rho}^1(\R_+^n)}^2
+
 \frac{c}{\epsilon}\|w_{s0}\|_{\rho}^{2q}.
\end{array}
\]
Furthermore
the mean value theorem implies
\begin{eqnarray*}
 \int_{\pa\R_+^n}
 \left( w_s^q-w_{s0}^q \right)|w_{s-}|
 \rho\h dy'
\hspace{-2mm}&\leq&\hspace{-2mm}
 q\int_{\pa\R_+^n}(w_s+|w_{s0}|)^{q-1}|w_{s-}|^2\rho\h dy'
\\
\hspace{-2mm}&\leq&\hspace{-2mm}
 qM_s(\tau)
 \int_{\pa\R_+^n}(1+|y'|^2)^{1/2}|w_{s-}|^2\rho\h dy',
\end{eqnarray*}
where $M_s(\tau)$ is given by
\[
 M_s(\tau)
=
 \sup_{y'\in\pa\R_+^n}
 (1+|y'|^2)^{-1/2}(w_s(y',\tau)+|w_{s0}(y',\tau)|)^{q-1}.
\]
Hence
from Lemma \ref{31-lem} and Lemma \ref{Compact-lem},
we get
\[
 \int_{\pa\R_+^n}
 \left( w_s^q-w_{s0}^q \right)|w_{s-}|
 \rho\h dy'
\leq cM_s(\tau)\|w_{s-}\|_{H_\rho^1(\R_+^n)}^2.
\]
Therefore
since $\|w_{s0}\|_\rho\leq\|w_s\|_\rho$,
we obtain
\[
 \frac{1}{2}\pa_{\tau}\|w_{s-}\|_{\rho}^2
\leq
 -(1-c\epsilon-cM_s(\tau))
 \left( \|\nabla w_{s-}\|_{\rho}^2 + m\|w_{s-}\|_{\rho}^2 \right)
+
 \frac{c}{\epsilon}\|w_s\|_{\rho}^{2q}.
\]
Since $q>1$,
from  Lemma \ref{55-lem},
there exists $s_1^*,\tau_1,\nu_1>0$ such that for $s>s_1^*$
\[
 \|w_s(\tau)\|_{\rho}^{2q} \leq ce^{-2(m+\nu_1)(\tau-\tau_1)}
\hspace{5mm}\text{for }\tau>\tau_1.
\]
Moreover
by the same estimate as $K_s(\tau)$ in the proof of Lemma \ref{54-lem},
we find that
\[
 \lim_{s,\tau\to\infty}M_s(\tau)=0.
\]
Here
we note that
${\sf E}_0$ is the first eigenfunction with zero eigenvalue of
\begin{equation}
\label{eigenvaluehat-eq}
 -\left( \Delta-\frac{y}{2}\cdot\nabla \right){\sf E}
=
 \lambda{\sf E}
\hspace{3mm}\text{in } \R_+^n,
\hspace{5mm}
 \pa_{\nu}{\sf E} =0
\hspace{3mm}\text{on } \pa\R_+^n.
\end{equation}
Since the second eigenvalue of \eqref{eigenvaluehat-eq} is one,
it holds that
$\|\nabla w_{s-}\|_{\rho}\geq \|w_{s-}\|_{\rho}$
Hence
we take $\epsilon>0$ small enough,
then
there exists $\nu_0\in(0,\nu_1)$, $s_2^*>s_1^*$ and $\tau_2>\tau_1$ such that
for $s>s_2^*$
\[
 \pa_{\tau}\|w_{s-}\|_{\rho}^2
\leq
 -\|\nabla w_{s-}\|_{\rho}^2
-
 2(1+\nu_0)m\|w_{s-}\|_{\rho}^2
+
 ce^{-2(1+\nu_1)m(\tau-\tau_1)}
\hspace{5mm}\text{for }\tau>\tau_2.
\]
Multiplying by $e^{2(1+\nu_0)m\tau}$ and integrating both sides over
$(\tau_2,\tau)$,
we obtain the conclusion.
\end{proof}
%%%%%%%%%%%%%%%%%%%%%%%%%%%%%%%%%%%%%%%%%%%%%%%%%%%%%%%%%%%

%%%%%%%%%%%%%%%%%%%        Lemma        %%%%%%%%%%%%%%%%%%%
\begin{lem}\label{KY-lem}
There exists $c,s_3^*,\tau_3>0$ such that for $s\geq s_3^*$
\[
 \|w_{s0}(\tau)\|_{\rho} \leq ce^{-m\tau}
\hspace{5mm}\mathrm{for}\ \tau\geq\tau_3.
\]
\end{lem}
%%%%%%%%%%%%%%%%%%%%%%%%%%%%%%%%%%%%%%%%%%%%%%%%%%%%%%%%%%%

%%%%%%%%%%%%%%%%%%%        Proof        %%%%%%%%%%%%%%%%%%%
\begin{proof}
From definition of $w_{s0}$,
we easily see that
\[
\begin{array}{lll}
\dis
 \frac{1}{2}\pa_{\tau}\|w_{s0}\|_{\rho}^2
\hspace{-2mm}&\leq&\hspace{-2mm} \dis
 -m\|w_{s0}\|_{\rho}^2
+
 c\|w_{s0}\|_{\rho}\int_{\pa\R_+^n}w_s^q\rho\h dy'
\\ \hspace{-2mm}&\leq&\hspace{-2mm} \dis
 -m\|w_{s0}\|_{\rho}^2
+
 c\|w_{s0}\|_{\rho}\int_{\pa\R_+^n}
 \left( |w_{s0}|^q+|w_{s-}|^q \right)\rho\h dy'.
\end{array}
\]
Since $q>1$,
we fix $r>2$ such that $r'q>2$,
where $r'$ is defined by $1=1/r+1/r'$.
Then
by the H${\rm \ddot{o}}$lder inequality and a boundedness of $w_{s-}(y,\tau)$,
we obtain
\[
\begin{array}{lll}
\dis
 \|w_{s0}\|_{\rho}\int_{\pa\R_+^n}|w_{s-}|^q\rho\h dy'
\hspace{-2mm}&\leq&\hspace{-2mm} \dis
 c\|w_{s0}\|_{\rho}^r +\int_{\pa\R_+^n}|w_{s-}|^{r'q}\rho\h dy'
\\[4mm]
\hspace{-2mm}&\leq&\hspace{-2mm} \dis
 c\|w_{s0}\|_{\rho}^r +c\int_{\pa\R_+^n}|w_{s-}|^2\rho\h dy'
\\[4mm]
\hspace{-2mm}&\leq&\hspace{-2mm} \dis
 c\|w_{s0}\|_{\rho}^r +c\|w_{s-}\|_{H_{\rho}^1(\R_+^n)}^2.
\end{array}
\]
From Lemma \ref{55-lem},
there exists $s_1,\tau_1,\nu_0>0$ such that
for $s>s_1$ and $\tau>\tau_1$
\[
 \frac{1}{2}\pa_{\tau}\|w_{s0}\|_{\rho}^2
\leq
 -m\|w_{s0}\|_{\rho}^2
+
 ce^{-2(1+\nu_0)m\tau}
+
 c\|w_{s-}\|_{H_{\rho}^1(\R_+^n)}^2.
\]
Thus multiplying by $e^{2m\tau}$ and integrating over $(\tau_1,\tau)$,
from Lemma \ref{56-lem},
we obtain the conclusion.
\end{proof}
%%%%%%%%%%%%%%%%%%%%%%%%%%%%%%%%%%%%%%%%%%%%%%%%%%%%%%%%%%%

%%%%%%%%%%%%%%%%%%%        Proof        %%%%%%%%%%%%%%%%%%%
\begin{proof}
[\bf Proof of Proposition {\rm\ref{51-pro}} $($upper bound$)$]
We apply Lemma \ref{53-lem} with $R=1$ to obtain
\[
 w_s(0,\tau)
\leq
 c\sup_{y\in B_1,\tau'\in(\tau-1,\tau)}w_s(y,\tau)
\leq
 c\sup_{\tau'\in(\tau-4,\tau)}\|w_s(\tau')\|_{\rho}.
\]
Therefore
from Lemma \ref{56-lem} and Lemma \ref{KY-lem},
there exists $s_0^*,\tau_0>0$ such that for $s>s_0^*$
\[
 w_s(0,\tau) \leq ce^{-m\tau}
\hspace{5mm}\text{for }\tau>\tau_0.
\]
By definition of $w_s(y,\tau)$,
we note that
$v_s(0,1-e^{-\tau})=e^{m\tau}w_s(0,\tau)$.
Thus
we conclude that for $s\geq s_0^*$
\[
 v_s(0,1-e^{-\tau}) \leq c
\hspace{5mm}\text{for }\tau\geq\tau_1,
\]
which completes the proof.
\end{proof}
%%%%%%%%%%%%%%%%%%%%%%%%%%%%%%%%%%%%%%%%%%%%%%%%%%%%%%%%%%%

%%%%%%%%%%%%%%%%%%%        Proof        %%%%%%%%%%%%%%%%%%%
\subsection{Proof of Proposition \ref{51-pro} (lower bound)}
%%%%%%%%%%%%%%%%%%%%%%%%%%%%%%%%%%%%%%%%%%%%%%%%%%%%%%%%%%%

%%%%%%%%%%%%%%%%%%%        Proof        %%%%%%%%%%%%%%%%%%%
\begin{proof}
[\bf Proof of Proposition {\rm\ref{51-pro}} $($lower bound$)$]
The proof of lower bound is much easier than that of upper bound.
From \eqref{Kenchan1-eq} and $|\nabla w_s(y,\tau)|\leq c$,
there exist a nonnegative smooth function $w_*(y)\not\equiv0$ and $s_1^*\gg1$
such that for $s>s_1^*$
\[
 w_s(y,0) \geq w_*(y)
\hspace{5mm}\text{for } y\in\R_+^n.
\]
Let $\tilde{w}(y,\tau)$ be a solution of
\[
\begin{cases}
\dis
 \pa_{\tau}\tilde{w}
=
 \Delta\tilde{w} - \frac{y}{2}\cdot\nabla\tilde{w} - m\tilde{w}
& \text{in }\in\R_+\times(0,\infty),
\\ \dis
 \pa_{\nu}\tilde{w} = 0
& \text{on }\in\pa\R_+^n\times(0,\infty),
\\ \dis
 \tilde{w}(y,0) = w_*(y)
& \text{in }\R_+^n.
\end{cases}
\]
Then
a comparison argument shows that for $s>s_1^*$
\[
 w_s(y,\tau) \geq \tilde{w}(y,\tau).
\]
As in Section \ref{upper-sec},
we expand $\tilde{w}(y,s)$ by using eigenfunctions of
\eqref{eigenvaluehat-eq}.
Let ${\sf E}_0$ be a positive constant with $\|{\sf E}_0\|_{\rho}=1$.
Then
${\sf E}_0$ turns out to be the first eigenfunction of \eqref{eigenvaluehat-eq}.
We decompose $\tilde{w}(y,\tau)$ as follows.
\[
 \tilde{w}(\tau) = \tilde{w}_0(\tau) +\tilde{w}_-(\tau),
\]
where $\tilde{w}_0=(\tilde{w},{\sf E}_0)_\rho{\sf E}_0$.
Since the second eigenvalue of \eqref{eigenvaluehat-eq} is one,
we get
\[
 \tilde{w}_0(\tau) = (w_*,{\sf E}_0)_{\rho}e^{-m\tau}{\sf E}_0,
\hspace{7.5mm}
 \|\tilde{w}_-(\tau)\|_{\rho} \leq \|w_*\|_{\rho}e^{-(m+1)\tau}
\hspace{5mm}\text{for }\tau>0.
\]
Hence
a local parabolic regularity theory shows that
\[
 \sup_{y\in B_1}|\tilde{w}_{s-}(y,\tau)| \leq ce^{-(m+1)\tau}
\hspace{5mm}\text{for }\tau>0.
\]
As a consequence,
since $a_0:=(w_*,{\sf E}_0)_{\rho}>0$,
we obtain
\[
 \tilde{w}(y,\tau) = a_0\left( 1+O(e^{-\tau}) \right)e^{-m\tau}{\sf E}_0
\hspace{5mm}\text{for }y\in B_1.
\]
Thus
we conclude that for $s>s_1^*$
\[
 w_s(0,\tau) \geq \tilde{w}(0,\tau) = a_*\left( 1+O(e^{-\tau}) \right)e^{-m\tau}{\sf E}_0,
\]
which completes the proof.
\end{proof}
%%%%%%%%%%%%%%%%%%%%%%%%%%%%%%%%%%%%%%%%%%%%%%%%%%%%%%%%%%%%%%

%%%%%%%%%%%%%%%%%%%%%%%%%%%%%%%%%%%%%%%%%%%%%%%%%%%%%%%%%%%%%%
\appendix
\section{Appendix}
%%%%%%%%%%%%%%%%%%%%%%%%%%%%%%%%%%%%%%%%%%%%%%%%%%%%%%%%%%%%%%

%%%%%%%%%%%%%%%%%%%%%%%%%%%%%%%%%%%%%%%%%%%%%%%%%%%%%%%%%%%%%%
\subsection{Compact embedding inequality}
%%%%%%%%%%%%%%%%%%%%%%%%%%%%%%%%%%%%%%%%%%%%%%%%%%%%%%%%%%%%%%

Here
we provide
the embedding inequality on a weighted Sobolev space.

%%%%%%%%%%%%%%%%%%%%%%%%%%%%%%%%%%%%%%%%%%%%%%%%%%%%%%%%%%%%%%
\begin{lem}\label{Compact-lem}
It holds that for $u\in H_{\rho}^1(\R_+^n)$
\[
 \int_{\R_+^n}|y_i|^2u^2\rho\hspace{0.5mm}dy
\leq
 16\|\pa_iu\|_{\rho}^2+4\|u\|_{\rho}^2
\hspace{5mm}(i=1,\cdots,n).
\]
In particular,
it holds that
for $u\in H_{\rho}^1(\R_+^n)$
\[
 \int_{\R_+^n}|y|^2u^2\rho\hspace{0.5mm}dy
\leq
 16\|\nabla u\|_{\rho}^2+4n\|u\|_{\rho}^2.
\]
\end{lem}
%%%%%%%%%%%%%%%%%%%%%%%%%%%%%%%%%%%%%%%%%%%%%%%%%%%%%%%%%%%%%%

%%%%%%%%%%%%%%%%%%%%%%%%%%%%%%%%%%%%%%%%%%%%%%%%%%%%%%%%%%%%%%
\begin{proof}
The proof is based on that of Lemma 2.1 in \cite{Naito-S}
(see also Lemma 2.1 in \cite{Kavian} p.\hspace{1mm}430).
Since $C_c^{\infty}(\overline{\R_+^n})$ is dense in $H^1_{\rho}(\R_+^n)$,
we assume that $u\in C_c^{\infty}(\overline{\R_+^n})$.
We set
\[
 v(y)=u(y)e^{-|y|^2/8}.
\]
A direct computations shows that
\[
\begin{array}{lll}
\dis
 |\pa_iv|^2
\hspace{-2mm}&=&\hspace{-2mm} \dis
 |\pa_iu|^2e^{-|y|^2/4}
+
 \frac{y_i^2}{16}u^2e^{-|y|^2/4}
-
 \frac{1}{2}y_i(\pa_iu)ue^{-|y|^2/4}
\\[3mm]
\hspace{-2mm}&=&\hspace{-2mm} \dis
 |\pa_iu|^2e^{-|y|^2/4}
+
 \frac{y_i^2}{16}u^2e^{-|y|^2/4}
-
 \frac{1}{4}y_i(\pa_iu^2)e^{-|y|^2/4}.
\end{array}
\]
Then integrating by parts,
we see that
\[
\begin{array}{lll}
\dis
 \int_{\R_+^n}y_i(\pa_iu^2)e^{-|y|^2/4}dy
\hspace{-2mm}&=&\hspace{-2mm} \dis
 -\int_{\R_+^n}\pa_i\left( y_ie^{-|y|^2/4} \right)u^2dy
\\[4mm]
\hspace{-2mm}&=&\hspace{-2mm} \dis
 -\int_{\R_+^n}u^2e^{-|y|^2/4}dy
+
 \frac{1}{2}
 \int_{\R_+^n}|y_i|^2u^2e^{-|y|^2/4}dy.
\end{array}
\]
Hence it holds that
\[
 \|\pa_iv\|_{L^2(\R_+^n)}^2
=
 \|\pa_iu\|_{\rho}^2
+
 \frac{1}{4}\|u\|_{\rho}^2
-
 \frac{1}{16}
 \int_{\R_+^n}
 y_i^2u^2e^{-|y|^2/4}dy,
\]
which completes the proof.
\end{proof}
%%%%%%%%%%%%%%%%%%%%%%%%%%%%%%%%%%%%%%%%%%%%%%%%%%%%%%%%%%%%%%

From this inequality,
we obtain a compact embedding
from $H_{\rho}^1(\R_+^n)$ to $L_{\rho}^2(\R_+)$.

%%%%%%%%%%%%%%%%%%%%%%%%%%%%%%%%%%%%%%%%%%%%%%%%%%%%%%%%%%%%%%
\begin{lem}
The embedding from $H_{\rho}^1(\R_+^n)$ to $L_{\rho}^2(\R_+^n)$ is compact.
\end{lem}
%%%%%%%%%%%%%%%%%%%%%%%%%%%%%%%%%%%%%%%%%%%%%%%%%%%%%%%%%%%%%%

%%%%%%%%%%%%%%%%%%%%%%%%%%%%%%%%%%%%%%%%%%%%%%%%%%%%%%%%%%%%%%
\begin{proof}
Let $\{u_k\}_{k\in\N}$ be a bounded sequence in $H_{\rho}^1(\R_+^n)$.
Then
there exists $u\in H_{\rho}^1(\R_+^n)$ and a subsequence $\{u_k\}_{k\in\N}$
which is denoted by the same symbol such that
$u_k\rightharpoonup u$ weakly in $H_{\rho}^1(\R_+^n)$. 
Then
from Lemma \ref{Compact-lem},
we verify that
\begin{eqnarray*}
 \int_{|y|>R}|u_k(y)-u(y)|^2\rho(y)dy
\hspace{-2mm}&\leq&\hspace{-2mm}
 R^{-2}
 \int_{|y|>R}|y|^2|u_k(y)-u(y)|^2\rho(y)dy
\\
\hspace{-2mm}&\leq&\hspace{-2mm}
 cR^{-2}\|u_k-u\|_{H_{\rho}^1(\R_+^n)}^2.
\end{eqnarray*}
Hence
for any $\epsilon>0$ there exists $R_0>0$ such that
\[
 \int_{|y|>R_0}|u_k(x)-u(x)|^2\rho(y)dy \leq \epsilon/2.
\]
Since the embedding from $H_1(B_{R_0})$ to $L^2(B_{R_0})$ is compact,
there exists $k_0\in\N$ such that for $k\geq k_0$
\[
 \int_{|y|<R_0}|u_k(x)-u(x)|^2\rho(y)dy \leq \epsilon/2.
\]
Combining these estimates,
we obtain for $k\geq k_0$
\[
 \|u_k-u\|_{\rho}^2 \leq \epsilon,
\]
which completes the proof.
\end{proof}
%%%%%%%%%%%%%%%%%%%%%%%%%%%%%%%%%%%%%%%%%%%%%%%%%%%%%%%%%%%%%%

%%%%%%%%%%%%%%%%%%%%%%%%%%%%%%%%%%%%%%%%%%%%%%%%%%%%%%%%%%%%%%
\subsection{Linear operator $A$}
%%%%%%%%%%%%%%%%%%%%%%%%%%%%%%%%%%%%%%%%%%%%%%%%%%%%%%%%%%%%%%

In this subsection,
we show the operator
\[
 A_0v = \left( \Delta-\frac{y}{2}\cdot\nabla \right)v
\]
with
$D(A_0) = \{v\in H_{\rho}^2(\R_+^n);\pa_{\nu}v=Kv \text{ on } \pa\R_+^n\}$
($K\in\R$ is a constant) is self-adjoint.
From Lemma \ref{31-lem} with $g(y')\equiv1$,
there exists $c>0$ such that for $v\in H_{\rho}^1(\R_+^n)$
\[
 \int_{\pa\R_+^n}v^2\rho\h dy'
\leq
 \epsilon\|\nabla v\|_{\rho}^2 + \frac{c}{\epsilon}\|v\|_{\rho}^2.
\]
Hence there exists $\lambda_0>0$ such that for $v\in H_{\rho}^1(\R_+^n)$
\[
 K\int_{\pa\R_+^n}v^2\rho\h dy'
\leq
 \frac{1}{2}\|\nabla v\|_{\rho}^2
+
 \frac{\lambda_0}{2}\|v\|_{\rho}^2.
\]
Here
we show that the operator $A_{\lambda_0}=A_0-\lambda_0$ with $D(A_{\lambda_0})=D(A_0)$
is self-adjoint on $L_\rho^2(\R_+^n)$.
Once this is proved,
it is clear that the operator $A_0$ with $D(A_0)$ is also self-adjoint on $L_\rho^2(\R_+^n)$. 
By definition of $\lambda_0$,
it is verified that
$A_{\lambda_0}$ is a symmetric operator and satisfies
\begin{equation}\label{monotone-eq}
 (-A_{\lambda_0}v,v)_{\rho} \geq 0,
\hspace{5mm} v\in D(A_0).
\end{equation}
Hence
it is sufficient to show that 
for any $f\in L_{\rho}^2(\R_+^n)$
there exist $v\in D(A_0)$ such that $-A_{\lambda_0}v=f$.
First
we assume that $f\in C_c^{\infty}(\overline{\R_+^n})$.
From \eqref{monotone-eq},
there exists a weak solution $v\in H_{\rho}^1(\R_+^n)$
such that
for $\psi\in H_{\rho}^1(\R_+^n)$
\[
 \int_{\R_+^n}
 \left( \nabla v\cdot\nabla\psi+\lambda_0v\psi \right)
 \rho\hspace{0.5mm}dy
-
 K\int_{\pa\R_+^n}v\psi\rho\h dy'
=
 \int_{\R_+^n}f\psi\rho\hspace{0.5mm}dy.
\]
By using $\psi=v$ as a test function,
we obtain
\begin{equation}
\label{Bodinitial-eq}
 \|v\|_{H_{\rho}^1(\R_+^n)} \leq c\|f\|_{\rho}.
\end{equation}
In particular,
since $f\in C_c^{\infty}(\overline{\R_+^n})$ is a smooth function,
a standard elliptic regularity theory shows that
$v\in C^{\infty}(\overline{\R_+^n})$.
Let $\eta_k(r)\in C_c^{\infty}(\overline{\R_+})$ be a cut off function
such that
$\eta_k(r)=1$ if $r\in(0,k)$,
$\eta_k(r)=0$ if $r\in(2k,\infty)$.
By using
$\psi_1=|y'|^2v\eta_k(|y|)^2$ and $\psi_2=y_n^2v\eta_k(|y|)^2$
as test functions respectively
and
from Lemma \ref{31-lem} and Lemma \ref{Compact-lem},
we obtain
\begin{equation}\label{Bod-eq}
\begin{array}{c}
\dis
 \int_{\R_+^n}
 |y'|^2|\nabla v|^2
 \eta_k^2\rho\hspace{0.5mm}dy
\leq
 2\int_{\R_+^n}|y'|^2|fv|\eta_k^2\rho\hspace{0.5mm}dy
+
 c\|v\|_{H_{\rho}^1(\R_+^n)}^2,
\\[4mm] \dis
 \int_{\R_+^n}
 y_n^2|\nabla v|^2
 \eta_k^2\rho\hspace{0.5mm}dy
\leq
 2\int_{\R_+^n}y_n^2|fv|\eta_k^2\rho\hspace{0.5mm}dy
+
 c\|v\|_{H_{\rho}^1(\R_+^n)}^2.
\end{array}
\end{equation}
Next
we use $\psi=\nabla'(\eta_k(|y|)^2\nabla'v)$ as a test function.
Here we note that
\[
 \nabla\{\nabla'(\eta_k^2\nabla'v)\}
=
 \nabla'\{\nabla(\eta_k^2\nabla'v)\}
=
 \nabla'(\eta_k^2\nabla\nabla'v)
+
 \nabla'\{(\nabla\eta_k^2)\nabla'v\}.
\]
Hence integrating by parts,
we get
\[
\begin{array}{l}
\dis
 \int_{\R_+^n}\nabla v\cdot\nabla\{\nabla'(\eta_k^2\nabla'v)\}
 \rho\hspace{0.5mm}dy
=
 -\sum_{i=1}^n\sum_{j=1}^{n-1}\int_{\R_+^n}
 (\pa_i\pa_jv)^2\eta_k^2\rho\hspace{0.5mm}dy
\\ \dis \hspace{25mm}
+
 \frac{1}{2}\sum_{j=1}^{n-1}\int_{\R_+^n}
 y_j(\nabla v\cdot\nabla\pa_jv)
 \eta_k^2\rho\hspace{0.5mm}dy
+
 \int_{\R_+^n}
 \nabla v\cdot\nabla'(\nabla\eta_k^2\nabla'v)\rho\hspace{0.5mm}dy.
\end{array}
\]
Hence
it follows that
\begin{equation}
\label{Bodol-eq}
\begin{array}{l}
\dis
 -\int_{\R_+^n}\nabla v\cdot\nabla\{\nabla'(\eta_k^2\nabla'v)\}
 \rho\hspace{0.5mm}dy
\geq
 \frac{1}{2}\sum_{j=1}^{n-1}\int_{\R_+^n}
 |\nabla\pa_jv|^2\eta_k^2\rho\hspace{0.5mm}dy
\\ \dis \hspace{65mm}
-
 c\int_{\R_+^n}
 |y'|^2|\nabla v|^2\eta_k^2\rho\hspace{0.5mm}dy
-
 c\|v\|_{H_{\rho}^1(\R_+^n)}^2.
\end{array}
\end{equation}
The boundary integral is calculated as follows:
\begin{eqnarray*}
 \left|
 \int_{\pa\R_+^n}v\nabla'(\eta_k^2\nabla'v)\rho\h dy'
 \right|
\hspace{-2mm}&=&\hspace{-2mm}
 \left|
 -\int_{\pa\R_+^n}|\nabla'v|^2\eta_k^2\rho\h dy'
 +
 \frac{1}{2}\int_{\pa\R_+^n}
 y'\cdot(\nabla'v)v\eta_k^2\rho\h dy'
 \right|
\\
\hspace{-2mm}&\leq&\hspace{-2mm}
 c\int_{\pa\R_+^n}
 \left( |\nabla'v|^2\eta_k^2+|y'|^2v^2\eta_k^2 \right)\rho\h dy'.
\end{eqnarray*}
We set $V_1=|\nabla'v|\eta_k$ and $V_2=v\eta_k$.
Then
from Lemma \ref{31-lem},
it is verified that
\[
\begin{array}{lll}
\dis
 \int_{\pa\R_+^n}
 \left( V_1^2+|y'|^2V_2^2 \right)\rho\h dy'
\leq
 c\int_{\R_+^n}
 \left(
 \left( V_1^2+V_1|\nabla V_1| \right)
 +
 |y'|^2\left( V_2^2+V_2|\nabla V_2| \right)
 \right)
 \rho\hspace{0.5mm}dy
\\ \dis \hspace{20mm}
\leq
 c\sum_{j=1}^{n-1}\int_{\R_+^n}
 |\nabla'v||\nabla\pa_jv|\eta_k^2\rho\hspace{0.5mm}dy
+
 c\int_{\R_+^n}
 |y'|^2|\nabla v|^2\eta_k^2\rho\hspace{0.5mm}dy
+
 c\|v\|_{H_{\rho}^1(\R_+^n)}^2.
\end{array}
\]
Therefore we find that
\begin{equation}\label{Bodon-eq}
\begin{array}{l}
\dis 
\left|
 \int_{\pa\R_+^n}v\nabla'(\eta_k^2\nabla'v)\rho\h dy'
 \right|
\leq
 c\sum_{j=1}^{n-1}\int_{\R_+^n}
 |\nabla'v||\nabla\pa_jv|\eta_k^2\rho\hspace{0.5mm}dy
\\ \dis \hspace{60mm}
+
 c\int_{\R_+^n}
 |y'|^2|\nabla v|^2\eta_k^2\rho\hspace{0.5mm}dy
+
 c\|v\|_{H_{\rho}^1(\R_+^n)}^2.
\end{array}
\end{equation}
From \eqref{Bod-eq}, \eqref{Bodol-eq} and \eqref{Bodon-eq},
we get
\[
 \sum_{j=1}^{n-1}\int_{\R_+^n}
 |\nabla\pa_jv|^2\eta_k^2\rho\hspace{0.5mm}dy
\leq
 c\|f\|_{\rho}
 \left(
 \int_{\R_+^n}
 |y'|^4v^2\eta_k^2\rho\hspace{0.5mm}dy
 \right)^{1/2}
+
 c\|v\|_{H_{\rho}^1(\R_+^n)}^2.
\]
Then
applying Lemma \ref{Compact-lem} two times,
we compute the first term on the right-hand side.
\begin{equation}\label{Bodp-eq}
\begin{array}{lll}
\dis
 \int_{\R_+^n}
 |y'|^2(|y'||v|\eta_k)^2\rho\hspace{0.5mm}dy
\hspace{-2mm}&\leq&\hspace{-2mm} \dis
 c\int_{\R_+^n}
 \left(
 |\nabla'(|y'||v|\eta_k)|^2+(|y'||v|\eta_k)^2
 \right)
 \rho\hspace{0.5mm}dy
\\
\hspace{-2mm}&\leq&\hspace{-2mm} \dis
 c\int_{\R_+^n}
 |y'|^2(|\nabla'v|\eta_k)^2\rho\hspace{0.5mm}dy
+
 c\|v\|_{H_{\rho}^1(\R_+^n)}^2
\\
\hspace{-2mm}&\leq&\hspace{-2mm} \dis
 c\sum_{i,j=1}^{n-1}\int_{\R_+^n}
 |\pa_i\pa_jv|^2\eta_k^2\rho\hspace{0.5mm}dy
+
 c\|v\|_{H_{\rho}^1(\R_+^n)}^2.
\end{array}
\end{equation}
Thus finally we obtain
\begin{equation}
\label{Bodbod-eq}
 \sum_{j=1}^{n-1}\int_{\R_+^n}
 |\nabla\pa_jv|^2\eta_k^2\rho\hspace{0.5mm}dy
\leq
 c\|f\|_{\rho}^2
+
 c\|v\|_{H_{\rho}^1(\R_+^n)}^2.
\end{equation}
Since $v$ is a solution of $-A_{\lambda_0}v=f$,
it follows that
\[
 \pa_n^2v = \frac{y_n}{2}\pa_nv + F,
\]
where $F$ is given by
\[
 F = -\Delta'v + \frac{y'}{2}\cdot\nabla'v - \lambda_0v + f.
\]
Multiplying by $(\pa_n^2v)\rho\eta_k^2$ and integrating over $\R_+^n$,
we obtain
\[
 \|(\pa_n^2v)\eta_k\|_{\rho}^2
\leq
 c\int_{\R_+^n}y_n^2|\nabla v|^2\eta_k^2\rho\hspace{0.5mm}dy
+
 c\|F\|_{\rho}^2.
\]
Then
it follows from \eqref{Bod-eq} that
\[
 \|(\pa_n^2v)\eta_k\|_{\rho}^2
\leq
 c\|f\|_{\rho}
 \left(
 \int_{\R_+^n}y_n^4v^2\eta_k^2\rho\hspace{0.5mm}dy
 \right)^{1/2}
+
 c\|v\|_{H_{\rho}^1(\R_+^n)}^2
+
 c\|F\|_{\rho}^2.
\]
By the same calculation as \eqref{Bodp-eq},
we see that
\[
 \int_{\R_+^n}y_n^4v^2\eta_k^2\rho\hspace{0.5mm}dy
\leq
 c\|(\pa_n^2v)\eta_k\|_{\rho}^2
+
 c\|v\|_{H_{\rho}^1(\R_+^n)}^2,
\]
which implies
\begin{equation}\label{Bodfinal-eq}
 \|(\pa_n^2v)\eta_k\|_{\rho}^2
\leq
 c\|f\|_{\rho}^2
+
 c\|v\|_{H_{\rho}^1(\R_+^n)}^2
+
 c\|F\|_{\rho}^2.
\end{equation}
Thus
combining \eqref{Bodinitial-eq}, \eqref{Bodbod-eq} and \eqref{Bodfinal-eq}
and
taking $k\to\infty$,
we conclude that
\[
 \|v\|_{H_{\rho}^2(\R_+^n)}
\leq
 c\|f\|_{\rho}.
\]
Since
$C_c^{\infty}(\overline{\R_+^n})$ is dense in $L_{\rho}^2(\R_+^n)$,
by a density argument,
we complete the proof.

%%%%%%%%%%%%%%%%%%%%%%%%%%%%%%%%%%%%%%%%%%%%%%%%%%%%%%%%%%%
\section*{Acknowledgement}
The author would like to express his gratitude to Professor
Mitsuharu \^Otani for his valuable comments, suggestions and his encouragements.
%%%%%%%%%%%%%%%%%%%%%%%%%%%%%%%%%%%%%%%%%%%%%%%%%%%%%%%%%%%

%%%%%%%%%%%%%%%%%%%%%%%%%%%%%%%%%%%%%%%%%%%%%%%%%%%%%%%%%%%%%%%%%%%

%%%%%%%%%%%%%%%%%%%%%%%%%%%%%%%%%%
%%
%%
          \end{document}